\documentclass[11pt]{amsart}
\usepackage[margin=1.35in]{geometry}
\usepackage{amscd,amsmath,amsxtra,amsthm,amssymb,stmaryrd,xr,mathrsfs,mathtools,enumerate,commath, comment}
\usepackage{thmtools} % I think I could remove this package in the final version.  I need it for vscode on my new mac
\usepackage{stmaryrd}
\usepackage{xcolor}
\usepackage{commath}
\usepackage{comment}
\usepackage{tikz-cd}
\usepackage{float}
\usepackage{longtable} % for 'longtable' environment
\usepackage{pdflscape} % for 'landscape' environment
\usepackage{booktabs}
\usepackage{hyperref}
\usepackage[capitalize,nameinlink,noabbrev]{cleveref}
\usepackage[normalem]{ulem}

\usepackage{multirow}

\newcommand{\DV}[1]{\textcolor{blue}{#1}}
\newcommand{\RG}[1]{\textcolor{red}{#1}}

\usepackage[utf8]{inputenc}
\definecolor{vegasgold}{rgb}{0.77, 0.7, 0.35}
\definecolor{darkgoldenrod}{rgb}{0.72, 0.53, 0.04}
\definecolor{gold(metallic)}{rgb}{0.83, 0.69, 0.22}
\hypersetup{
 colorlinks=true,
 linkcolor=darkgoldenrod,
 filecolor=brown,      
 urlcolor=gold(metallic),
 citecolor=darkgoldenrod,
 pdftitle={Iwasawa theory for branched \texorpdfstring{$\mathbb{Z}_{p}$}{}-towers of finite graphs and Ihara zeta and \texorpdfstring{$L$}{}-functions},
 }

\usepackage[all,cmtip]{xy}

\usetikzlibrary{shapes.geometric}
\usetikzlibrary{decorations.markings}
\usetikzlibrary{shapes.geometric}
\tikzset{every loop/.style={min distance=10mm,looseness=10}}

\DeclareFontFamily{U}{wncy}{}
\DeclareFontShape{U}{wncy}{m}{n}{<->wncyr10}{}
\DeclareSymbolFont{mcy}{U}{wncy}{m}{n}
\DeclareMathSymbol{\Sh}{\mathord}{mcy}{"58}
\usepackage[T2A,T1]{fontenc}
\usepackage[OT2,T1]{fontenc}

\newtheorem{theorem}{Theorem}[section]
\newtheorem{lemma}[theorem]{Lemma}

\newtheorem{proposition}[theorem]{Proposition}
\newtheorem{corollary}[theorem]{Corollary}
\newtheorem{definition}[theorem]{Definition}

\numberwithin{equation}{section}

\theoremstyle{remark}
\newtheorem{remark}[theorem]{Remark}

\theoremstyle{definition}
\newtheorem{example}[theorem]{Example}

\font\manual=manfnt

\newcommand\xqed[1]{%
  \leavevmode\unskip\penalty9999 \hbox{}\nobreak\hfill
  \quad\hbox{#1}}
\newcommand\demo{\xqed{\manual\char'170}}

\begin{document}
\title[Iwasawa theory for graphs and Ihara zeta and $L$-functions]{Iwasawa theory for branched $\mathbb{Z}_{p}$-towers of finite graphs and Ihara zeta and $L$-functions}

\author[R.~Gambheera]{Rusiru Gambheera}
\address{Rusiru Gambheera\newline Department of Mathematics, University of California Santa Barbara, CA 93106-3080, USA}
\email{rusiru@ucsb.edu}

\author[D.~Valli\`{e}res]{Daniel Valli\`{e}res}
\address{Daniel Valli\`{e}res\newline Mathematics and Statistics Department, California State University, Chico, CA 95929, USA}
\email{dvallieres@csuchico.edu}

\begin{abstract}
We revisit the theory of Ihara $L$-functions in the context initially studied by Bass and Hashimoto and more recently by Zakharov.  In particular, we study if the Artin formalism is satisfied by these $L$-functions.  As an application, we give a proof of the analogue of Iwasawa's asymptotic class number formula for the $p$-part of the number of spanning trees in branched $\mathbb{Z}_{p}$-towers of finite connected graphs using Ihara zeta and $L$-functions.  Moreover, we relate a generator for the characteristic ideal of the finitely generated torsion Iwasawa module that governs the growth of the $p$-part of the number of spanning trees in such towers with Ihara $L$-functions.
\end{abstract}

\subjclass[2020]{Primary: 11M41; Secondary: 11R23, 05C25, 05C31} 
\date{\today} 
\keywords{Ihara zeta and $L$-functions, Graph theory, Spanning trees, Iwasawa theory}

\maketitle

\tableofcontents

\section{Introduction} \label{Introduction}
Throughout this paper, we use Serre's formalism to describe graphs as detailed in \cite{Serre:1977}, and we let $p$ denote a fixed rational prime.  In \cite{Gambheera/Vallieres:2024}, we initiated the study of Iwasawa theory for \emph{branched} $\mathbb{Z}_{p}$-towers of finite connected graphs arising from voltage assignments as explained in \cite[\S 4.3]{Gambheera/Vallieres:2024}.  If 
\begin{equation*}
X = X_{0} \leftarrow X_{1} \leftarrow X_{2} \leftarrow \ldots \leftarrow X_{n} \leftarrow \ldots 
\end{equation*}
is such a $\mathbb{Z}_{p}$-tower, \cite[Theorem 5.6]{Gambheera/Vallieres:2024} shows that there exist $\mu,\lambda,n_{0} \in \mathbb{Z}_{\ge 0}$ and $\nu \in \mathbb{Z}$ such that
\begin{equation} \label{iwasawa}
{\rm ord}_{p}(\kappa(X_{n})) = \mu p^{n} + \lambda n + \nu, 
\end{equation}
when $n \ge n_{0}$, where $\kappa(X_{n})$ denotes the number of spanning trees of $X_{n}$, and ${\rm ord}_{p}$ denotes the usual $p$-adic valuation on the field of rational numbers.  This result can be seen as being a graph theoretical analogue of Iwasawa's asymptotic class number formula in algebraic number theory (see \cite{Iwasawa:1959, Iwasawa:1973}).  Our approach in \cite{Gambheera/Vallieres:2024} was module theoretical.  To every finite connected graph $X$ is associated a finite abelian group ${\rm Pic}^{0}(X)$, called the Picard group of degree zero of $X$, whose cardinality is the number of spanning trees of $X$.  We obtained the result (\ref{iwasawa}) above by applying the usual structure theorem for finitely generated Iwasawa modules to the Iwasawa module, denoted by ${\rm Pic}_{\Lambda}^{0}$ in \cite[\S 4.3]{Gambheera/Vallieres:2024}, obtained by taking the projective limit of the Sylow $p$-subgroups of the finite abelian groups ${\rm Pic}^{0}(X_{n})$.  In fact, the Iwasawa invariants $\mu$ and $\lambda$ appearing in (\ref{iwasawa}) above are precisely the Iwasawa invariants of the Iwasawa module ${\rm Pic}_{\Lambda}^{0}$.  We built upon previous works of Gonet in \cite{Gonet:2021a, Gonet:2022}, of Kleine and M\"{u}ller in \cite{Kleine/Muller:2023}, and of Kataoka in \cite{Kataoka:2024} all done in the \emph{unramified} situation to obtain the result above in the \emph{branched} situation.

On the other hand, in \cite{Vallieres:2021,mcgownvallieresII,mcgownvallieresIII}, the authors obtained the result (\ref{iwasawa}) above in the unramified situation by using the Ihara zeta and $L$-functions and their special value at $u=1$.  The goal of this paper is to adapt those ideas to the branched situation to obtain another proof of (\ref{iwasawa}) that makes use of Ihara zeta and $L$-functions.  This approach is more in line with Iwasawa's proof of his asymptotic class number formula for the minus part of the class number in cyclotomic $\mathbb{Z}_{p}$-extensions of some cyclotomic number fields using his construction of $p$-adic $L$-functions and the analytic class number formula as detailed in \cite[\S 7]{Iwasawa:1972}.  If $X$ is a finite graph, then its Ihara zeta function is denoted by $Z_{X}(u)$ and it is known to be the reciprocal of a polynomial with integer coefficients.  The theory of Ihara zeta and $L$-functions for branched covers of finite graphs is more subtle than for unramified covers.  In the unramified situation, there are clear notions of Galois covers and Galois groups as groups of deck transformations.  Moreover, there is a Galois correspondence and a well-established theory of Ihara $L$-functions that satisfy the usual Artin formalism (see for instance the series of papers \cite{Stark/Terras:1996,Stark/Terras:2000,Stark/Terras:2007} by Stark and Terras).  Any unramified Galois cover of finite graphs can be obtained by starting with a finite connected graph $Y$ on which a finite group $G$ acts without inversion and freely on \emph{vertices}.  Since the action is without inversion, the quotient graph $Y_{G} = G \backslash Y$ makes sense, and since the action of $G$ is free on vertices, the natural morphism of graphs $Y \rightarrow Y_{G}$ is an unramified Galois cover with group of deck transformations isomorphic to $G$.  It is therefore natural to consider more general covers of finite graphs by relaxing the condition that $G$ acts freely on vertices.  When a group $G$ acts freely on vertices, it also necessarily acts freely on directed edges, so a first step would be to consider the situation where $G$ acts freely on \emph{directed edges}, but not necessarily on vertices.  This leads to the notion of \emph{branched cover}, and this is the type of covers we studied in \cite{Gambheera/Vallieres:2024} (branched covers are examples of harmonic morphisms as studied in \cite{Baker/Norine:2009} for instance, but not all harmonic morphisms are branched covers).  In this situation, one also has the data of a stabilizer group at each vertex.  Now, Corry pointed out for instance in \cite{Corry:2012} that the group of deck transformations of what ought to be a Galois cover in the branched setting is usually too big to get a neat Galois correspondence between subgroups of the group of deck transformations and (equivalence classes) of intermediate covers.  Moreover, in \cite{Malmskog/Manes:2010}, it was pointed out that the graph theoretical analogue of Dedekind's conjecture in algebraic number theory does not hold for branched covers of finite graphs.  Indeed, Malmskog and Manes gave examples of branched covers $Y \rightarrow X$ for which
$$Z_{X}(u)^{-1}  \nmid Z_{Y}(u)^{-1}.$$
As far as we know, in the branched and even more general situation in which a group $G$ acts without inversion but otherwise arbitrarily on a graph $Y$, Ihara $L$-functions have been introduced in \cite{Bass:1992} and \cite{Hashimoto:1992}.  

Bass's starting point is a group $\Gamma$ acting without inversion and discretely (meaning that the stabilizer at each vertex is finite) on a locally finite tree $T$ for which the quotient $\Gamma \backslash T$ is a finite graph.  He does not assume though that $\Gamma$ acts freely on directed edges.  Such a group $\Gamma$ is necessarily virtually free, and it seems to us that although this approach is quite general, it does not quite capture what we need for our purposes in this paper for a few different reasons.  For instance, the natural graph $X_{\infty}$ that resides at the top of the tower (\ref{iwasawa}), obtained by taking the projective limit of the finite graphs $X_{n}$, is a \emph{profinite} graph and is not locally finite when the tower is branched.  Moreover, the profinite group $\mathbb{Z}_{p}$ acts on $X_{\infty}$, but the stabilizers at the vertices of $X_{\infty}$ are not finite groups.  We point out in passing that the Bass-Serre theory in geometric group theory has been extended to the profinite situation by Ribes and his collaborators (see for instance \cite{Ribes:2017}) and this might well be the natural setup one should consider here. 

Hashimoto's starting point is one of a finite group acting on a finite graph just like we will consider in this paper with no assumption on the action whatsoever.  On the other hand, there is no study of the Artin formalism which is needed for our purposes.  Moreover, Ihara's determinant formula (also known as the three-term determinant formula) is also not considered in \cite{Hashimoto:1992}.

More recently, Zakharov explained in \cite{Zakharov:2021} how to adapt the theory of Ihara $L$-functions, as treated in \cite{Terras:2011} for instance, when a finite group $G$ acts freely on the directed edges of a finite graph which is not necessarily a tree (and allowing inversion so that the notion of legs has to be taken into account as well).  This approach is well suited for our purposes and can be applied at every level $n$ of a branched $\mathbb{Z}_{p}$-tower.  There is one caveat though.  The Ihara $L$-functions satisfy the usual additivity and induction properties of the Artin formalism, but \emph{not} the inflation property in general.  In retrospect, this failure of the inflation property is closely related to the fact that the analogue of Dedekind's conjecture alluded to above is not satisfied in general for branched covers of graphs.  Nonetheless, we show in this paper that one can study the Iwasawa theory of branched $\mathbb{Z}_{p}$-towers of finite connected graphs arising from voltage assignments using Ihara zeta and $L$-functions despite the failure of the inflation property. 

We could have used Zakharov's approach directly, since \cite[Proposition 4.14 and Corollary 4.15]{Zakharov:2021} give the needed product formula for the Ihara zeta function of the top graph $Y$ in terms of Ihara $L$-functions associated to characters of $G$.  But while thinking about this question, we wanted to understand how Bass's approach was related to Zakharov's approach.  Therefore, instead of directly using the definition of Ihara $L$-functions included in \cite{Zakharov:2021}, we present here our own approach to Ihara $L$-functions which is a hybrid between Bass's approach in \cite{Bass:1992} and Zakharov's approach in \cite{Zakharov:2021}.  In contrast to Bass's approach, we will restrict ourselves to the abelian situation in order to avoid non-commutative determinants.  Moreover, instead of having a (usually) infinite discrete group acting on a (usually) infinite tree, we will have a finite abelian group acting on a finite graph.  This is sufficient for the applications we have in mind to branched $\mathbb{Z}_{p}$-towers of finite connected graphs arising from voltage assignments.  In contrast to Zakharov's approach, our (finite) groups acting on (finite) graphs will be abelian, but we do not assume that $G$ acts freely on directed edges.  Moreover we do not study the situation where legs are allowed.  Allowing legs should be possible with minor modifications throughout.  Although allowing legs is advantageous and desirable for a few different reasons, we choose here to not include them, since our goal is to give a complementary approach to the results contained in \cite{Gambheera/Vallieres:2024} in which legs were not studied.  We take the shortest path by having as our starting point Ihara's determinant formula and reexpressing everything in an equivariant way.  We obtain in this way a satisfactory theory of Ihara $L$-functions for the situation of an arbitrary finite abelian group acting on a finite graph without inversion, which we can apply to the study of branched $\mathbb{Z}_{p}$-towers of finite connected graphs coming from voltage assignments.  Everything is concrete enough that calculations can readily be made, but our exposition is incomplete in the sense that the non-commutative setting, the connections with the fundamental group of graphs of groups and its abelianization, the analogue of the Artin reciprocity map, the analogue of Frobenius automorphisms, and the connection with the work of Hashimoto in \cite{Hashimoto:1992} are not studied here.  Thus, there is still room for further studies and clarifications.  

The paper is organized as follows.  In \cref{section:pre}, we gather together some results from commutative algebra, representation theory, and on the $p$-adic valuation that we shall need thoughout the paper.  \cref{section:gra} is dedicated to the notions in graph theory, such as Serre's formalism, the Euler characteristic, the laplacian operator, and the Ihara zeta function, that we will make use of throughout.  In \cref{equ}, we revisit the previously introduced notions in graph theory, but taking into account the action of a finite abelian group; the theory then becomes an equivariant one. Using the equivariant Ihara zeta function of \cref{section:theequiv}, we associate to any finite dimensional $\mathbb{C}$-linear representation $V$ of the underlying finite abelian group an Ihara $L$-function denoted by $L_{V}(u)$.  The Artin formalism satisfied by these functions is studied in \cref{af} and is viewed as a consequence of the Artin formalism studied previously in an equivariant setting in \cref{ind_inf}.  In \cref{br_section}, we apply the theory of Ihara $L$-functions detailed in \cref{equ,section:art} to reprove the analogue of Iwasawa's asymptotic class number formula for branched $\mathbb{Z}_{p}$-towers of finite connected graphs arising from voltage assignments.  This result is contained in \cref{ana_sec}.  In \cref{section:revi}, we relate a generator for the characteristic ideal of the finitely generated torsion Iwasawa module that governs the growth of the $p$-part of the number of spanning trees in such towers with Ihara $L$-functions.  At last, we revisit the three examples of \cite[\S 6]{Gambheera/Vallieres:2024} from this point of view in \cref{exa}.

\subsection*{Acknowledgments}
The second author would like to thank John Lind for several very useful and interesting discussions.  Moreover, the second author also acknowledges support from an AMS-Simons Research Enhancement Grant for PUI Faculty.  Furthermore, both authors would like to thank Dmitry Zakharov for providing comments and references pertaining to this work.

%\section{Preliminaries}
\section{Preliminaries} \label{section:pre}

\subsection{Rings and modules} \label{section:ring}
Throughout this paper, by a ring we shall always mean a unital ring.  If $R$ is a commutative ring, and $A$ is an (not necessarily commutative) $R$-algebra, then it will also always be assumed to be a unital $R$-algebra.  If $R$ and $S$ are rings, then by a morphism $\varphi:R \rightarrow S$ of rings, we mean a ring morphism that satisfies $\varphi(1_{R}) = 1_{S}$, and similarly for $R$-algebras.  If $A$ is an $R$-algebra, then by an $A$-module we always mean a left $A$-module.  When $R$ is commutative, if necessary we will view an $R$-module as an $(R,R)$-bimodule with the right multiplication by elements of $R$ given by $mr := rm$.  

For us, the commutative ring $R$ will often be the group algebra $\mathbb{C}[G]$, the polynomial ring $\mathbb{C}[G][u]$, or still the power series ring $\mathbb{C}[G]\llbracket u \rrbracket$, where $G$ is a finite abelian group.  We remind the reader about polynomials and power series below.  The $R$-algebra $A$ will often be an endomorphism ring ${\rm End}_{R}(M)$, where $M$ will typically be a finitely generated projective $R$-module.

%\subsection{Polynomials and power series}
\subsubsection{Polynomials and power series} \label{section:pol}
If $R$ is a commutative ring, we let $R[u]$ be the commutative ring consisting of polynomials with coefficients in $R$; it consists of expressions of the form
$$\sum_{i=0}^{\infty} r_{i}u^{i}, $$
where $r_{i}\in R$ for all indices $i \ge 0$ and $r_{i}=0$ for all but finitely indices $i$.  Polynomials are added and multiplied in the usual way.  

More generally, if $R$ is a commutative ring and $A$ is an (not necessarily commutative) $R$-algebra, we let $A[u]$ denote the $R[u]$-algebra of polynomials with coefficients in $A$.  If $\varphi:A \rightarrow B$ is a morphism of $R$-algebras, then it induces a morphism $\varphi:A[u] \rightarrow B[u]$ of $R[u]$-algebras by applying $\varphi$ to each coefficient of a polynomial in $A[u]$.  We use the same letter $\varphi$ to denote this morphism.

Similarly, if $M$ is an $R$-module for a commutative ring $R$, then we let $M[u]$ be the abelian group consisting of polynomials with coefficients in $M$; it is an $R[u]$-module, the action of $R[u]$ being given as usual by
\begin{equation} \label{module_st}
\left(\sum_{i=0}^{\infty}r_{i}u^{i} \right) \cdot \left( \sum_{j=0}^{\infty} m_{j}u^{j}\right) = \sum_{k=0}^{\infty}\left(\sum_{i+j = k} r_{i} m_{j} \right) u^{k}.
\end{equation}
If $\varphi:M \rightarrow N$ is a morphism of $R$-modules, then it induces yet another morphism $M[u] \rightarrow N[u]$ of $R[u]$-modules which is obtained by applying $\varphi$ to the coefficients of a polynomial in $M[u]$, and which we denote by the same symbol $\varphi$ as well.  In terms of tensor products, we have a natural isomorphism 
$$R[u] \otimes_{R}M \simeq  M[u] $$
of $R[u]$-modules, where the structure of $R[u]$-module on the tensor product is induced from the natural structure of $(R[u],R)$-bimodule on $R[u]$.  Clearly, if $M$ is a free $R$-module of finite rank, then $M[u]$ is also a free $R[u]$-module of the same rank over $R[u]$. 
\begin{theorem} \label{conv_iso_pol_vs}
Let $R$ be a commutative ring, and $M$ a finitely generated $R$-module.  The function 
$$\omega_{M}:{\rm End}_{R}(M)[u] \rightarrow {\rm End}_{R[u]}(M[u])$$ 
defined via $P \mapsto \omega_{M}(P)$, where
$$\omega_{M}\left( \sum_{i=0}^{\infty}f_{i}u^{i}\right)\left(\sum_{j=0}^{\infty}m_{j}u^{j}\right) = \sum_{k=0}^{\infty} \left(\sum_{i+j = k}f_{i}(m_{j})\right)u^{k} $$
is an isomorphism of $R[u]$-algebras.
\end{theorem}
\begin{proof}
Everything is routine, except for the surjectivity which is also the only place where the hypothesis that $M$ is a finitely generated $R$-module is used:  If $f \in {\rm End}_{R[u]}(M[u])$, and $m \in M$, then
$$f(m) = \sum_{i=0}^{\infty}f_{i}(m)u^{i} $$
for some $f_{i}(m) \in M$ that are all zero except for finitely many indices $i$.  Since $f$ is a morphism of $R[u]$-module, one readily checks that $f_{i} \in {\rm End}_{R}(M)$.  Moreover the hypothesis that $M$ is finitely generated over $R$ implies that $f_{i}=0$ for all but finitely many indices $i$; therefore,
$$\sum_{i=0}^{\infty}f_{i}u^{i} \in {\rm End}_{R}(M)[u]. $$
The result follows after noticing that $\omega_{M} \left( \sum_{i=0}^{\infty}f_{i}u^{i}\right) = f$.
\end{proof}

Back to a commutative ring $R$, we let $R\llbracket u \rrbracket$ denote the commutative ring of formal power series with coefficients in $R$; it consists of expressions of the form
$$\sum_{i=0}^{\infty} r_{i}u^{i}, $$
with no condition on the $r_{i} \in R$ and are added and multiplied in the usual way.  The commutative ring $R\llbracket u \rrbracket$ can be thought of as the completion of $R[u]$ with respect to the $u$-adic topology. 

More generally, if $R$ is a commutative ring and $A$ is an (not necessarily commutative) $R$-algebra, we let $A \llbracket u \rrbracket$ denote the $R\llbracket u \rrbracket$-algebra of formal power series with coefficients in $A$.  Just as for polynomials, if $\varphi:A \rightarrow B$ is a morphism of $R$-algebras, then it induces a morphism $\varphi:A \llbracket u \rrbracket \rightarrow B \llbracket u \rrbracket$ of $R\llbracket u \rrbracket$-algebras by applying $\varphi$ to each coefficient of a power series in $A \llbracket u \rrbracket$, which we also denote by the same symbol $\varphi$.

If $M$ is an $R$-module, then we let $M\llbracket u \rrbracket$ be the $R\llbracket u \rrbracket$-module consisting of formal power series with coefficients in $M$, the action of $R\llbracket u \rrbracket$ being given as above by (\ref{module_st}).  Again, if $\varphi:M \rightarrow N$ is a morphism of $R$-modules, then it induces yet another morphism $M\llbracket u \rrbracket \rightarrow N\llbracket u \rrbracket$ of $R\llbracket u \rrbracket$-modules which is obtained by applying $\varphi$ to the coefficients of a power series in $M\llbracket u \rrbracket$, and which we denote by the same symbol $\varphi$.  The $R\llbracket u \rrbracket$-module $M\llbracket u \rrbracket$ can be thought as the $u$-adic completion of $M[u]$, since one has a natural isomorphism
$$M\llbracket u \rrbracket \stackrel{\simeq}{\longrightarrow} \varprojlim_{n \ge 1} M[u]/u^{n}M[u]. $$  
In terms of tensor products, we have the following proposition.
\begin{proposition}\label{comp_base_change}
If $M$ is a finitely generated $R$-module and $R$ is noetherian, then we have an isomorphism
$$ R \llbracket u \rrbracket \otimes_{R[u]}M[u] \simeq M\llbracket u \rrbracket $$
of $R\llbracket u \rrbracket$-modules, where the structure of $R\llbracket u \rrbracket$-module on the tensor product is induced from the natural structure of $(R\llbracket u \rrbracket,R[u])$-bimodule on $R \llbracket u \rrbracket$.  
\end{proposition} 
\begin{proof}
See \cite[Proposition 10.13]{Atiyah/MacDonald:1969}.
\end{proof}
Just as for polynomials, if $M$ is a free $R$-module of finite rank, then $M\llbracket u \rrbracket$ is also a free $R\llbracket u \rrbracket$-module of the same rank over $R \llbracket u \rrbracket$.  

\begin{theorem} \label{conv_iso}
Let $R$ be a commutative ring, and $M$ an $R$-module, then the function 
$$\omega_{M}:{\rm End}_{R}(M)\llbracket u \rrbracket \rightarrow {\rm End}_{R\llbracket u \rrbracket}(M \llbracket u \rrbracket)$$ 
defined via $P \mapsto \omega_{M}(P)$, where
$$\omega_{M}\left( \sum_{i=0}^{\infty}f_{i}u^{i}\right)\left(\sum_{j=0}^{\infty}m_{j}u^{j}\right) = \sum_{k=0}^{\infty} \left(\sum_{i+j = k}f_{i}(m_{j})\right)u^{k} $$
is an isomorphism of $R\llbracket u \rrbracket$-algebras.
\end{theorem}
\begin{proof}
The proof is similar to the proof of \cref{conv_iso_pol_vs} and left to the reader.
\end{proof}

%\subsubsection{Exponential and logarithm}
\subsubsection{Exponential and logarithm} \label{section:exp}
If $A$ is assumed to be a (not necessarily commutative) $\mathbb{Q}$-algebra, then recall that we have the usual exponential power series given by
$$E = \sum_{n=0}^{\infty} \frac{u^{n}}{n!} \in 1 + u\mathbb{Q}\llbracket u \rrbracket. $$
We then get the exponential function ${\rm exp}_{A}: u A\llbracket u \rrbracket \rightarrow 1 + u A \llbracket u \rrbracket$ defined via $P \mapsto E(P)$.  If we let 
$$\lambda = \sum_{n=1}^{\infty}(-1)^{n-1} \frac{u^{n}}{n} \in u \mathbb{Q} \llbracket u \rrbracket, $$
then the logarithm function $\log_{A}:1 + u A \llbracket u \rrbracket \rightarrow u A\llbracket u \rrbracket$ is defined via $1 + Q \mapsto \log_{A}(1+Q) = \lambda(Q)$.
We list in the following proposition a few properties of $\exp_{A}$ and $\log_{A}$ that we shall need in this paper.
\begin{proposition} \label{basic_prop} 
The exponential and logarithm functions 
$$\exp_{A}:u A\llbracket u \rrbracket \rightarrow 1 + u A\llbracket u \rrbracket \text{ and } \log_{A}:1 + u A \llbracket u \rrbracket \rightarrow u A\llbracket u \rrbracket$$ 
satisfy the following properties:
\begin{enumerate}
\item They are inverses of one another and are homeomorphisms for the $u$-adic topology.
\item If $B$ is another $\mathbb{Q}$-algebra, and if $\varphi:A\llbracket u \rrbracket \rightarrow B\llbracket u \rrbracket$ is a morphism of $\mathbb{Q}\llbracket u \rrbracket$-algebras, then
$$\varphi \circ {\rm exp}_{A} = {\rm exp}_{B} \circ \varphi \text{ and } \varphi \circ {\rm log}_{A} = {\rm log}_{B} \circ \varphi.$$
\item If $P_{1}, P_{2} \in u A \llbracket u \rrbracket$ commute, one has
$${\rm exp}_{A}(P_{1} + P_{2}) = {\rm exp}_{A}(P_{1}) \cdot {\rm exp}_{A}(P_{2}). $$
\end{enumerate}
\end{proposition} 
\begin{proof}
See \cite[II, \S 6, no.1]{Bourbaki:1998}. 
\end{proof}
Note in particular that if $P \in 1 + uA\llbracket u \rrbracket$, and $a \in A$, then the expression $P^{a} $ makes sense and is defined as usual via the equality
\begin{equation} \label{exponent}
P^{a} = \exp_{A} \left( a \cdot \log_{A}(P) \right) \in 1 + uA \llbracket u \rrbracket. 
\end{equation}
If $B$ is another $\mathbb{Q}$-algebra and $\varphi:A \rightarrow B$ is a morphism of $\mathbb{Q}$-algebras, then it induces another $\mathbb{Q}$-algebra morphism $A\llbracket u \rrbracket \rightarrow B\llbracket u \rrbracket$ which we shall denote by the same symbol $\varphi$.  Recall that $\varphi$ is simply defined via
$$\sum_{i=0}^{\infty} a_{i}u^{i} \mapsto \varphi \left( \sum_{i=0}^{\infty} a_{i}u^{i}\right) = \sum_{i=0}^{\infty} \varphi(a_{i})u^{i}.$$
\begin{proposition} \label{exp_beh}
Let $A$ be a commutative $\mathbb{Q}$-algebra, and let $P \in 1 + uA\llbracket u \rrbracket$.
\begin{enumerate}
\item If $a_{1}, a_{2} \in A$, then 
$$P^{a_{1} + a_{2}}  = P^{a_{1}} \cdot P^{a_{2}}. $$
\item If $B$ is another commutative $\mathbb{Q}$-algebra, and if $\varphi:A \rightarrow B$ is a morphism of $\mathbb{Q}$-algebras, then for all $a \in A$, one has
$$\varphi \left(P^{a} \right) = \varphi(P)^{\varphi(a)}.$$
\end{enumerate}
\end{proposition}
\begin{proof}
This is a direct consequence of \cref{basic_prop}.
\end{proof}

%\subsection{Commutative algebra}
\subsection{Commutative algebra} \label{com_alg}
Let $R$ be a ring.  If $n$ is a positive integer, we let $M_{n}(R)$ denote the ring consisting of $n \times n$ matrices with entries in $R$.  If $\varphi:R \rightarrow S$ is a ring morphism, then it induces a ring morphism $M_{n}(R) \rightarrow M_{n}(S)$ by applying $\varphi$ to the entries of a matrix in $M_{n}(R)$.  We denote this latter ring morphism by the same symbol $\varphi$.  

Assume now that $R$ is a commutative ring.  The trace and determinant functions
$${\rm tr}_{R}:M_{n}(R) \rightarrow R \text{ and } {\rm det}_{R}:M_{n}(R) \rightarrow R $$
are defined as usual.  The trace is $R$-linear and the determinant is multiplicative.  The trace, determinant and exponential are related via the following identity.  For a proof, we refer the reader to \cite[Theorem 9.2]{Sambale:2023} in which an argument is presented when $R$ is a field of characteristic zero, but the reader can check that the exact same proof works more generally when $R$ is any commutative $\mathbb{Q}$-algebra.
\begin{theorem} \label{det_tr_exp}
Let $R$ be a commutative $\mathbb{Q}$-algebra.  Then, the following diagram
\begin{equation*}
\begin{tikzcd}
uM_{n}(R\llbracket u \rrbracket) \arrow["{\rm exp}",d] \arrow["{\rm tr}_{R\llbracket u \rrbracket}",r] & uR \llbracket u \rrbracket \arrow["{\rm exp}_{R}",d]\\
I+uM_{n}(R\llbracket u \rrbracket) \arrow["{\rm det}_{R\llbracket u \rrbracket}",r] & 1+ uR \llbracket u \rrbracket
\end{tikzcd}
\end{equation*}
commutes, where the left vertical arrow is the usual exponential of matrices, and $I$ is the $n\times n$ identity matrix.
\end{theorem} 

If $F$ is a free $R$-module of finite rank and $f \in {\rm End}_{R}(F)$, then the trace and determinant functions
\begin{equation} \label{tr_and_det}
{\rm tr}_{R}:{\rm End}_{R}(F) \rightarrow R \text{ and } {\rm det}_{R}:{\rm End}_{R}(F) \rightarrow R
\end{equation}
are defined as usual.  Those definitions are extended to finitely generated projective modules over $R$ and we now recall one way to do that (see for instance \cite{Goldman:1961}).  Let $M$ be a finitely generated projective $R$-module and let $f \in {\rm End}_{R}(M)$.  Since $M$ is projective over $R$, there exists a free $R$-module $F$ of finite rank and another $R$-module $N$ such that $F = M \oplus N$.  Consider now $f \oplus 0_{N}, f \oplus {\rm id}_{N} \in {\rm End}_{R}(F)$, where $0_{N}$ and ${\rm id}_{N}$ are the zero and identity map on $N$, respectively.  The trace and determinant of $f$ are defined to be
$${\rm tr}_{R}(f) = {\rm tr}_{R}(f \oplus 0_{N}) \text{ and } {\rm det}_{R}(f) = {\rm det}_{R}(f \oplus {\rm id}_{N}).$$
It can be shown that those are well defined, meaning that the choices of $F$ and $N$ are irrelevant, and also that both the trace and the determinant satisfy the usual base change property with respect to a ring morphism $R \rightarrow S$.  Moreover if $M_{1}$ and $M_{2}$ are two finitely generated projective $R$-modules and $f_{i} \in {\rm End}_{R}(M_{i})$ for $i=1,2$, then for $f_{1} \oplus f_{2} \in {\rm End}_{R}(M_{1} \oplus M_{2})$, one has
\begin{equation} \label{det_prop}
{\rm tr}_{R}(f_{1} \oplus f_{2}) = {\rm tr}_{R}(f_{1}) + {\rm tr}_{R}(f_{2}) \text{ and } {\rm det}_{R}(f_{1} \oplus f_{2}) = {\rm det}_{R}(f_{1}) \cdot {\rm det}_{R}(f_{2}).
\end{equation}
If $F$ is a free $R$-module of rank one, then the function 
\begin{equation}\label{simple_but_useful}
{\rm End}_{R}(F) \rightarrow R
\end{equation}
given by $f \mapsto r_{f}$, where $r_{f}$ is the unique element in $R$ satisfying $f(m) = r_{f} \cdot m$ for all $m \in F$ is an isomorphism of $R$-algebras, and it coincides with both ${\rm tr}_{R}$ and ${\rm det}_{R}$.

We will also make use of the (Hattori-Stallings) rank of a finitely generated projective $R$-module $M$ which is defined to be
$${\rm rank}_{R}(M) = {\rm tr}_{R}({\rm id}_{M}) \in R. $$
The rank is additive, meaning that if 
$$0 \rightarrow M_{1} \rightarrow M_{2} \rightarrow M_{3} \rightarrow 0 $$
is a short exact sequence of projective $R$-modules, then
$${\rm rank}_{R}(M_{2}) = {\rm rank}_{R}(M_{1}) + {\rm rank}_{R}(M_{3}), $$
which follows from (\ref{det_prop}).

At last, we will make use of the following result on determinants of block matrices, for which the reader can find a proof in \cite{Silvester:2000}.
\begin{theorem} \label{silvester}
Let $R$ and $S$ be commutative rings, and assume that we are given a ring morphism $\rho: S \rightarrow M_{n}(R)$.  Then the diagram
\begin{equation*}
\begin{tikzcd}
M_{m}(S) \arrow[,d] \arrow["{\rm det}_{S}",r] & S  \arrow["\rho",r] & M_{n}(R) \arrow["{\rm det}_{R}",d]\\
M_{m}(M_{n}(R))  \arrow["\simeq",r] & M_{mn}(R) \arrow["{\rm det}_{R}",r] & R
\end{tikzcd}
\end{equation*}
commutes, where the first vertical arrow is the one induced by $\rho$ as explained in the first paragraph of \cref{com_alg}.
\end{theorem}

%\subsection{Representation theory}
\subsection{Representation theory} \label{section:rep}
Let $G$ be a finite abelian group.  We remind the reader that the group algebra $\mathbb{C}[G]$ is semisimple, and thus in particular any $\mathbb{C}[G]$-module is projective.  We choose to work with $\mathbb{C}$, but later in \cref{ana_sec} we will use $\mathbb{C}_{p}$.  If the readers desire, they can pretend $\mathbb{C}$ is an arbitrary algebraically closed field of characteristic zero throughout.  In this paper, we use the module theoretical approach to representation theory of finite groups.  Thus we view a finite dimensional linear representation of $G$ over $\mathbb{C}$ as a finitely generated $\mathbb{C}[G]$-module which amounts to the same as a finite dimensional $\mathbb{C}$-linear representation of the $\mathbb{C}$-algebra $\mathbb{C}[G]$, namely a $\mathbb{C}$-algebra morphism
\begin{equation} \label{rep_pt_view}
\mathbb{C}[G] \rightarrow {\rm End}_{\mathbb{C}}(V), 
\end{equation}
for some finite dimensional $\mathbb{C}$-vector space $V$.  Moreover, we will refer to any such finitely generated $\mathbb{C}[G]$-module $V$ as an Artin representation of $G$.  If $V$ is an Artin representation of $G$, then we let $\psi_{V}:\mathbb{C}[G] \rightarrow \mathbb{C}$ denote its character which is obtained by composing (\ref{rep_pt_view}) with the trace map ${\rm tr}_{\mathbb{C}}: {\rm End}_{\mathbb{C}}(V) \rightarrow \mathbb{C}$ from (\ref{tr_and_det}).  Two Artin representations are isomorphic as $\mathbb{C}[G]$-modules if and only if they have the same characters.  

A character is called irreducible if it comes from an irreducible representation.  The collection of all irreducible characters will be denoted by ${\rm Ir}(G)$.  It is known to be a finite set, and since $G$ is assumed to be abelian, there is a one-to-one correspondence between ${\rm Ir}(G)$ and $\widehat{G}$, where $\widehat{G}$ denotes the collection of group morphisms from $G$ into $\mathbb{C}^{\times}$.  Recall that for an arbitrary Artin representation, $\psi_{V}$ is only a $\mathbb{C}$-linear morphism, and is not necessarily a $\mathbb{C}$-algebra morphism.  On the other hand, if $V$ is assumed to be an irreducible representation, then its character $\psi_{V}$ is a morphism of $\mathbb{C}$-algebras.  Given $\psi \in \widehat{G}$, we will use the same symbol $\psi$ to denote the corresponding morphism of $\mathbb{C}$-algebras $\psi:\mathbb{C}[G] \rightarrow \mathbb{C}$.  This should not cause any confusion, since the restriction of the latter to $G$ is the same as the former.  

If $\psi \in \widehat{G}$, we let 
$$e_{\psi} = \frac{1}{|G|} \sum_{\sigma \in G}\psi(\sigma) \cdot \sigma^{-1} \in \mathbb{C}[G] $$
be the usual primitive idempotent of $\mathbb{C}[G]$ associated with the character $\psi$.  These idempotents are orthogonal, meaning that $e_{\psi_{1}} \cdot e_{\psi_{2}} = 0$, when $\psi_{1} \neq \psi_{2}$, as a simple calculation using the usual orthogonality relation
\begin{equation}\label{ortho}
\frac{1}{|G|} \sum_{\sigma \in G} \psi_{1}(\sigma) \cdot \psi_{2}(\sigma^{-1})=
\begin{cases}
0, &\text{ if } \psi_{1} \neq \psi_{2}; \\
1, &\text{ otherwise},
\end{cases}
\end{equation}
shows.  We will also make use of the following (usually non-primitive) idempotents:  If $H \le G$, we let
\begin{equation} \label{idem}
e_{H} = \frac{1}{|H|}N_{H} \in \mathbb{C}[G], 
\end{equation}
where
$$N_{H} = \sum_{h \in H} h. $$  
The following lemma will be used often, and it is simple enough that we omit its proof.
\begin{lemma} \label{useful_dec}
Let $G$ be a finite abelian group and let $H \le G$.  Consider the natural ring morphism $\pi_{H}:\mathbb{C}[G] \rightarrow \mathbb{C}[G/H]$ and let $I(H)$ be its kernel so that we have a short exact sequence
\begin{equation} \label{ses_subgroup}
0 \rightarrow I(H) \rightarrow \mathbb{C}[G] \stackrel{\pi_{H}}{\longrightarrow} \mathbb{C}[G/H] \rightarrow 0 
\end{equation}
of $\mathbb{C}[G]$-modules.  The $\mathbb{C}$-linear function $\iota_{H}:\mathbb{C}[G/H] \rightarrow \mathbb{C}[G]$ defined on cosets via
$$\sigma H \mapsto \iota_{H}(\sigma H) = \sigma e_{H} $$
is a well-defined morphism of $\mathbb{C}[G]$-modules which gives a splitting of the short exact sequence (\ref{ses_subgroup}) so that one has $\mathbb{C}[G] = I(H) \oplus \iota_{H}(\mathbb{C}[G/H])$.  Moreover,
$$\iota_{H}(\mathbb{C}[G/H]) = \mathbb{C}[G] e_{H} \text{ and } I(H) = \mathbb{C}[G](1-e_{H}). $$
\end{lemma}

%\subsection{Permutation modules}
\subsection{Permutation modules} \label{permutation}
Let $T$ be a finite set on which a finite abelian group $G$ acts, and let $S = T_{G}$, where we write $T_{G}$ instead of $G \backslash T$ for the orbits.  For $t \in T$, we let $G_{t} = {\rm Stab}_{G}(t)$.  Since $G$ is abelian, $G_{t}$ depends only on $s = Gt \in S$, and thus we shall often write $G_{s}$ instead of $G_{t}$, where $s = Gt$.  Let  $\mathbb{C}T$ denote the $\mathbb{C}$-vector space whose basis is given by the elements in $T$.  In this section, we study the structure of $\mathbb{C}T$ as a $\mathbb{C}[G]$-module (it is a permutation module).  

We fix a labeling $S = \{s_{1},\ldots,s_{g} \}$ of $S$ and for each $i = 1,\ldots,g$, we let $t_{i}$ be a fixed element of $T$ such that $Gt_{i} = s_{i}$.  For $i=1,\ldots,g$, we write $G_{i}$ instead of $G_{s_{i}}$, and similarly we write $e_{i}$ instead of $e_{G_{i}}$ for the idempotent corresponding to the subgroup $G_{i}$ that was defined above in (\ref{idem}).  Since 
$$\mathbb{C}T = \bigoplus_{i=1}^{g}\mathbb{C}[G] t_{i}, $$
we have to understand the $\mathbb{C}[G]$-modules $\mathbb{C}[G] t_{i}$.  For that purpose, let $i \in \{1,\ldots,g \}$, and consider the natural surjective $\mathbb{C}[G]$-module morphism $\phi:\mathbb{C}[G] \rightarrow \mathbb{C}[G] t_{i}$ given by $\lambda \mapsto \phi(\lambda) = \lambda t_{i}$.  The map $\phi$ leads to a short exact sequence
\begin{equation} \label{ses_us}
0 \rightarrow {\rm Ann}_{\mathbb{C}[G]}(t_{i}) \rightarrow \mathbb{C}[G] \stackrel{\phi}{\rightarrow} \mathbb{C}[G] t_{i} \rightarrow 0   
\end{equation}
of $\mathbb{C}[G]$-modules.  In order to understand ${\rm Ann}_{\mathbb{C}[G]}(t_{i})$, we shall make use of the following lemma.
\begin{lemma} \label{useful_lemma}
With the notation as above, if $\psi \in \widehat{G}$ is a non-trivial character satisfying $G_{i} \not \subseteq {\rm ker}(\psi)$, then in $\mathbb{C}T$, we have $e_{\psi}  t_{i} = 0$.
\end{lemma}
\begin{proof}
Let ${\sigma_{1},\ldots,\sigma_{r}}$ be a complete set of representatives of $G/G_{i}$.  Then, we calculate
\begin{equation*}
\begin{aligned}
e_{\psi} t_{i} &= \frac{1}{|G|} \sum_{\sigma \in G}\psi(\sigma) \sigma^{-1} t_{i} \\
&= \frac{1}{|G|} \sum_{j = 1}^{r}\sum_{h \in G_{i}} \psi(\sigma_{j}h) (\sigma_{j}h)^{-1} t_{i} \\
&= \frac{1}{|G|} \left(\sum_{h \in G_{i}} \psi(h) \right)\sum_{j=1}^{r}  \psi(\sigma_{j}) \sigma_{j}^{-1} t_{i} \\
&= 0,
\end{aligned}
\end{equation*}
by the orthogonality relation (\ref{ortho}).
\end{proof}
This last lemma allows us to describe ${\rm Ann}_{\mathbb{C}[G]}(t_{i})$.
\begin{corollary}
With the notation as above, for each $i \in \{1,\ldots,g \}$, one has
$${\rm Ann}_{\mathbb{C}[G]}(t_{i}) = \mathbb{C}[G](1-e_{i}). $$
\end{corollary}
\begin{proof}
One has
$$e_{i} = \sum_{\substack{\psi \in \widehat{G} \\ G_{i} \subseteq {\rm ker}(\psi)}}e_{\psi}, $$
and thus \cref{useful_lemma} implies $\mathbb{C}[G](1-e_{i}) \subseteq {\rm Ann}_{\mathbb{C}[G]}(t_{i})$.  Noticing that ${\rm dim}_{\mathbb{C}}(\mathbb{C}[G]t_{i}) = (G:G_{i})$, a dimension count combined with the short exact sequence (\ref{ses_us}) gives the desired result.
\end{proof}
Combining this last corollary with \cref{useful_dec} shows also that $\mathbb{C}[G]t_{i}$ is isomorphic to $\mathbb{C}[G/G_{i}]$ as $\mathbb{C}[G]$-modules.  An explicit isomorphism is given by the map $\mathbb{C}[G]t_{i} \rightarrow \mathbb{C}[G/G_{i}]$ defined via
$$\lambda t_{i} \mapsto \pi_{i}(\lambda), $$
where $\pi_{i} = \pi_{G_{i}}:\mathbb{C}[G] \rightarrow \mathbb{C}[G/G_{i}]$ is the natural projection map.  Moreover, we also have that $\mathbb{C}[G]t_{i}$ is isomorphic to $\mathbb{C}[G]e_{i}$ by \cref{useful_dec}, where an explicit isomorphism $\mathbb{C}[G]t_{i} \rightarrow \mathbb{C}[G]e_{i} $ is given by
\begin{equation} \label{iso_u}
\lambda t_{i} \mapsto \lambda e_{i}. 
\end{equation}
It follows that if we let 
\begin{equation} \label{useful_sup}
M = \bigoplus_{i=1}^{g}\mathbb{C}[G](1-e_{i}),
\end{equation}
then $\mathbb{C}T \oplus M \simeq \mathbb{C}[G]^{g}$ and this gives a convenient way to view the projective $\mathbb{C}[G]$-module $\mathbb{C}T$ as a direct summand of a free $\mathbb{C}[G]$-module of finite rank.

We end this section with the following proposition that gives the dimension of the isotypic components of the $\mathbb{C}[G]$-module $\mathbb{C}T$.  For that purpose, if $\psi \in \widehat{G}$, then we let
\begin{equation} \label{the_r}
r_{T}(\psi) = |\{s \in S:G_{s} \not \subseteq {\rm ker}(\psi) \}|. 
\end{equation}
\begin{proposition} \label{dim_iso_component}
Let $\psi \in \widehat{G}$.  One has
\begin{equation*}
\begin{aligned}
{\rm dim}_{\mathbb{C}} (e_{\psi} \mathbb{C}T) &= |\{s \in S : G_{s} \subseteq {\rm ker}(\psi) \}| \\
&= |S| - r_{T}(\psi).
\end{aligned}
\end{equation*}
A $\mathbb{C}$-basis for the $\mathbb{C}$-vector space $e_{\psi}\mathbb{C}T$ is given by $\{e_{\psi} t_{i}: G_{i} \subseteq {\rm ker}(\psi) \}$.
\end{proposition}
\begin{proof}
We have
\begin{equation*}
\mathbb{C}T \simeq \bigoplus_{i=1}^{g} \mathbb{C}[G] e_{i}.
\end{equation*}
If $\psi \in \widehat{G}$, then 
\begin{equation*}
e_{\psi}\mathbb{C}[G]e_{i} \simeq 
\begin{cases}
0, & \text{ if } G_{i} \not \subseteq {\rm ker}(\psi);\\
\mathbb{C}, & \text{ otherwise}.
\end{cases}
\end{equation*}
It follows that ${\rm dim}_{\mathbb{C}} (e_{\psi} \mathbb{C}T) = |\{s \in S : G_{s} \subseteq {\rm ker}(\psi) \}|$.
\end{proof}

%\subsection{The $p$-adic valuation}
\subsection{The \texorpdfstring{$p$}{p}-adic valuation} \label{section:pad}
Although it is not strictly necessary, it is convenient in \cref{br_section} to work $p$-adically where $p$ is a rational prime, so here we introduce some relevant notation.  (One could instead choose to work globally with the cyclotomic number field of prime conductor $p$ and the valuation associated to its unique prime ideal lying above $p$.)  The symbol $\mathbb{Z}_{p}$ denotes the ring of $p$-adic integers and $\mathbb{Q}_{p}$ the field of $p$-adic numbers.  All the calculations will happen in an algebraic closure $\overline{\mathbb{Q}}_{p}$ of $\mathbb{Q}_{p}$, but it is convenient to work in a completion of $\overline{\mathbb{Q}}_{p}$ which we denote by $\mathbb{C}_{p}$.  The usual $p$-adic valuation on $\mathbb{Q}_{p}$ will be denoted by ${\rm ord}_{p}$ as well as its extension to $\mathbb{C}_{p}$.

\begin{lemma} \label{useful_r_t}
Let $p$ be a rational prime.  If $P(T) \in \overline{\mathbb{Q}}_{p}[T]$ is a nonzero polynomial, then 
$${\rm ord}_{p}(P(p^{n})) = \lambda n + \nu $$
provided $n$ is large enough, where 
$$\lambda = {\rm ord}_{T=0}(P(T)) \text{ and } \nu = {\rm ord}_{p}(a_{\lambda})$$
if $a_{\lambda}$ is the coefficient of $T^{\lambda}$.
\end{lemma}
\begin{proof}
Write 
$$P(T) = a_{\lambda}T^{\lambda} + \ldots + a_{d}T^{d} $$
for some $a_{j} \in \overline{\mathbb{Q}}_{p}$ with $a_{\lambda} \neq 0$.  Note that ${\rm ord}_{p}(a_{j}p^{jn}) = jn + {\rm ord}_{p}(a_{j})$ and thus
$$\lambda n + {\rm ord}_{p}(a_{\lambda}) < {\rm ord}_{p}(a_{j}p^{jn}) $$
for all $j > \lambda$ provided $n$ is large enough.  It follows that
\begin{equation*}
\begin{aligned} 
{\rm ord}_{p}(P(p^{n})) &= {\rm min}\{{\rm ord}_{p}(a_{j}p^{jn}) : j=\lambda,\ldots,d \}\\
&= \lambda n + \nu,
\end{aligned}
\end{equation*}
provided $n$ is large enough as we wanted to show.
\end{proof}

Recall that given a non-zero power series
$$P(T) = \sum_{i=0}^{\infty} a_{i}T^{i} \in \mathbb{Z}_{p}\llbracket T \rrbracket,$$ 
one defines
$$\mu(P(T)) = {\rm min}\{{\rm ord}_{p}(a_{i}): i \in \mathbb{Z}_{\ge 0}\}\text{ and } \lambda(P(T)) = {\rm min}\{i \in \mathbb{Z}_{\ge 0}: {\rm ord}_{p}(a_{i}) = \mu(P(T)) \}. $$

\begin{lemma} \label{useful_p_s}
Given $P(T) \in \mathbb{Z}_{p}\llbracket T \rrbracket$ such that $P(T) \neq 0$, and letting $\mu = \mu(P(T))$ and $\lambda = \lambda(P(T))$, there exists $\nu \in \mathbb{Z}$ such that if $n$ is large enough, one has
$$\sum_{j=0}^{n}\sum_{\substack{\zeta \in \mu_{p^{n}} \\ {\rm ord}(\zeta) = p^{j}}}{\rm ord}_{p}(P(\zeta - 1)) = \mu p^{n} + \lambda n + \nu, $$
where the second sum is over the $p^{n}$th roots of unity that have order precisely $p^{j}$.
\end{lemma}
\begin{proof}
Let $R(T)\in \mathbb{Z}_{p}\llbracket T \rrbracket$ be such that $P(T)=p^{\mu} R(T)$. Then, if $R(T)=\sum_{k=0}^{\infty} e_k T^k$, we have that $\lambda={\rm min} \{k \in \mathbb{Z}_{\geq 0} : {\rm ord}_{p}(e_k)=0\}$. Now, notice that it is enough to prove that for large $i$, 
\begin{equation} \label{R-p-evaluation}
{\rm ord}_{p} (R(\zeta-1))=\frac{\lambda}{\phi(p^i)}
\end{equation}
where ${\rm ord}(\zeta)=p^i$, and $\phi$ denotes the Euler $\phi$-function. We assume $i$ is large enough so that $\lambda < \phi(p^i)$ as this ensures that the values 
$${\rm ord}_{p}(e_k(\zeta-1)^k)= {\rm ord}_{p}(e_k)+\frac{k}{\phi(p^i)}$$
are distinct for $k=0,  ... , \lambda$. (Indeed, if two such values were equal with $0\leq j< k\leq \lambda$, then $(k-j)/\phi(p^i) \in \mathbb{Z}$ and $0<(k-j)/\phi(p^i)\leq \lambda/\phi(p^{i}) < 1$ which is a contradiction.) From this, together with the fact that for $k<\lambda$ one has 
$${\rm ord}_{p}(e_k(\zeta-1)^k) > {\rm ord}_{p}(e_k) \geq 1 > \frac{\lambda}{\phi(p^i)},$$
we obtain  
$${\rm ord}_{p} \left(\sum_{k=0}^{\lambda} e_k(\zeta-1)^k\right)=\frac{\lambda}{\phi(p^i)}.$$
In addition we have
$${\rm ord}_{p}\left(\sum_{k=\lambda+1}^{\infty} e_k(\zeta-1)^k\right) \geq    \min \Big\{\frac{k}{\phi(p^i)}: k> \lambda\Big\}=\frac{\lambda+1}{\phi(p^i)}$$ and (\ref{R-p-evaluation}) follows. This completes the proof.
\end{proof}

%\section{Graph theory}
\section{Graph theory} \label{section:gra}
%\subsection{Basic definitions}
\subsection{Basic definitions} \label{section:bas}
Throughout, we use Serre's formalism for graphs as detailed in \cite{Serre:1977} and \cite{Sunada:2013} for instance.  Here, we are deliberately brief and we refer the reader to \cite{Gambheera/Vallieres:2024} for more details, since we are using the exact same notation.  If $Y$ is a graph, $V_{Y}$ denotes its collection of vertices, $\mathbf{E}_{Y}$ its collection of directed edges, and $E_{Y}$ its collection of undirected edges.  If $\varepsilon \in \mathbf{E}_{Y}$, then $o(\varepsilon)$ denotes its origin vertex and $t(\varepsilon)$ its terminus vertex.  Moreover, the inversion map is denoted by $\varepsilon \mapsto \bar{\varepsilon}$.  Given $w \in V_{Y}$, we let
$$\mathbf{E}_{Y,w} = \{\varepsilon \in \mathbf{E}_{Y}: o(\varepsilon) = w \}, $$
and we set ${\rm val}_{Y}(w) = |\mathbf{E}_{Y,w}|$ provided the set $\mathbf{E}_{Y,w}$ is finite.  The graph $Y$ is called locally finite if $\mathbf{E}_{Y,w}$ is a finite set for all $w \in V_{Y}$ and is called finite if both $V_{Y}$ and $\mathbf{E}_{Y}$ are finite sets.  Paths in graphs are defined as usual.  A graph $Y$ is connected if there is a path between any two vertices in $Y$.  

Morphisms of graphs are defined as usual (see \cite[Definition 2.1]{Gambheera/Vallieres:2024}).  Branched covers are a particular kind of morphisms of graphs whose definition can be found in \cite[Definition 3.1]{Gambheera/Vallieres:2024}.  If $f:Y \rightarrow X$ is a branched cover of finite graphs and $w \in V_{Y}$, we denote its ramification index by $m_{w}$.  Under the assumption that $Y$ and $X$ are finite and connected, the quantity
$$\sum_{w \in f^{-1}(v)} m_{w}, $$
does not depend on $v \in V_{X}$.  This constant is denoted by $[Y:X]$, and is called the degree of the finite branched cover.  

%\subsection{Homology and the Euler characteristic}
\subsection{The Euler characteristic} \label{section:the}
Let $Y$ be a finite graph.  We set $C_{0}(Y) = \mathbb{C}V_{Y}$; using the notation of \cite{Gambheera/Vallieres:2024}, one has $C_{0}(Y) \simeq \mathbb{C} \otimes_{\mathbb{Z}} {\rm Div}(Y)$.  Inside $\mathbb{C}\mathbf{E}_{Y}$, we let $W$ be the $\mathbb{C}$-vector subspace generated by $\{\varepsilon + \bar{\varepsilon} : \varepsilon \in \mathbf{E}_{Y} \}$, and we define $C_{1}(Y) = \mathbb{C}\mathbf{E}_{Y}/W$.  We have a natural morphism of $\mathbb{C}$-vector spaces $\partial: \mathbb{C}\mathbf{E}_{Y} \rightarrow C_{0}(Y)$ defined on the $\mathbb{C}$-basis $\mathbf{E}_{Y}$ via
$$\varepsilon \mapsto \partial(\varepsilon) = t(\varepsilon) - o(\varepsilon). $$
Since $\partial(\varepsilon + \bar{\varepsilon}) = 0$, the map $\partial$ induces a morphism of $\mathbb{C}$-vector spaces $\partial:C_{1}(Y) \rightarrow C_{0}(Y)$.  This gives us a chain complex of $\mathbb{C}$-vector spaces
\begin{equation} \label{chain_complex}
0 \rightarrow C_{1}(Y) \stackrel{\partial}{\longrightarrow} C_{0}(Y) \rightarrow 0. 
\end{equation}
\begin{definition} \label{def_euler_char}
The Euler characteristic of $Y$ is the Euler characteristic of the chain complex (\ref{chain_complex}), that is $\chi(Y) = {\rm dim}_{\mathbb{C}}(C_{0}(Y)) - {\rm dim}_{\mathbb{C}}(C_{1}(Y))$.
\end{definition}

The Euler characteristic of a finite graph can therefore be calculated using the following simple formula.
\begin{proposition} \label{eul_form}
Let $Y$ be a finite graph, then $\chi(Y) = |V_{Y}| - |E_{Y}|$.
\end{proposition}
\begin{proof}
This follows directly from \cref{def_euler_char} and the observations that
$${\rm dim}_{\mathbb{C}}(C_{0}(Y)) = |V_{Y}| \text{ and } {\rm dim}_{\mathbb{C}}(C_{1}(Y)) = |E_{Y}|. $$
\end{proof}
In \cite{Baker/Norine:2009}, Baker and Norine proved an analogue of the Riemann-Hurwitz formula for branched covers of graphs (and more generally for harmonic morphisms).  We will make use of this formula in several instances throughout this paper (see also \cite[Theorem 5.12]{Sunada:2013}).  We record it here for future reference.
\begin{theorem}[Riemann-Hurwitz formula] \label{riemann_hurwitz}
Let $f:Y \rightarrow X$ be a branched cover of finite connected graphs.  Then, one has
$$\chi(Y) = [Y:X] \cdot \chi(X) - \sum_{w \in V_{Y}}(m_{w}-1).$$
\end{theorem}

\begin{comment}
\begin{proof}
Observe first that we have
\begin{equation*}
\begin{aligned}
|\mathbf{E}_{Y}| &= \sum_{v \in V_{X}} \sum_{w \in f^{-1}(v)}|\mathbf{E}_{Y,w}| \\
&= \sum_{v \in V_{X}} \sum_{w \in f^{-1}(v)} m_{w} \cdot |\mathbf{E}_{X,f(w)}| \\
&= \sum_{v \in V_{X}} |\mathbf{E}_{X,v}| \sum_{w \in f^{-1}(v)} m_{w} \\
&= [Y:X] \cdot |\mathbf{E}_{X}|,
\end{aligned}
\end{equation*}
where the last equality follows from \cite[(3.2)]{Gambheera/Vallieres:2024}.  We then have
\begin{equation*}
\begin{aligned}
\chi(Y) &= |V_{Y}| - \frac{|\mathbf{E}_{Y}|}{2} \\
&= |V_{Y}| - [Y:X] \cdot \frac{|\mathbf{E}_{X}|}{2} \\
&= |V_{Y}| - [Y:X] \cdot \frac{|\mathbf{E}_{X}|}{2} + [Y:X] \cdot |V_{X}| - [Y:X] \cdot |V_{X}| \\
&= [Y:X] \cdot \chi(X) + |V_{Y}| - [Y:X] \cdot |V_{X}| \\
&= [Y:X] \cdot \chi(X) + \sum_{w \in V_{Y}} 1 - \sum_{v \in V_{X}}[Y:X] \\
&= [Y:X] \cdot \chi(X) + \sum_{w \in V_{Y}} 1 - \sum_{v \in V_{X}} \sum_{w \in f^{-1}(v)} m_{w} \\
&= [Y:X] \cdot \chi(X) + \sum_{w \in V_{Y}} 1 - \sum_{v \in V_{Y}} m_{w} \\
&= [Y:X] \cdot \chi(X) - \sum_{w \in V_{Y}}(m_{w}-1),
\end{aligned}
\end{equation*}
and this ends the proof.
\end{proof}
\end{comment} 

%\subsection{The laplacian operator}
\subsection{The laplacian operator} \label{op}
In this section, we define a few operators on $C_{0}(Y)$ under the assumption that $Y$ is locally finite.  First, one defines the valency (or degree) operator $\mathcal{D}:C_{0}(Y) \rightarrow C_{0}(Y)$ via 
$$w \mapsto {\rm val}_{Y}(w) \cdot w.$$  
Second, one defines the adjacency operator $\mathcal{A}:C_{0}(Y) \rightarrow C_{0}(Y)$ via 
$$w \mapsto \mathcal{A}(w) = \sum_{\varepsilon \in \mathbf{E}_{Y,w}} t(\varepsilon). $$
If moreover $Y$ is finite and one chooses a labeling $V_{Y} = \{w_{1},\ldots,w_{q} \}$ of the vertices, then the matrices of $\mathcal{D}$ and $\mathcal{A}$ with respect to the ordered $\mathbb{C}$-basis $(w_{1},\ldots,w_{q})$ are called the valency (or degree) matrix and the adjacency matrix of $Y$, respectively.  Those two matrices are simply denoted by $D$ and $A$.  The laplacian operator is defined to be $\mathcal{L} = \mathcal{D} - \mathcal{A}$ and its matrix with respect to the same ordered $\mathbb{C}$-basis as above is called the laplacian matrix.  One sets ${\rm Pr}(Y) = \mathcal{L}(C_{0}(Y))$ and when $Y$ is finite and connected, one has the short exact sequence
\begin{equation} \label{ses_laplacian}
0 \rightarrow \mathbb{C} \sum_{w \in V_{Y}} w \rightarrow C_{0}(Y) \stackrel{\mathcal{L}}{\longrightarrow} {\rm Pr}(Y) \rightarrow 0
\end{equation}
of $\mathbb{C}$-vector spaces.  We refer the reader to \cite[Theorem 2.3]{Gambheera/Vallieres:2024} for more details.  A consequence of (\ref{ses_laplacian}) is that
\begin{equation} \label{lap_singular}
{\rm det}_{\mathbb{C}}(\mathcal{L}) = 0,
\end{equation}
since the kernel of $\mathcal{L}:C_{0}(Y) \rightarrow C_{0}(Y)$ is non-trivial.

%\subsection{The Ihara zeta function}
\subsection{The Ihara zeta function} \label{section:iha}
We refer the reader to \cite{Terras:2011} for the basic definitions and results pertaining to the Ihara zeta function of a finite graph, and here we will be deliberately brief.  If $Y$ is a finite graph, we let
$$Z_{Y}(u) = \exp_{\mathbb{Q}}\left(\sum_{k=1}^{\infty} N_{k} \frac{u^{k}}{k} \right) \in 1+ u\mathbb{Q}\llbracket u \rrbracket, $$
where $N_{k}$ is the number of reduced closed paths in $Y$ of length $k$.  At first, $Z_{Y}(u)$ is a formal power series with rational coefficients, but since
$$Z_{Y}(u) = \prod_{\mathfrak{c}}(1-u^{{\rm len}(\mathfrak{c})})^{-1},$$
where the product is over the primes $\mathfrak{c}$ of $Y$ and ${\rm len}(\mathfrak{c})$ denotes the length of $\mathfrak{c}$, it follows that $Z_{Y}(u) \in 1+ u\mathbb{Z}\llbracket u \rrbracket$.  In fact, $Z_{Y}(u)$ is the reciprocal of a polynomial with integer coefficients as the following theorem shows for which an elegant proof can be found in \cite{Kotani/Sunada:2000} for instance.
\begin{theorem}[Ihara's determinant formula] \label{ihara}
Let $Y$ be a finite graph.  Its Ihara zeta function satisfies
$$Z_{Y}(u)^{-1} = (1-u^{2})^{-\chi(Y)}\cdot {\rm det}(I - Au + (D - I)u^{2}) \in \mathbb{Z}[u], $$
where $A$ is the adjacency matrix and $D$ the degree matrix of $Y$.
\end{theorem}
From now on, we let
$$c_{Y}(u) = (1-u^{2})^{-\chi(Y)} \text{ and } h_{Y}(u) = {\rm det}(I - Au + (D - I)u^{2}). $$
Note that $h_{Y}(u) \in 1+u\mathbb{Z}[u]$.  We will be particularly interested in the special value at $u=1$ of Ihara zeta and $L$-functions.  For the following theorem, we refer the reader to \cite{Hashimoto:1990}.
\begin{theorem} \label{hashimoto}
Let $Y$ be a finite connected graph.  Then, one has
$$h_{Y}'(1) = -2 \chi(Y) \kappa(Y), $$
where $\kappa(Y)$ denotes the number of spanning trees of $Y$.
\end{theorem}

%\section{Equivariant graph theory}
\section{Equivariant graph theory} \label{equ}
%\subsection{Groups acting on graphs}
\subsection{Groups acting on graphs} \label{gr_act}
A group $G$ acts on a graph $Y$ if we are given a group morphism $G \rightarrow {\rm Aut}(Y)$.  The action is said to be without inversion if one has $\sigma \cdot \varepsilon \neq \bar{\varepsilon}$ for all $\varepsilon \in \mathbf{E}_{Y}$ and all $\sigma \in G$.  In this case, the quotient $G \backslash Y = (G \backslash V_{Y}, G \backslash \mathbf{E}_{Y})$ is a graph in itself with the incidence and inversion maps induced from the ones on $Y$.  We shall write $Y_{G}$ instead of $G \backslash Y$ throughout. 

As explained in \cite[Proposition 3.3]{Gambheera/Vallieres:2024}, if $G$ is a finite group acting on a graph $Y$ without inversion and freely on $\mathbf{E}_{Y}$, then the natural morphism of graphs $f: Y \rightarrow X$, where $X = Y_{G}$, is a branched cover of graphs for which $m_{w} = |{\rm Stab}_{G}(w)|$ for all $w \in V_{Y}$.  From now on, we let $G_{w} = {\rm Stab}_{G}(w)$.  Note that since stabilizers of vertices in the same orbits are conjugate, the ramification indices $m_{w}$ depend only on $f(w)$.  We will often write $m_{v}$ for the ramification index of a vertex $w \in V_{Y}$, where $v = f(w)$.  

In this paper, $G$ will always be a finite abelian group.  Therefore, the stabilizer group $G_{w}$ also depends only on $v = f(w)$, and we shall often denote this stabilizer group by $G_{v}$, where $v = f(w)$.  Thus, when $G$ is abelian $m_{v} = |G_{v}|$, and there are $(G:G_{v})$ vertices in $Y$ lying above $v$.

%\subsection{Group actions obtained from voltage assignments}
\subsection{Group actions obtained from voltage assignments} \label{volta}
As explained in \cite[\S 4.1]{Gambheera/Vallieres:2024}, one can use voltage assignments in order to construct graphs on which a group acts freely on directed edges.  We refer the reader to \cite[\S 4.1]{Gambheera/Vallieres:2024} for the detailed construction.  Briefly, the construction starts with a base graph $X$, a group $G$, a function $\alpha:\mathbf{E}_{X} \rightarrow G$ satisfying $\alpha(\bar{s}) = \alpha(s)^{-1}$ for all $s \in \mathbf{E}_{X}$, and a collection of subgroups 
$$\mathcal{G} = \{(v,G_{v}) : v \in V_{X} \text{ and } G_{v} \le G \}. $$
To this data, one associates a graph $Y = X(G,\mathcal{G},\alpha)$ which is a branched cover of $X$.  Moreover, the group $G$ acts without inversion on $Y$ and freely on directed edges.  See \cite[Proposition 4.1]{Gambheera/Vallieres:2024} for the complete list of properties that this branched cover satisfies.  

For us in this paper, $G$ will always be a finite abelian group.  One thing we did not mention in \cite{Gambheera/Vallieres:2024} is that one always has $Y_{G} \simeq X$.  In fact, we leave it to the reader to check that the following more general statement holds true:  If $H$ is a subgroup of $G$ and $\Gamma = G/H$, then the map 
\begin{equation} \label{identification}
f: Y_{H} \rightarrow X(\Gamma,\pi_{H}(\mathcal{G}),\pi_{H} \circ \alpha)
\end{equation}
given by
$$H (v, \sigma G_{v}) \mapsto f(H (v,\sigma G_{v})) = (v,\pi_{H}(\sigma) \pi_{H}(G_{v})) \text{ and } H(e,\sigma) \mapsto f(H(e,\sigma)) = (e,\pi_{H}(\sigma)) $$
is a well-defined $\Gamma$-equivariant isomorphism of graphs, where $\pi_{H}:G \rightarrow \Gamma$ denotes the natural surjective group morphism, and where we use the notation of \cite[\S 4.2]{Gambheera/Vallieres:2024}.  This will allow us to identify $Y_{H}$ with $X(\Gamma,\pi_{H}(\mathcal{G}),\pi_{H} \circ \alpha)$ in the future.  For more information on voltage assignments, and more generally on the classification of harmonic $G$-covers, we refer the reader to \cite{Len/Ulirsch/Zakharov:2024}.

\subsection{The equivariant Euler characteristic} \label{equiv_sec}
Let $G$ be a finite abelian group acting on a finite graph $Y$ without inversion and let $X = Y_{G}$.  The $\mathbb{C}$-vector space $C_{0}(Y)$ becomes a $\mathbb{C}[G]$-module and is a permutation module so we can apply the results of \cref{permutation}.  Thus, we now introduce a similar notation to the one used in \cref{permutation} for $C_{0}(Y)$.  We fix a labeling $V_{X} = \{v_{1},\ldots,v_{g} \}$ of the vertices of $X$ and for each $i = 1,\ldots,g$, we let $w_{i}$ be a fixed vertex of $Y$ lying above $v_{i}$.  For $i=1,\ldots,g$, we write $G_{i}$ instead of $G_{v_{i}} = {\rm Stab}_{G}(w_{i})$.

The $\mathbb{C}$-vector space $\mathbb{C}\mathbf{E}_{Y}$ is also a permutation module, and therefore a similar notation could be used.  In particular, we will use the notation $G_{\varepsilon}$ to denote ${\rm Stab}_{G}(\varepsilon)$ whenever $\varepsilon \in \mathbf{E}_{Y}$.  Moreover, $W$ is a $\mathbb{C}[G]$-submodule of $\mathbb{C}\mathbf{E}_{Y}$, and is in fact a permutation module itself.  Indeed, if $O$ is an orientation for $Y$, that is $O \subseteq \mathbf{E}_{Y}$ is such that 
$O \cap \bar{O} = \varnothing$ and $\mathbf{E}_{Y} = O \cup \bar{O}$, then $\Omega = \{\varepsilon + \bar{\varepsilon} : \varepsilon \in O\}$ is $G$-stable so that $\Omega$ is a $G$-set, and $W \simeq \mathbb{C}\Omega$.  Since both $W$ and $\mathbb{C}\mathbf{E}_{Y}$ are $\mathbb{C}[G]$-modules, then so is the quotient $C_{1}(Y)$.

The $\mathbb{C}$-linear morphism $\partial$ can be checked to be $G$-equivariant, from which it follows that the chain complex (\ref{chain_complex}) becomes a chain complex of finitely generated $\mathbb{C}[G]$-modules.  Since $\mathbb{C}[G]$ is semisimple, both $C_{0}(Y)$ and $C_{1}(Y)$ are projective $\mathbb{C}[G]$-modules and their (Hattori-Stallings) ranks over $\mathbb{C}[G]$ therefore make sense, as explained in  \cref{com_alg}.

\begin{definition}
Let $Y$ be a finite graph on which a finite abelian group $G$ acts without inversion.  One defines the equivariant Euler characteristic of $Y$ to be
$$\chi_{\mathbb{C}[G]}(Y) = {\rm rank}_{\mathbb{C}[G]} (C_{0}(Y)) - {\rm rank}_{\mathbb{C}[G]} (C_{1}(Y)).$$
\end{definition}

Our next goal is to find an explicit formula for the equivariant Euler characteristic similar to the one in \cref{eul_form} for the usual Euler characteristic.  
\begin{theorem} \label{exp_equi_euler_char}
Let $G$ be a finite abelian group acting on a finite graph $Y$ without inversion and set $X = Y_{G}$.  Let $O_{X}$ be an orientation of $X$.  Then
$$\chi_{\mathbb{C}[G]}(Y) = \sum_{v \in V_{X}} e_{G_{v}} - \sum_{s \in O_{X}}e_{G_{s}}.$$
In particular, if $G$ acts freely on directed edges, then
$$\chi_{\mathbb{C}[G]}(Y) = \sum_{v \in V_{X}} e_{G_{v}} - |E_{X}|.$$
\end{theorem}
\begin{proof}
Since
$$C_{0}(Y) = \bigoplus_{i = 1}^{g}\mathbb{C}[G]w_{i} \simeq \bigoplus_{i=1}^{g} \mathbb{C}[G]e_{i}, $$
by the isomorphisms (\ref{iso_u}) and since the rank over $\mathbb{C}[G]$ is additive (see \cref{com_alg}), it suffices to understand the rank over $\mathbb{C}[G]$ of the $\mathbb{C}[G]$-modules $\mathbb{C}[G]e_{i}$.  But a simple calculation using the supplement $\mathbb{C}[G](1-e_{i})$ of $\mathbb{C}[G]e_{i}$ gives ${\rm rank}_{\mathbb{C}[G]}(\mathbb{C}[G]e_{i}) = e_{i}$.  Therefore, we have
\begin{equation} \label{c_0}
{\rm rank}_{\mathbb{C}[G]} (C_{0}(Y)) = \sum_{v \in V_{X}} e_{G_{v}}. 
\end{equation}
For $C_{1}(Y)$, we consider the short exact sequence
$$0 \rightarrow W \rightarrow \mathbb{C}\mathbf{E}_{Y} \rightarrow C_{1}(Y) \rightarrow 0 $$
of finitely generated projective $\mathbb{C}[G]$-modules.  Both $W$ and $\mathbb{C}\mathbf{E}_{Y}$ are permutation modules and thus their rank can be calculated similarly as above.  Noting that $G_{\varepsilon} = G_{\bar{\varepsilon}} = G_{\varepsilon + \bar{\varepsilon}}$, one calculates as above
$${\rm rank}_{\mathbb{C}[G]}(\mathbb{C}\mathbf{E}_{Y}) = \sum_{s \in \mathbf{E}_{X}}e_{G_{s}} \text{ and } {\rm rank}_{\mathbb{C}[G]}(W) = \sum_{s \in O_{X}}e_{G_{s}}, $$
and thus
\begin{equation} \label{c_1}
{\rm rank}_{\mathbb{C}[G]}(C_{1}(Y)) = \sum_{s \in O_{X}}e_{G_{s}}.
\end{equation}
Putting (\ref{c_0}) and (\ref{c_1}) together gives the desired result.

\end{proof}

%\subsection{The equivariant Ihara zeta function}
\subsection{The equivariant Ihara zeta function} \label{section:theequiv}
We place ourselves again in the situation where a finite abelian group $G$ acts on a finite graph $Y$ without inversion.  Any $\sigma \in G$ induces a bijection $\mathbf{E}_{Y,w} \rightarrow \mathbf{E}_{Y,\sigma \cdot w}$ from which it follows that both the degree operator $\mathcal{D}$ and the adjacency operator $\mathcal{A}$ are morphisms of $\mathbb{C}[G]$-modules.  We let $\mathcal{I}$ be the identity operator on $C_{0}(Y)$, and we define also $\mathcal{Q} = \mathcal{D} - \mathcal{I}$.  The following polynomial
\begin{equation} \label{delta_}
\Delta(u) := \mathcal{I} - \mathcal{A}u + \mathcal{Q}u^{2} \in {\rm End}_{\mathbb{C}[G]}(C_{0}(Y))[u] 
\end{equation}
will play an important role.  Via the isomorphism 
$$\omega_{C_{0}(Y)}:{\rm End}_{\mathbb{C}[G]}(C_{0}(Y))[u] \stackrel{\simeq}{\longrightarrow} {\rm End}_{\mathbb{C}[G][u]}(C_{0}(Y)[u]) $$
given by \cref{conv_iso_pol_vs}, we will identify polynomials in ${\rm End}_{\mathbb{C}[G]}(C_{0}(Y))[u]$ with endomorphisms in ${\rm End}_{\mathbb{C}[G][u]}(C_{0}(Y)[u])$.  In particular, since $C_{0}(Y)[u]$ is a projective $\mathbb{C}[G][u]$-module, it makes sense to talk about the determinant of the $\mathbb{C}[G][u]$-operator $\Delta(u)$ defined above in (\ref{delta_}), and we shall denote it by
$$\eta_{Y}(u) := {\rm det}_{\mathbb{C}[G][u]}(\Delta(u)) \in 1 + u\mathbb{C}[G][u]. $$
Moreover, the expression
$$\gamma_{Y}(u) := (1-u^{2})^{-\chi_{\mathbb{C}[G]}(Y)} \in 1 +u\mathbb{C}[G] \llbracket u \rrbracket  $$
makes sense as well and was defined in (\ref{exponent}) above.  Note that
$$\gamma_{Y}(u), \eta_{Y}(u) \in 1 + u\mathbb{C}[G]\llbracket u \rrbracket \subseteq \mathbb{C}[G]\llbracket u \rrbracket^{\times}.$$
We can now define the equivariant Ihara zeta function $\theta_{Y}(u)$ which is the central object of study in this paper.  
\begin{definition} \label{equiv_ihara}
With the notation as above, we define the equivariant Ihara zeta function $\theta_{Y}(u)$ via
$$\theta_{Y}(u)^{-1} = \gamma_{Y}(u) \cdot \eta_{Y}(u) \in 1 + u\mathbb{C}[G]\llbracket u \rrbracket \subseteq \mathbb{C}[G]\llbracket u \rrbracket^{\times}. $$
\end{definition}
One way to calculate $\eta_{Y}(u)$ explicitly, is to find a $\mathbb{C}[G]$-module $M$ such that
$$C_{0}(Y)[u] \oplus M[u] \simeq (C_{0}(Y) \oplus M)[u] $$
is free of finite rank over $\mathbb{C}[G][u]$, and then calculate ${\rm det}_{\mathbb{C}[G][u]}(\Delta(u) \oplus {\rm id}_{M[u]})$.  Such a module $M$ is provided by (\ref{useful_sup}), and in \cref{explicit} below we carry this out in detail.

%\subsection{An explicit formula for the equivariant Ihara zeta function}
\subsection{An explicit formula for the equivariant Ihara zeta function} \label{explicit}
We use the same notation as in the previous section.  Given $w, w' \in V_{Y}$, we define
$$a_{w}(w') = |\{\varepsilon \in \mathbf{E}_{Y}:o(\varepsilon) = w' \text{ and } t(\varepsilon) = w\}|, $$
and for $i=1,\ldots,g$, we also define $\ell_{i}:C_{0}(Y) \rightarrow \mathbb{C}[G]$ via
$$w \mapsto \ell_{i}(w) = \frac{1}{|G_{i}|}\sum_{\sigma \in G}a_{w_{i}}(\sigma w) \sigma^{-1}. $$
We leave it to the reader to check that $\ell_{i}$ is a morphism of $\mathbb{C}[G]$-modules.  We will often use the fact that 
\begin{equation} \label{useful_identity}
a_{w}(\sigma w') = a_{\sigma^{-1}w}(w')
\end{equation}
for all $w,w' \in V_{Y}$ and for all $\sigma \in G$, which follows from the observation that the function
$$\{\varepsilon \in \mathbf{E}_{Y}: {\rm inc}(\varepsilon) = (\sigma w',w) \} \rightarrow \{\varepsilon \in \mathbf{E}_{Y}: {\rm inc}(\varepsilon) = (w',\sigma^{-1}w) \} $$
given by $\varepsilon \mapsto \sigma^{-1} \varepsilon$ is a bijection.
\begin{proposition} \label{dec_adj}
For all $D \in C_{0}(Y)$, one has
$$\mathcal{A}(D) = \sum_{i=1}^{g} \ell_{i}(D) w_{i} $$
\end{proposition}
\begin{proof}
Note that it suffices to show the claim for $D = w$, since the vertices of $Y$ generate $C_{0}(Y)$ as a $\mathbb{C}$-vector space.  Using (\ref{useful_identity}), we calculate
\begin{equation*}
\begin{aligned} 
\mathcal{A}(w) &= \sum_{w_{0} \in V_{Y}}a_{w_{0}}(w) w_{0} \\
&= \sum_{i=1}^{g} \sum_{\sigma \in G/G_{i}} a_{\sigma w_{i}}(w)\sigma w_{i} \\
&= \sum_{i=1}^{g} \frac{1}{|G_{i}|} \sum_{\sigma \in G}a_{\sigma w_{i}}(w) \sigma w_{i} \\
&= \sum_{i=1}^{g} \frac{1}{|G_{i}|} \sum_{\sigma \in G}a_{w_{i}}(\sigma^{-1} w) \sigma w_{i} \\
&= \sum_{i=1}^{g} \ell_{i}(w)w_{i},
\end{aligned}
\end{equation*}
as we wanted to show.
\end{proof}
From now on, we let
$$\widetilde{C}_{0}(Y) = \bigoplus_{i=1}^{g} \mathbb{C}[G]e_{i}.$$
Recall that we have a $\mathbb{C}[G]$-module isomorphism
\begin{equation} \label{tri_iso}
C_{0}(Y)  \stackrel{\simeq}{\longrightarrow} \widetilde{C}_{0}(Y), 
\end{equation}
which is induced by the $\mathbb{C}[G]$-module isomorphisms $\mathbb{C}[G]w_{i} \stackrel{\simeq}{\longrightarrow} \mathbb{C}[G]e_{i}$ from (\ref{iso_u}) above.  We let $\widetilde{\mathcal{A}}$ and $\widetilde{\mathcal{D}}$ be the $\mathbb{C}[G]$-module morphisms $\widetilde{C}_{0}(Y) \rightarrow \widetilde{C}_{0}(Y)$ that make the two diagrams
\begin{equation*} 
	\begin{CD}
		 C_{0}(Y)    @>{\simeq}>>   \widetilde{C}_{0}(Y) \\
		 @VV{\mathcal{A}}V       @VV{ \widetilde{\mathcal{A}}}V    &  \\
		 C_{0}(Y)    @>{\simeq}>>   \widetilde{C}_{0}(Y)
	\end{CD}
\, \, \text{ and } \, \, 
        \begin{CD}
		 C_{0}(Y)    @>{\simeq}>>   \widetilde{C}_{0}(Y) \\
		 @VV{ \mathcal{D}}V       @VV{\widetilde{\mathcal{D}}}V    &  \\
		 C_{0}(Y)    @>{\simeq}>>   \widetilde{C}_{0}(Y)
	\end{CD}
\end{equation*}
commutative, where the horizontal $\mathbb{C}[G]$-module isomorphisms are the map (\ref{tri_iso}).  The following lemma will simplify a few formulas later on.
\begin{lemma} \label{sim_for}
With the same notation as above, for all $i=1,\ldots,g$ and all $D \in C_{0}(Y)$, one has
$$\ell_{i}(D)e_{i} = \ell_{i}(D). $$
Moreover,
$$\ell_{i}(w_{j})e_{j} = \ell_{i}(w_{j})$$
for all $i,j=1,\ldots,g$.
\end{lemma}
\begin{proof}
It suffices to show this when $D=w$.  Using (\ref{useful_identity}), we calculate
\begin{equation*}
\begin{aligned} 
\ell_{i}(w) e_{i} &= \frac{1}{|G_{i}|} \sum_{\sigma \in G}a_{w_{i}}(\sigma w) \sigma^{-1}e_{i} \\
&= \frac{1}{|G_{i}|^{2}} \sum_{\sigma \in G} \sum_{h \in G_{i}} a_{w_{i}}(\sigma w) (h\sigma)^{-1} \\
&= \frac{1}{|G_{i}|^{2}} \sum_{\sigma \in G} \sum_{h \in G_{i}} a_{w_{i}}(h^{-1}\sigma w) \sigma^{-1} \\
&= \frac{1}{|G_{i}|^{2}} \sum_{\sigma \in G} \sum_{h \in G_{i}} a_{hw_{i}}(\sigma w) \sigma^{-1} \\
&= \frac{1}{|G_{i}|} \sum_{\sigma \in G} a_{w_{i}}(\sigma w) \sigma^{-1} \\
&= \ell_{i}(w).
\end{aligned}
\end{equation*}
The second claim is proved similarly, and we leave it to the reader.
\end{proof}
Consider now
$$\widetilde{\Delta}(u) = \widetilde{\mathcal{I}} - \widetilde{\mathcal{A}}u + \widetilde{\mathcal{Q}}u^{2} \in {\rm End}_{\mathbb{C}[G]}(\widetilde{C}_{0}(Y))[u], $$
where $\widetilde{Q} = \widetilde{\mathcal{D}}- \widetilde{\mathcal{I}}$, and $\widetilde{\mathcal{I}}$ is the identity operator on $\widetilde{C}_{0}(Y)$.  As before, we view 
$$\widetilde{\Delta}(u) \in {\rm End}_{\mathbb{C}[G][u]}(\widetilde{C}_{0}(Y)[u])$$ 
via the isomorphism
$$\omega_{\widetilde{C}_{0}(Y)}:{\rm End}_{\mathbb{C}[G]}(\widetilde{C}_{0}(Y))[u] \stackrel{\simeq}{\longrightarrow} {\rm End}_{\mathbb{C}[G][u]}(\widetilde{C}_{0}(Y)[u]) $$
given by \cref{conv_iso_pol_vs}.  Recall now the $\mathbb{C}[G]$-module $M$ from (\ref{useful_sup}).  We have 
\begin{equation} \label{yo}
\widetilde{C}_{0}(Y)[u] \oplus M[u] \simeq (\mathbb{C}[G][u])^{g}, 
\end{equation}
and therefore, we have
$$\eta_{Y}(u) = {\rm det}_{\mathbb{C}[G][u]}(\widetilde{\Delta}(u) \oplus {\rm id}_{M[u]}). $$
A simple calculation and \cref{sim_for} shows that the matrix $(\widetilde{\delta}_{ij}(u))$ for the $\mathbb{C}[G][u]$-operator $\widetilde{\Delta}(u) \oplus {\rm id}_{M[u]}$ with respect to the standard $\mathbb{C}[G][u]$-basis of the free $\mathbb{C}[G][u]$-module (\ref{yo}) is given by
\begin{equation*}
\widetilde{\delta}_{ij}(u) =
\begin{cases}
- \ell_{i}(w_{j})u, &\text{ if } i \neq j;\\
1 - \ell_{i}(w_{i})u + ({\rm val}_{Y}(w_{i}) - 1)e_{i}u^{2}, &\text{ if } i=j.
\end{cases}
\end{equation*}
This last formula can be rephrased as follows.
\begin{theorem} \label{equi_exp}
With the notation as above, let $\mathbf{A} = (\mathbf{a}_{ij}) \in M_{g}(\mathbb{C}[G])$ be the matrix given by $\mathbf{a}_{ij} = \ell_{i}(w_{j})$, and let $\mathbf{Q} = (\mathbf{q}_{ij}) \in M_{g}(\mathbb{C}[G])$ be the diagonal matrix given by $\mathbf{q}_{ii} = ({\rm val}_{Y}(w_{i}) - 1)e_{i}$.  Then, one has
$$\eta_{Y}(u) = {\rm det}(\mathbf{I} - \mathbf{A}u + \mathbf{Q}u^{2}), $$
where $\mathbf{I}$ is the $g \times g$ identity matrix in $M_{g}(\mathbb{C}[G])$.
\end{theorem}
This last theorem is convenient to calculate the polynomial $\eta_{Y}(u)$.  We show how this can be done in the following example. 

\begin{example} \label{ex11}
We use the construction outlined in \cref{volta} and explained in detail in \cite[\S 4.1]{Gambheera/Vallieres:2024} to give an example.  We take $X$ to be the graph consisting of two vertices joined by two undirected edges.  We label the vertices $V_{X} = \{v_{1},v_{2}\}$ and we let $s_{1},s_{2}$ be the two directed edges going from $v_{1}$ to $v_{2}$.  We take $G = \mathbb{Z}/4\mathbb{Z}$ to be a cyclic group of order four.  We take the voltage assigment $\alpha:\mathbf{E}_{X} \rightarrow G$ given by
$$\alpha(s_{1}) = \bar{1}, \text{ and } \alpha(s_{2}) = \bar{2}. $$
Moreover, we let $G_{1} = \{\bar{0} \}$, and $G_{2} = \{\bar{0},\bar{2}\} = \langle \bar{2} \rangle$.  The graph $Y=X(G,\mathcal{G},\alpha)$ is a branched cover of $X$ with ramification above the vertex $v_{2}$ only, and can be visualized as follows:
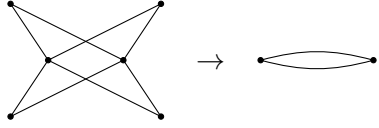
\begin{equation*}  
\begin{tikzpicture}[baseline={([yshift=-0.7ex] current bounding box.center)}]
%vertices
\draw[fill=black] (-1,3/4) circle (1pt);
\draw[fill=black] (-1,-3/4) circle (1pt);
\draw[fill=black] (1,3/4) circle (1pt);
\draw[fill=black] (1,-3/4) circle (1pt);
\draw[fill=black] (-1/2,0) circle (1pt);
\draw[fill=black] (1/2,0) circle (1pt);
%edges
\path (-1,3/4) edge (-1/2,0);
\path (-1,3/4) edge (1/2,0);
\path (1,3/4) edge (-1/2,0);
\path (1,3/4) edge (1/2,0);
\path (1,-3/4) edge (-1/2,0);
\path (1,-3/4) edge (1/2,0);
\path (-1,-3/4) edge (-1/2,0);
\path (-1,-3/4) edge (1/2,0);
\end{tikzpicture}
\, \, \, \, \, \rightarrow \, \, \, \, \, 
\begin{tikzpicture}[baseline={([yshift=-0.7ex] current bounding box.center)}]
%vertices
\draw[fill=black] (-3/4,0) circle (1pt);
\draw[fill=black] (3/4,0) circle (1pt);
%edges
\path (-3/4,0) edge [bend right=15] (3/4,0);
\path (-3/4,0) edge [bend left=15] (3/4,0);
\end{tikzpicture}
\end{equation*}
We let
$$w_{1} = (v_{1},\bar{0}) \text{ and } w_{2} = (v_{2},G_{2}). $$
With this choice of the vertices $w_{i} \in V_{Y}$, one calculates the $\ell_{i}(w_{j})$ and $({\rm val}_{Y}(w_{i})-1)e_{i}$ to obtain
\begin{equation*}
\mathbf{A} = \begin{pmatrix} 0 & {N_{G}} \\ \frac{1}{2}N_{G} & 0 \end{pmatrix}
\text{ and }
\mathbf{Q} = \begin{pmatrix} 1 & 0 \\ 0&3e_{2} \end{pmatrix}
\end{equation*}
so that
\begin{equation*}
\begin{aligned}
\eta_{Y}(u) &= {\rm det}(\mathbf{I} - \mathbf{A}u + \mathbf{Q}u^{2}) \\
&= 1 + \frac{1}{2}(\bar{0} - 4\cdot \bar{1} - \bar{2} - 4\cdot\bar{3})u^{2} + \frac{3}{2}(\bar{0} + \bar{2})u^{4} \in \mathbb{C}[G][u]. 
\end{aligned}
\end{equation*}
\demo
\end{example}

%\subsection{The induction and inflation properties}
\subsection{The induction and inflation properties} \label{ind_inf}
Throughout this section, $G$ will be a finite abelian group, $H \le G$, and $\Gamma = G/H$.  Let $Y$ be a finite graph on which $G$ acts without inversion.  Then $H$ also acts on $Y$ without inversion.  In addition to the natural projection map $\pi_{H}:\mathbb{C}[G] \rightarrow \mathbb{C}[\Gamma]$, which is a $\mathbb{C}$-algebra morphism, we will introduce two functions
\begin{equation} \label{t_and_n}
T_{G/H}:\mathbb{C}[G] \rightarrow \mathbb{C}[H] \text{ and } N_{G/H}:\mathbb{C}[G] \rightarrow \mathbb{C}[H] 
\end{equation}
that are not $\mathbb{C}$-algebra morphisms in general, but are nonetheless useful.  Their definitions are as follows.  The module $\mathbb{C}[G]$ is a free $\mathbb{C}[H]$-module of finite rank, for if $\{\sigma_{1},\ldots,\sigma_{n}\}$ is a complete set of representatives for $\Gamma$, then the morphisms $\mathbb{C}[H] \rightarrow \mathbb{C}[H] \sigma_{i}$ of $\mathbb{C}[H]$-modules defined via $\lambda \mapsto \lambda \sigma_{i}$ are isomorphisms for all $i=1,\ldots,n$; therefore,
$$\mathbb{C}[G] = \bigoplus_{i=1}^{n}\mathbb{C}[H] \sigma_{i} \simeq (\mathbb{C}[H])^{n}. $$
Given $\lambda \in \mathbb{C}[G]$, we let 
$$m_{\lambda}:\mathbb{C}[G] \rightarrow \mathbb{C}[G]$$ 
denote the multiplication by $\lambda$ map; it is a morphism of $\mathbb{C}[H]$-modules, and the maps $T_{G/H}$ and $N_{G/H}$ from (\ref{t_and_n}) above are defined via
$$T_{G/H}(\lambda) = {\rm tr}_{\mathbb{C}[H]}(m_{\lambda}) \text{ and } N_{G/H}(\lambda) = {\rm det}_{\mathbb{C}[H]}(m_{\lambda}), $$
respectively.  It follows that $T_{G/H}$ is $\mathbb{C}[H]$-linear, and $N_{G/H}$ is multiplicative.  The two maps $T_{G/H}$ and $N_{G/H}$ can be extended to power series as follows.  The $\mathbb{C}$-algebra of power series $\mathbb{C}[G]\llbracket u \rrbracket$ is a free $\mathbb{C}[H]\llbracket u \rrbracket$-module of rank $n$ as well.  Therefore, if $P \in \mathbb{C}[G] \llbracket u \rrbracket$, then the multiplication by $P$ map 
\begin{equation} \label{mult_ps}
m_{P}:\mathbb{C}[G]\llbracket u \rrbracket \rightarrow \mathbb{C}[G]\llbracket u \rrbracket
\end{equation}
is a morphism of $\mathbb{C}[H]\llbracket u \rrbracket$-modules, and one defines
$$T_{G/H}(P) = {\rm tr}_{\mathbb{C}[H]\llbracket u \rrbracket}(m_{P}) \text{ and } N_{G/H}(P) = {\rm det}_{\mathbb{C}[H]\llbracket u \rrbracket}(m_{P}). $$
As such, we obtain two functions 
\begin{equation} \label{t_and_n_ps}
T_{G/H}, N_{G/H}:\mathbb{C}[G]\llbracket u \rrbracket \rightarrow \mathbb{C}[H]\llbracket u \rrbracket, 
\end{equation}
which are $\mathbb{C}[H]\llbracket u \rrbracket$-linear and multiplicative, respectively.  
\begin{remark} \label{base_ch_rem}
For clarity, let us denote in this remark the map in (\ref{mult_ps}) above by $\widehat{m}_{P}$.  Given $P \in \mathbb{C}[G][u]$, we also have a $\mathbb{C}[H][u]$-module morphism 
$$m_{P}:\mathbb{C}[G][u] \rightarrow \mathbb{C}[G][u].$$  
By \cref{comp_base_change}, one has
\begin{equation} \label{iso_bc}
\mathbb{C}[H]\llbracket u \rrbracket \otimes_{\mathbb{C}[H][u]}\mathbb{C}[G][u] \simeq \mathbb{C}[G]\llbracket u \rrbracket,
\end{equation}
and we leave it to the reader to check that for all $P \in \mathbb{C}[G][u]$, the diagram
\begin{equation*}
\begin{tikzcd}
\mathbb{C}[G]\llbracket u \rrbracket \arrow[,d] \arrow["\widehat{m}_{P}",r] & \mathbb{C}[G]\llbracket u \rrbracket \arrow[,d]\\
\mathbb{C}[H]\llbracket u \rrbracket \otimes_{\mathbb{C}[H][u]}\mathbb{C}[G][u] \arrow["{\rm id} \otimes m_{P}",r]  & \mathbb{C}[H]\llbracket u \rrbracket \otimes_{\mathbb{C}[H][u]}\mathbb{C}[G][u]
\end{tikzcd}
\end{equation*}
where the two vertical arrows are the isomorphism (\ref{iso_bc}) above is commutative.  It follows that
$$N_{G/H}(\mathbb{C}[G][u]) \subseteq \mathbb{C}[H][u] \text{ and } T_{G/H}(\mathbb{C}[G][u]) \subseteq \mathbb{C}[H][u].$$
\end{remark}

We will now study various properties of the two functions $T_{G/H}$ and $N_{G/H}$ that will culminate in \cref{ind_really} below.  We start with the following one.
\begin{lemma} \label{exp_and_T_N}
The following diagram
\begin{equation*} 
\begin{CD}
	u\mathbb{C}[G]\llbracket u \rrbracket  @>{{\rm exp}_{\mathbb{C}[G]}}>>   1 + u\mathbb{C}[G]\llbracket u \rrbracket  \\
    @V{T_{G/H}}VV     @VV{{N_{G/H}} }V    &  \\
	u\mathbb{C}[H]\llbracket u \rrbracket    @>{{\rm exp}_{\mathbb{C}[H]}}>>   1 + u\mathbb{C}[H]\llbracket u \rrbracket
\end{CD}
\end{equation*}
commutes.
\end{lemma}
\begin{proof}
We note first that
$$T_{G/H}(u\mathbb{C}[G] \llbracket u \rrbracket) \subseteq u \mathbb{C}[H]\llbracket u \rrbracket \text{ and } N_{G/H}(1+u\mathbb{C}[G]\llbracket u \rrbracket) \subseteq 1 + u \mathbb{C}[H]\llbracket u \rrbracket$$
so that both vertical arrows land in the correct codomains.  Indeed, if $P \in u \mathbb{C}[G]\llbracket u \rrbracket$, then the matrix of $m_{P}$ consists of power series in $u \mathbb{C}[H]\llbracket u \rrbracket$, and thus the first inclusion follows.  For the second inclusion, if $P \in 1 + u \mathbb{C}[G]\llbracket u \rrbracket$, then the elements on the diagonal of the matrix of $m_{P}$ are all in $1 + u \mathbb{C}[H]\llbracket u \rrbracket$ and the other power series off the diagonal are in $u \mathbb{C}[H]\llbracket u \rrbracket$.  Since $u \mathbb{C}[H]\llbracket u \rrbracket$ is an ideal of $\mathbb{C}[H]\llbracket u \rrbracket$, the determinant of such a matrix lands back in $1 + u\mathbb{C}[H]\llbracket u \rrbracket$, and the second inclusion follows as well.

To simplify the notation, we now let $E = {\rm End}_{\mathbb{C}[H]}(\mathbb{C}[G])$ and we let also $\beta:u\mathbb{C}[G]\llbracket u \rrbracket \rightarrow u{\rm End}_{\mathbb{C}[H]\llbracket u \rrbracket}(\mathbb{C}[G]\llbracket u \rrbracket)$ be the map $P \mapsto m_{P}$.  Note first that the diagram
\begin{equation*}
\begin{tikzcd}
u\mathbb{C}[G]\llbracket u \rrbracket \arrow["{\rm exp}_{\mathbb{C}[G]}",d] \arrow["\beta",r] & u{\rm End}_{\mathbb{C}[H]\llbracket u \rrbracket}(\mathbb{C}[G]\llbracket u \rrbracket)  \arrow["\omega_{\mathbb{C}[G]}^{-1}",r] & uE\llbracket u \rrbracket \arrow["{\rm exp}_{E}",d]\\
1+u\mathbb{C}[G]\llbracket u \rrbracket\arrow["\beta",r] & 1+u{\rm End}_{\mathbb{C}[H]\llbracket u \rrbracket}(\mathbb{C}[G]\llbracket u \rrbracket) \arrow["\omega_{\mathbb{C}[G]}^{-1}",r] & 1+uE\llbracket u \rrbracket
\end{tikzcd},
\end{equation*}
commutes by \cref{basic_prop} after noting that $\omega_{\mathbb{C}[G]}^{-1} \circ \beta$ is a morphism of $\mathbb{Q}\llbracket u \rrbracket$-algebras.  The result would then follow if we could show that the following diagram
\begin{equation*}
\begin{tikzcd}
uE\llbracket u \rrbracket \arrow["{\rm exp}_{E}",d] \arrow["\omega_{\mathbb{C}[G]}",r] & u{\rm End}_{\mathbb{C}[H]\llbracket u \rrbracket}(\mathbb{C}[G]\llbracket u \rrbracket) \arrow[,d] \arrow["{\rm tr}_{\mathbb{C}[H]\llbracket u \rrbracket}",r] & u\mathbb{C}[H] \llbracket u \rrbracket \arrow["{\rm exp}_{\mathbb{C}[H]}",d]\\
1+uE\llbracket u \rrbracket\arrow["\omega_{\mathbb{C}[G]}",r] & 1+u{\rm End}_{\mathbb{C}[H]\llbracket u \rrbracket}(\mathbb{C}[G]\llbracket u \rrbracket) \arrow["{\rm det}_{\mathbb{C}[H]\llbracket u \rrbracket}",r] & 1+ u\mathbb{C}[H] \llbracket u \rrbracket
\end{tikzcd}
\end{equation*}
commutes, where the middle vertical arrow is induced by the left vertical arrow.  But after picking an ordered $\mathbb{C}[H]\llbracket u \rrbracket$-basis for $\mathbb{C}[G]\llbracket u \rrbracket$, the right square becomes
\begin{equation*}
\begin{tikzcd}
uM_{n}(\mathbb{C}[H]\llbracket u \rrbracket) \arrow[,d] \arrow["{\rm tr}_{\mathbb{C}[H]\llbracket u \rrbracket}",r] & u\mathbb{C}[H] \llbracket u \rrbracket \arrow["{\rm exp}_{\mathbb{C}[H]}",d]\\
I+uM_{n}(\mathbb{C}[H]\llbracket u \rrbracket) \arrow["{\rm det}_{\mathbb{C}[H]\llbracket u \rrbracket}",r] & 1+ u\mathbb{C}[H] \llbracket u \rrbracket
\end{tikzcd}
\end{equation*}
where the left vertical arrow is the usual exponential of matrices.  One concludes the proof by appealing to \cref{det_tr_exp}.
\end{proof}
The next property we study is the relationship between the function $T_{G/H}$ and the equivariant Euler characteristic.
\begin{lemma} \label{eul_and_T}
With the notation as above, one has
$$T_{G/H}(\chi_{\mathbb{C}[G]}(Y)) = \chi_{\mathbb{C}[H]}(Y). $$
\end{lemma}
\begin{proof}
Let $H,K \le G$.  We start by showing that
\begin{equation} \label{to_pro}
\frac{1}{(G:K)}T_{G/H}(e_{K}) = \frac{1}{(H:H \cap K)} e_{H \cap K}. 
\end{equation}
For that purpose, let $\{\sigma_{1},\ldots, \sigma_{n} \}$ be a complete set of representatives of $G/H$.  For $j=1,\ldots,n$, we have
\begin{equation*} 
\begin{aligned}
e_{K} \cdot \sigma_{j} &= \frac{1}{|K|} \sum_{k \in K} k \cdot \sigma_{j} \\
&= \frac{1}{|K|} \sum_{i=1}^{n} \lambda_{ij}(K) \cdot \sigma_{i},
\end{aligned}
\end{equation*}
where
$$\lambda_{ij}(K) = \sum_{x \in K\sigma_{j} \cap H \sigma_{i}} h(x), $$
with the convention that $h(x)$ is the unique element in $H$ satisfying $x = h(x)\sigma_{i}$.  Then, we have
\begin{equation*}
\begin{aligned}
T_{G/H}(e_{K}) &= \frac{1}{|K|} \sum_{i=1}^{n} \lambda_{ii}(K) \\
&= \frac{1}{|K|} \sum_{i=1}^{n} \sum_{x \in K\sigma_{i} \cap H\sigma_{i}} h(x) \\
&= \frac{1}{|K|} \sum_{i=1}^{n} \sum_{x \in (K\cap H )\sigma_{i}}h(x) \\
&= \frac{1}{|K|} \sum_{i=1}^{n} \sum_{x \in K \cap H} x \\
&= \frac{(G:H)}{|K|} |K \cap H| e_{K \cap H}.
\end{aligned}
\end{equation*}
Using Lagrange's theorem, a simple calculation shows that
$$\frac{(G:H)}{(G:K)}\frac{|K \cap H|}{|K|} = \frac{1}{(H:K\cap H)} $$
from which (\ref{to_pro}) above follows.  Thus for each $w \in V_{Y}$, one has
$$T_{G/H}(e_{G_{w}}) = \frac{(G:G_{w})}{(H:H_{w})} e_{H_{w}}, $$
since $H_{w} = G_{w} \cap H$.  The quotient group $\Gamma = G/H$ acts on $Y_{H}$, and one has a short exact sequence
\begin{equation}
0 \rightarrow H \cdot G_{w}/G_{w} \rightarrow G/G_{w} \rightarrow \Gamma/\Gamma_{Hw} \rightarrow 0
\end{equation}
of groups, where $\Gamma_{Hw} = {\rm Stab}_{\Gamma}(Hw)$.  Since $H/H_{w} \simeq H \cdot G_{w}/G_{w}$, one has
$$\frac{(G:G_{w})}{(H:H_{w})} = (\Gamma:\Gamma_{Hw}). $$
Now, since $(\Gamma:\Gamma_{Hw})$ is precisely the number of vertices of $Y_{H}$ that maps to the vertex $Gw \in Y_{G}$ via the natural graph morphism $Y_{H} \rightarrow Y_{G}$, one obtains the equality
$$\sum_{v \in V_{Y_G}}T_{G/H}(e_{G_{v}}) = \sum_{w \in V_{Y_{H}}}e_{H_{w}}. $$
A similar argument works for edges as well, and thus the result follows from \cref{exp_equi_euler_char}. 
\end{proof}
Next, we look at the relationship between $N_{G/H}$ and the determinant in the following lemma.
\begin{lemma} \label{det_and_N}
Let $M$ be a finitely generated $\mathbb{C}[G]$-module.  Then the diagram
\begin{equation*} 
\begin{CD}
	{\rm End}_{\mathbb{C}[G][u]}(M[u])  @>>> {\rm End}_{\mathbb{C}[H][u]}(M[u])     \\
    @V{{\rm det}_{\mathbb{C}[G][u]}}VV     @VV{ {\rm det}_{\mathbb{C}[H][u]}}V    &  \\
	\mathbb{C}[G][u]    @>{N_{G/H}}>>   \mathbb{C}[H][u]
\end{CD}
\end{equation*}
commutes, where the first horizontal arrow is obtained by viewing a $\mathbb{C}[G][u]$-module operator on $M[u]$ as a $\mathbb{C}[H][u]$-module operator which is possible, since $H \le G$.  
\end{lemma}
\begin{proof}
It suffices to show the theorem for when $M$ is free.  So let us assume that $F$ is a free $\mathbb{C}[G]$-module of finite rank.  We fix an ordered $\mathbb{C}[G]$-basis $(f_{1},\ldots,f_{m})$ of $F$, and also an ordered $\mathbb{C}[H]$-basis $(\sigma_{1},\ldots,\sigma_{n})$ of $\mathbb{C}[G]$.  Note that $(f_{1},\ldots,f_{m})$ is also an ordered $\mathbb{C}[G][u]$-basis of $F[u]$ and $(\sigma_{1},\ldots,\sigma_{n})$ an ordered $\mathbb{C}[H][u]$-basis of $\mathbb{C}[G][u]$.  Consider now the ring morphism
$$\rho: \mathbb{C}[G][u] \rightarrow M_{n}(\mathbb{C}[H][u]) $$
given by $P \mapsto \rho(P) = [m_{P}]$, where $[m_{P}]$ is the matrix representing the $\mathbb{C}[H][u]$-operator $m_{P}$ with respect to the basis $(\sigma_{1},\ldots,\sigma_{n})$.  \cref{silvester} gives a commutative diagram
\begin{equation*}
\begin{tikzcd}
M_{m}(\mathbb{C}[G][u]) \arrow[,d] \arrow["{\rm det}_{\mathbb{C}[G][u]}",r] & \mathbb{C}[G][u]  \arrow["\rho",r] & M_{n}(\mathbb{C}[H][u]) \arrow["{\rm det}_{\mathbb{C}[H][u]}",d]\\
M_{m}(M_{n}(\mathbb{C}[H][u]))  \arrow["\simeq",r] & M_{mn}(\mathbb{C}[H][u]) \arrow["{\rm det}_{\mathbb{C}[H][u]}",r] & \mathbb{C}[H][u]
\end{tikzcd}
\end{equation*}
Moreover, we also have another commutative diagram
\begin{equation*}
\begin{tikzcd}
{\rm End}_{\mathbb{C}[G][u]}(F[u]) \arrow["\simeq",r] \arrow[,d] & M_{m}(\mathbb{C}[G][u]) \arrow[,d]  \\
{\rm End}_{\mathbb{C}[H][u]}(F[u]) \arrow["\simeq",r] & M_{m}(M_{n}(\mathbb{C}[H][u])),
\end{tikzcd}
\end{equation*}
where the first horizontal arrow is the one obtained by writing the matrix of a $\mathbb{C}[G][u]$-operator on $F[u]$ using the ordered $\mathbb{C}[G][u]$-basis $(f_{1},\ldots,f_{m})$ for $F[u]$, and the bottom horizontal arrow is the one obtained by writing the matrix of a $\mathbb{C}[H][u]$-operator on $F[u]$ using the ordered $\mathbb{C}[H][u]$-basis $(\sigma_{1}f_{1},\ldots,\sigma_{n}f_{1},\sigma_{1}f_{2},\ldots)$ of $F[u]$.  Putting these two commutative diagrams together gives the desired result.
\end{proof}

We are now in a position to formulate our main result of this section which is the equivariant way of thinking about the induction property of the Artin formalism.  (Compare with the second statement of \cite[Proposition 1.8 on page 87]{Tate:1984}.)  Since we need to keep track of the group that is acting on our graph $Y$, we shall write $\gamma_{(Y,G)}(u)$ instead of $\gamma_{Y}(u)$, and similarly for $\eta_{Y}(u)$, and $\theta_{Y}(u)$.
\begin{theorem} \label{ind_really}
One has
$$N_{G/H}(\gamma_{(Y,G)}(u)) = \gamma_{(Y,H)}(u) \text{ and } N_{G/H}(\eta_{(Y,G)}(u)) = \eta_{(Y,H)}(u), $$
and thus, the following equality $N_{G/H}(\theta_{(Y,G)}(u)) = \theta_{(Y,H)}(u)$ also holds true.
\end{theorem}
\begin{proof}
The first identity follows from \cref{exp_and_T_N,eul_and_T}, and the observation that
\begin{equation*}
\begin{aligned} 
T_{G/H}\left(-\chi_{\mathbb{C}[G]}(Y) \cdot {\rm log}_{\mathbb{C}[G]}(1-u^{2}) \right) &= T_{G/H}\left(-\chi_{\mathbb{C}[G]}(Y) \cdot {\rm log}_{\mathbb{Q}}(1-u^{2})\right)\\ 
&= T_{G/H}\left(-\chi_{\mathbb{C}[G]}(Y)\right) \cdot {\rm log}_{\mathbb{Q}}(1-u^{2}) \\
&= -\chi_{\mathbb{C}[H]}(Y) \cdot{\rm log}_{\mathbb{C}[H]}(1-u^{2}).
\end{aligned}
\end{equation*}
The second identity follows from \cref{det_and_N}, and the third one from the multiplicativity property of $N_{G/H}$.
\end{proof}
\begin{example} \label{ex222}
Let us go back to \cref{ex11}, where we take $H = \langle \bar{2} \rangle = G_{2}$.  The graph $Y_{H}$ has four vertices, and the branched cover $Y \rightarrow Y_{H}$ can be visualized as follows:
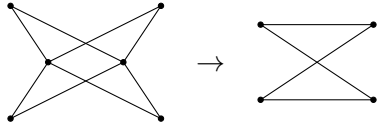
\begin{equation*}  
\begin{tikzpicture}[baseline={([yshift=-0.7ex] current bounding box.center)}]
%vertices
\draw[fill=black] (-1,3/4) circle (1pt);
\draw[fill=black] (-1,-3/4) circle (1pt);
\draw[fill=black] (1,3/4) circle (1pt);
\draw[fill=black] (1,-3/4) circle (1pt);
\draw[fill=black] (-1/2,0) circle (1pt);
\draw[fill=black] (1/2,0) circle (1pt);
%edges
\path (-1,3/4) edge (-1/2,0);
\path (-1,3/4) edge (1/2,0);
\path (1,3/4) edge (-1/2,0);
\path (1,3/4) edge (1/2,0);
\path (1,-3/4) edge (-1/2,0);
\path (1,-3/4) edge (1/2,0);
\path (-1,-3/4) edge (-1/2,0);
\path (-1,-3/4) edge (1/2,0);
\end{tikzpicture}
\, \, \, \, \, \rightarrow \, \, \, \, \, 
\begin{tikzpicture}[baseline={([yshift=-0.7ex] current bounding box.center)}]
%vertices
\draw[fill=black] (-3/4,1/2) circle (1pt);
\draw[fill=black] (-3/4,-1/2) circle (1pt);
\draw[fill=black] (3/4,1/2) circle (1pt);
\draw[fill=black] (3/4,-1/2) circle (1pt);
%edges
\path (-3/4,1/2) edge (3/4,1/2);
\path (-3/4,-1/2) edge (3/4,-1/2);
\path (-3/4,1/2) edge (3/4,-1/2);
\path (-3/4,-1/2) edge (3/4,1/2);
\end{tikzpicture} 
\end{equation*}
We pick
$$w_{1} = (v_{1},\bar{0}), w_{2} = (v_{1},\bar{1}), w_{3} = (v_{2},H), \text{ and } w_{4} = (v_{2}, \bar{1} + H) \in V_{Y}, $$
and with this choice, we calculate
\begin{equation*}
\mathbf{A} = \begin{pmatrix} 0 & 0 & N_{H} & N_{H} \\ 0 & 0 & N_{H} & N_{H} \\ e_{H} & e_{H} & 0 & 0 \\ e_{H} & e_{H} & 0 & 0 \end{pmatrix}
\text{ and }
\mathbf{Q} = \begin{pmatrix} 1 & 0 & 0 &0 \\ 0 & 1 & 0 &0\\ 0 & 0 & 3e_{H} &0\\0 & 0 & 0 & 3e_{H} \end{pmatrix}
\end{equation*}
so that
\begin{equation} \label{equiv_H}
\begin{aligned}
\eta_{(Y,H)}(u) &= {\rm det}(\mathbf{I} - \mathbf{A}u + \mathbf{Q}u^{2}) \\
&= 1 + (\bar{0}-\bar{2})u^{2} - \frac{1}{2}(9\cdot \bar{0} + 11\cdot \bar{2})u^{4} + 9e_{H}u^{8} \in \mathbb{C}[H][u]. 
\end{aligned}
\end{equation}
On the other hand, let us pick the coset representatives $\{\bar{0},\bar{1}\}$ for $\Gamma = G/H$.  With respect to the ordered $\mathbb{C}[H][u]$-basis $(\bar{0},\bar{1})$ of $\mathbb{C}[G][u]$, one calculates the matrix representing $m_{\eta_{(Y,G)}(u)}$ to be
\begin{equation*}
\begin{pmatrix}
1 + \frac{1}{2}(\bar{0} - \bar{2})u^{2} + 3 e_{H}u^{4} & -2 N_{H}u^{2} \\
-2 N_{H} u^{2} & 1 + \frac{1}{2}(\bar{0} - \bar{2})u^{2} + 3 e_{H}u^{4}
\end{pmatrix} \in M_{2}(\mathbb{C}[H][u]),
\end{equation*}
and the calculation of its determinant leads to
$$N_{G/H}(\eta_{(Y,G)}(u)) = 1 + (\bar{0}-\bar{2})u^{2} - \frac{1}{2}(9\cdot \bar{0} + 11\cdot \bar{2})u^{4} + 9e_{H}u^{8} \in \mathbb{C}[H][u], $$
which coincides with (\ref{equiv_H}) as expected by \cref{ind_really}.  We will show in \cref{ex22} below that
\begin{equation*} 
\begin{aligned}
\gamma_{(Y,G)}(u) &= e_{\psi_{0}} + e_{\psi_{1}} + (1-u^{2})e_{\psi_{2}} + (1-u^{2})e_{\psi_{3}} \\
&= 1 + \frac{1}{2}(-\bar{0} + \bar{2})u^{2} \in \mathbb{C}[G][u],
\end{aligned}
\end{equation*}
where $\psi_{0}$ is the trivial character of $G$, $\psi_{1}$ the unique character of order $2$ of $G$, and $\psi_{2}, \psi_{3}$ the two characters of order $4$ of $G$.  A similar calculation gives 
\begin{equation} \label{ga_H}
\gamma_{(Y,H)}(u) = 1 + (-\bar{0} + \bar{2})u^{2} + \frac{1}{2}(\bar{0}-\bar{2})u^{4} \in \mathbb{C}[H][u]. 
\end{equation}
With respect to the ordered $\mathbb{C}[H][u]$-basis $(\bar{0},\bar{1})$ of $\mathbb{C}[G][u]$, one calculates the matrix representing $m_{\gamma_{(Y,G)}(u)}$ to be
\begin{equation*}
\begin{pmatrix}
1 + \frac{1}{2}(-\bar{0} + \bar{2})u^{2} & 0 \\
0 & 1 + \frac{1}{2}(-\bar{0} + \bar{2})u^{2}
\end{pmatrix} \in M_{2}(\mathbb{C}[H][u]),
\end{equation*}
whose determinant agrees with (\ref{ga_H}) as expected by \cref{ind_really}.
\demo
\end{example}

The $\mathbb{C}$-algebra morphism $\pi_{H}:\mathbb{C}[G] \rightarrow \mathbb{C}[\Gamma]$ can be extended to another $\mathbb{C}$-algebra morphism
$$\pi_{H}:\mathbb{C}[G][u] \rightarrow \mathbb{C}[\Gamma][u]$$
by applying $\pi_{H}$ to the coefficients of polynomials in $\mathbb{C}[G][u]$.  Moreover, if $G$ acts without inversion on $Y$, then so does $\Gamma$ on $Y_{H}$, and therefore one can consider
$$\eta_{(Y_{H},\Gamma)}(u) \in \mathbb{C}[\Gamma] [u]. $$
In general, 
\begin{equation} \label{failure_inf}
\pi_{H}(\eta_{(Y,G)}(u)) \neq \eta_{(Y_{H},\Gamma)}(u).
\end{equation}
Indeed, going back to \cref{ex11,ex22}, using the identification (\ref{identification}) and noting that $\pi_{H}(G_{i})$ is the trivial subgroup of $\Gamma$ for all $i = 1,2$, one has
$$X(\Gamma,\pi_{H}(\mathcal{G}),\pi_{H} \circ \alpha) = X(\Gamma,\pi_{H} \circ \alpha), $$
and we can identify $Y_{H}$ with $X(\Gamma, \pi_{H} \circ \alpha)$.  Using this identification and choosing $w_{i} = (v_{i},1_{\Gamma}) \in V_{X(\Gamma, \pi_{H} \circ \alpha)}$ for $i=1,2$, we calculate
\begin{equation*}
\mathbf{A} = \begin{pmatrix} 0 & {N_{\Gamma}} \\ N_{\Gamma} & 0 \end{pmatrix}
\text{ and }
\mathbf{Q} = \begin{pmatrix} 1 & 0 \\ 0&1 \end{pmatrix}
\end{equation*}
so that
$$\eta_{(Y_{H},\Gamma)}(u) = 1 - 2N_{\Gamma}u^{2} \in \mathbb{C}[\Gamma][u]. $$
On the other hand, we have
\begin{equation*}
\begin{aligned} 
\pi_{H} \left(\eta_{(Y,G)}(u) \right) &= \pi_{H}\left( 1 + \frac{1}{2}(\bar{0} - 4\cdot \bar{1} - \bar{2} - 4\cdot\bar{3})u^{2} + \frac{3}{2}(\bar{0} + \bar{2})u^{4}\right) \\
&= 1 - 4 au^{2}+3 u^{4} \\
&\neq \eta_{(Y_{H},\Gamma)}(u),
\end{aligned}
\end{equation*}
where we let $a$ be the unique non-trivial element of $\Gamma$.  The non-equality (\ref{failure_inf}) shows that the inflation property, expressed in an equivariant way, of the Artin formalism is not always satisfied in this situation.  (Compare with the first statement of \cite[Proposition 1.8 on page 87]{Tate:1984}.)

%\section{Artin representations}
\section{Artin representations} \label{section:art}
%\subsection{The Euler characteristic}
\subsection{The Euler characteristic} \label{section:Eul}
As usual, we let $G$ be a finite abelian group acting on a finite graph $Y$ without inversion.  Here, we view an Artin representation $V$ as a $(\mathbb{C},\mathbb{C}[G])$-bimodule with the right multiplication given by $v\lambda := \lambda v$.  Given an Artin representation $V$ of $G$, then tensoring the chain complex (\ref{chain_complex}) with $V$ over $\mathbb{C}[G]$ gives another chain complex
$$0 \longrightarrow V \otimes_{\mathbb{C}[G]}C_{1}(Y) \stackrel{\partial_{V}}{\longrightarrow} V \otimes_{\mathbb{C}[G]}C_{0}(Y) \longrightarrow 0$$
of $\mathbb{C}$-vector spaces, where $\partial_{V} = {\rm id}_{V} \otimes \partial$.  From now on, if $M$ is a finitely generated $\mathbb{C}[G]$-module and $V$ is an Artin representation, then we let
$$T_{V}(M) = V \otimes_{\mathbb{C}[G]}M. $$
The map $M \mapsto T_{V}(M)$ is a covariant functor from the category of finitely generated $\mathbb{C}[G]$-modules to the category of finite dimensional $\mathbb{C}$-vector spaces.  If $f:M \rightarrow N$ is a morphism of finitely generated $\mathbb{C}[G]$-modules, then we write $f_{V}$ instead of $T_{V}(f) = {\rm id}_{V} \otimes f$.
\begin{definition} \label{artin_euler}
Let $Y$ be a finite graph on which a finite abelian group $G$ acts without inversion, and let $V$ be an Artin representation of $G$.  One defines
$$\chi_{V}(Y) = {\rm dim}_{\mathbb{C}}(T_{V}(C_{0}(Y))) - {\rm dim}_{\mathbb{C}} (T_{V}(C_{1}(Y))). $$
\end{definition}
Clearly, $\chi_{V}(Y)$ depends only on the isomorphism class of $V$, and therefore we will often write $\chi_{\psi}(Y)$ instead of $\chi_{V}(Y)$ if $\psi$ is the character of $V$, since two Artin representations have the same character if and only if they are isomorphic.  From now on, if $\psi \in \widehat{G}$, we write 
$$r_{0}(\psi) = |\{v \in V_{X}: G_{v} \not \subseteq {\rm ker}(\psi) \}| \text{ and } r_{1}(\psi) = |\{s \in O_{X}: G_{s} \not\subseteq {\rm ker}(\psi) \}|,$$
where $O_{X}$ is an orientation of $X = Y_{G}$.
\begin{theorem} \label{art_eul_char}
Let $Y$ be a finite graph on which a finite abelian group $G$ acts without inversion, and let $V$ be an Artin representation of $G$.  Let also $X = Y_{G}$.  One has
$$\chi_{V}(Y) = \psi_{V}(\chi_{\mathbb{C}[G]}(Y)). $$
Moreover, if $V$ is irreducible with character $\psi$, then
$$\chi_{\psi}(Y) = \chi(X) - r_{0}(\psi) + r_{1}(\psi), $$
so that in particular, if $G$ acts freely on $\mathbf{E}_{Y}$, then $\chi_{\psi}(Y) = \chi(X) - r_{0}(\psi)$.
\end{theorem}
\begin{proof}
Since both $\chi_{V}$ and $\psi_{V}$ are additive in $V$, and since any Artin representation is a direct sum of irreducible ones, it suffices to show the claim for an irreducible Artin representation $V$.  Assume therefore that $V$ is irreducible, and let $\psi = \psi_{V}$ be its character.  On one hand, by \cref{artin_euler} we have 
$$\chi_{V}(Y) = {\rm dim}_{\mathbb{C}} (e_{\psi} C_{0}(Y)) - {\rm dim}_{\mathbb{C}} (e_{\psi} C_{1}(Y)). $$
Using \cref{dim_iso_component}, one has
\begin{equation*}
\begin{aligned}
{\rm dim}_{\mathbb{C}}(e_{\psi}C_{1}(Y)) &= {\rm dim}_{\mathbb{C}}(e_{\psi}\mathbb{C}\mathbf{E}_{Y}) - {\rm dim}_{\mathbb{C}}(e_{\psi}W) \\
&= |\{s \in \mathbf{E}_{X}: G_{s} \subseteq {\rm ker}(\psi)\}| -|\{s \in O_{X}: G_{s} \subseteq {\rm ker}(\psi)\}| \\
&= |E_{X}| - r_{1}(\psi),
\end{aligned}
\end{equation*}
where $O_{X}$ is an orientation of $X$ as usual, and similarly
$${\rm dim}_{\mathbb{C}}(e_{\psi}C_{0}(Y)) = |V_{X}| - r_{0}(\psi), $$
so that
$$\chi_{V}(Y) = |V_{X}| - r_{0}(\psi) - (|E_{X}| - r_{1}(\psi)) = \chi(X) - r_{0}(\psi) + r_{1}(\psi). $$ 
On the other hand, \cref{exp_equi_euler_char} gives
$$\psi(\chi_{\mathbb{C}[G]}(Y)) = |V_{X}| - r_{0}(\psi) - (|E_{X}| - r_{1}(\psi)) = \chi(X) - r_{0}(\psi) + r_{1}(\psi) $$
which concludes the proof.
\end{proof}

%\subsection{The function \texorpdfstring{$\sigma_{V}$}{}}
\subsection{The function \texorpdfstring{$\sigma_{V}$}{}} \label{the_function_sigma}
Throughout this section, $G$ is a finite abelian group acting on a finite graph $Y$ without inversion.  We let
$${\rm \sigma}_{V}:\mathbb{C}[G]\llbracket u \rrbracket \rightarrow \mathbb{C}\llbracket u \rrbracket$$
be the function which is obtained from the composition
\begin{equation}
\mathbb{C}[G] \llbracket u \rrbracket \stackrel{\rho_{V}}{\longrightarrow} {\rm End}_{\mathbb{C}}(V) \llbracket u \rrbracket \stackrel{\omega_{V}}{\longrightarrow} {\rm End}_{\mathbb{C}\llbracket u \rrbracket}(V \llbracket u \rrbracket) \stackrel{{\rm det}_{\mathbb{C}\llbracket u \rrbracket}}{\longrightarrow} \mathbb{C}\llbracket u \rrbracket, 
\end{equation}
where the first map $\rho_{V}$ is the $\mathbb{C}\llbracket u \rrbracket$-algebra morphism induced by the $\mathbb{C}$-algebra morphism
$$\rho_{V}:\mathbb{C}[G] \rightarrow {\rm End}_{\mathbb{C}}(V) $$
associated to the Artin representation $V$, and the second map $\omega_{V}$ is the isomorphism of $\mathbb{C}\llbracket u \rrbracket$-algebras of \cref{conv_iso}.  The function $\sigma_{V}$ is not a $\mathbb{C}\llbracket u\rrbracket$-algebra morphism in general, but is only multiplicative meaning that 
$${\rm \sigma}_{V}(P \cdot Q) = {\rm \sigma}_{V}(P) \cdot {\rm \sigma}_{V}(Q) $$ 
whenever $P,Q \in \mathbb{C}[G]\llbracket u \rrbracket$.  We will use the function $\sigma_{V}$ in \cref{iharalfun} below to define Ihara $L$-functions associated to Artin representations.  

\begin{remark}
An argument similar to the one used in \cref{base_ch_rem} shows that 
$$\sigma_{V}(\mathbb{C}[G][u]) \subseteq \mathbb{C}[u]. $$
\end{remark}

In this section, we study some properties satisfied by $\sigma_{V}$ that we will need later.
\begin{proposition} \label{det_up_to_iso}
If $V_{1}$ and $V_{2}$ are two isomorphic Artin representations, then one has
$${\rm \sigma}_{V_{1}} = {\rm \sigma}_{V_{2}}. $$
\end{proposition}
\begin{proof}
Since $V_{1} \simeq V_{2}$, there exists an isomorphism $f:V_{1} \rightarrow V_{2}$ of $\mathbb{C}[G]$-modules.  The map $f$ induces an isomorphism of $\mathbb{C}$-algebras 
$$f_{*}:{\rm End}_{\mathbb{C}}(V_{1}) \rightarrow {\rm End}_{\mathbb{C}}(V_{2}) $$
given by $\alpha \mapsto f_{*}(\alpha) = f \circ \alpha \circ f^{-1}$.  Therefore, we get a natural morphism of $\mathbb{C}\llbracket u \rrbracket$-algebras 
$${\rm End}_{\mathbb{C}}(V_{1})\llbracket u \rrbracket \rightarrow {\rm End}_{\mathbb{C}}(V_{2}) \llbracket u \rrbracket $$
which we denote by the same symbol $f_{*}$.  The map $f$ also induces an isomorphism 
$$\tilde{f}:V_{1}\llbracket u \rrbracket \rightarrow V_{2} \llbracket u \rrbracket$$ 
of $\mathbb{C}\llbracket u \rrbracket$-modules which, in turn, induces an isomorphism
$$\tilde{f}_{*}:{\rm End}_{\mathbb{C}\llbracket u \rrbracket}(V_{1}\llbracket u \rrbracket) \rightarrow {\rm End}_{\mathbb{C}\llbracket u \rrbracket}(V_{2}\llbracket u \rrbracket) $$
of $\mathbb{C}\llbracket u \rrbracket$-algebras given by $\beta \mapsto \tilde{f}_{*}(\beta) = \tilde{f} \circ \beta \circ \tilde{f}^{-1}$.  The result then follows from the commutativity of the diagram
\begin{equation*}
\begin{tikzcd}
\mathbb{C}[G]\llbracket u \rrbracket \arrow["{\rm id}_{\mathbb{C}[G]\llbracket u \rrbracket}",d] \arrow["\rho_{V_{1}}",r] & {\rm End}_{\mathbb{C}}(V_{1})\llbracket u \rrbracket \arrow["f_{*}",d] \arrow["\omega_{V_{1}}",r] & {\rm End}_{\mathbb{C}\llbracket u \rrbracket}(V_{1}\llbracket u \rrbracket) \arrow["\tilde{f}_{*}",d]\\
\mathbb{C}[G]\llbracket u \rrbracket \arrow["\rho_{V_{2}}",r] & {\rm End}_{\mathbb{C}}(V_{2})\llbracket u \rrbracket \arrow["\omega_{V_{2}}",r] & {\rm End}_{\mathbb{C}\llbracket u \rrbracket}(V_{2}\llbracket u \rrbracket)
\end{tikzcd}
\end{equation*}
which we leave to the reader to check, and the observation that 
$${\rm det}_{\mathbb{C}\llbracket u \rrbracket} \circ \tilde{f}_{*} (\beta) = {\rm det}_{\mathbb{C}\llbracket u \rrbracket}(\beta) $$
for all $\beta \in {\rm End}_{\mathbb{C}\llbracket u \rrbracket}(V_{1}\llbracket u \rrbracket)$.
\end{proof}

Although ${\rm \sigma}_{V}$ is only multiplicative in general, it is in fact a morphism of $\mathbb{C}\llbracket u \rrbracket$-algebras when $V$ is irreducible as the following proposition shows.
\begin{proposition} \label{sigma_irr}
If $V$ is an irreducible Artin representation, then ${\sigma}_{V}$ is a $\mathbb{C}\llbracket u \rrbracket$-algebra morphism and is simply the $\mathbb{C}\llbracket u \rrbracket$-algebra morphism $\psi:\mathbb{C}[G]\llbracket u \rrbracket \rightarrow \mathbb{C}\llbracket u \rrbracket$, obtained from the $\mathbb{C}$-algebra morphism $\psi:\mathbb{C}[G] \rightarrow \mathbb{C}$ applied to the coefficients of a power series in $\mathbb{C}[G]\llbracket u \rrbracket$, where $\psi$ is the character of $V$.
\end{proposition}
\begin{proof}
Since $G$ is finite abelian, $V$ is necessarily one dimensional over $\mathbb{C}$.  The claim then follows from the commutativity of the diagram
\begin{equation*}
\begin{tikzcd}
\mathbb{C}[G]\llbracket u \rrbracket \arrow[d] \arrow["\rho_{V}",r] & {\rm End}_{\mathbb{C}}(V) \llbracket u \rrbracket \arrow[d] \arrow["\omega_{V}",r] & {\rm End}_{\mathbb{C}\llbracket u \rrbracket}(V\llbracket u \rrbracket) \arrow[d] \arrow["{\rm det}_{\mathbb{C}\llbracket u \rrbracket}",r] & \mathbb{C}\llbracket u \rrbracket \arrow[d]\\
\mathbb{C}[G]\llbracket u \rrbracket \arrow["\psi",r] & \mathbb{C}\llbracket u \rrbracket \arrow["{\rm id}_{\mathbb{C}\llbracket u \rrbracket}",r] & \mathbb{C}\llbracket u \rrbracket \arrow["{\rm id}_{\mathbb{C}\llbracket u \rrbracket}",r] & \mathbb{C}\llbracket u \rrbracket.
\end{tikzcd}
\end{equation*}
for which we now explain the vertical morphisms involved.  The first one on the left is the identity, the second one is the isomorphism of $\mathbb{C}\llbracket u \rrbracket$-algebras induced from the $\mathbb{C}$-algebra isomorphism ${\rm End}_{\mathbb{C}}(V) \stackrel{\simeq}{\longrightarrow} \mathbb{C}$ coming from (\ref{simple_but_useful}), the third one is the isomorphism coming from (\ref{simple_but_useful}) applied to the free $\mathbb{C}\llbracket u \rrbracket$-module $V\llbracket u \rrbracket$ of rank one, and the fourth one is the identity.  We leave it to the reader to check that the three squares in the diagram above commute, from which the result follows.
\end{proof}

\begin{proposition} \label{mult_det}
Let $V_{1}$ and $V_{2}$ be two Artin representations of $G$ and let $V = V_{1} \oplus V_{2}$.  Then
$${\rm \sigma}_{V} = {\rm \sigma}_{V_{1}} \cdot {\rm \sigma}_{V_{2}}. $$
\end{proposition}
\begin{proof}
This follows directly from (\ref{det_prop}) and the observation that
$$\omega_{V}(\rho_{V}(P)) = \omega_{V_{1}}(\rho_{V_{1}}(P)) \oplus \omega_{V_{2}}(\rho_{V_{2}}(P)) $$
for all $P \in \mathbb{C}[G]\llbracket u \rrbracket$.
\end{proof}

In \cref{af}, we will study the Artin formalism for Ihara $L$-functions.  The following theorem will be useful to study the induction property (see \cref{formalism_two} below). 
\begin{theorem} \label{useful_for_induction}
Let $H \le G$, $V$ an Artin representation of $H$, and let $W = {\rm Ind}(V)$.  Then, the following diagram
\begin{equation*} 
\begin{tikzcd}
\mathbb{C}[G]\llbracket u \rrbracket \arrow["\sigma_{W}"',rd] \arrow["N_{G/H}",r] & \mathbb{C}[H]\llbracket u \rrbracket \arrow["\sigma_{V}",d] \\
 & \mathbb{C}\llbracket u \rrbracket  
\end{tikzcd}
\end{equation*}
commutes.
\end{theorem}
\begin{proof}
By \cref{mult_det}, and since induction is additive, meaning ${\rm Ind}(V_{1}\oplus V_{2}) \simeq {\rm Ind}(V_{1}) \oplus {\rm Ind}(V_{2})$, it suffices to show the claim when $V$ is an irreducible Artin representation of $H$.  Let $\psi$ be the character of such an irreducible Artin representation $V$ of $H$.  Then, a simple calculation using the orthogonality relation (\ref{ortho}) gives
\begin{equation} \label{dir_sum}
e_{\psi} = \sum_{\substack{\varphi \in \widehat{G} \\ \varphi \vert_{H} = \psi}} e_{\varphi}.
\end{equation}
It follows that one has
$$\mathbb{C}[G]\llbracket u \rrbracket e_{\psi} =\bigoplus_{\substack{\varphi \in \widehat{G}\\ \varphi|_{H} = \psi}}\mathbb{C}[G]\llbracket u \rrbracket e_{\varphi}. $$
Let $P \in \mathbb{C}[G]\llbracket u \rrbracket$, then $m_{P}$ induces by restriction a $\mathbb{C}$-linear morphism 
$$m_{P}^{\psi}:\mathbb{C}[G]\llbracket u \rrbracket e_{\psi} \rightarrow \mathbb{C}[G]\llbracket u \rrbracket e_{\psi},$$
which by (\ref{dir_sum}) above satisfies 
$$m_{P}^{\psi} = \bigoplus_{\substack{\varphi \in \widehat{G}\\ \varphi|_{H} = \psi}} m_{P}^{\varphi}, $$
where $m_{P}^{\varphi}$ denotes the $\mathbb{C}$-linear morphism obtained by restricting $m_{P}$ to $\mathbb{C}[G]\llbracket u \rrbracket e_{\varphi}$.  Noting that $m_{P}^{\varphi} = m_{\varphi(P)}$ and that ${\rm det}_{\mathbb{C}}(m_{\varphi(P)}) = \varphi(P)$, one obtains 
\begin{equation*}
\begin{aligned} 
\sigma_{V} \circ N_{G/H}(P) &= {\rm det}_{\mathbb{C}}(m_{P}^{\psi}) \\
&= \prod_{\substack{\varphi \in \widehat{G} \\ \varphi|_{H} = \psi}} {\rm det}_{\mathbb{C}}(m_{P}^{\varphi}) \\
&= \prod_{\substack{\varphi \in \widehat{G} \\ \varphi|_{H} = \psi}} \varphi(P).
\end{aligned}
\end{equation*}
On the other hand, one has
$${\rm Ind}(V) \simeq \mathbb{C}[G] \otimes_{\mathbb{C}[H]}(\mathbb{C}[H]e_{\psi}) \simeq \mathbb{C}[G]e_{\psi} \simeq \bigoplus_{\substack{\varphi \in \widehat{G} \\ \varphi|_{H} = \psi}} \mathbb{C}[G]e_{\varphi}, $$
by (\ref{dir_sum}) above from which it follows that the character of $W$ is given 
$$\sum_{\substack{\varphi \in \widehat{G} \\ \varphi|_{H} = \psi}} \varphi. $$
The result then follows from \cref{mult_det,sigma_irr}.
\end{proof}

%\subsection{The Ihara $L$-functions}
\subsection{The Ihara \texorpdfstring{$L$}{L}-functions} \label{iharalfun}
The multiplicative function ${\rm \sigma}_{V}$ from \cref{the_function_sigma} allows us to define the Ihara $L$-function associated to an Artin representation $V$.  As before, $Y$ is a finite graph on which a finite abelian group acts without inversion.
\begin{definition}
Let $V$ be an Artin representation of $G$.  Then, we define the Ihara $L$-function associated to $V$ to be
$$L_{V}(u) = {\rm \sigma}_{V}\left(\theta_{Y}(u) \right) \in 1 + u \mathbb{C}\llbracket u \rrbracket \subseteq \mathbb{C} \llbracket u \rrbracket^{\times}.$$
We let also
$$c_{V}(u) = \sigma_{V}(\gamma_{Y}(u)) \text{ and } h_{V}(u) = \sigma_{V}(\eta_{Y}(u)) $$
so that we have $L_{V}(u)^{-1} = c_{V}(u) \cdot h_{V}(u)$.
\end{definition}
The first thing to notice is that the functions $c_{V}(u)$, $h_{V}(u)$, and $L_{V}(u)$ depend only on the isomorphism class of $V$ as the following theorem shows.
\begin{theorem} \label{dep}
Let $V_{1}$ and $V_{2}$ be two Artin representations of $G$, and assume that they are isomorphic as $\mathbb{C}[G]$-modules, then
$$c_{V_{1}}(u) = c_{V_{2}}(u) \text{ and } h_{V_{1}}(u) = h_{V_{2}}(u) $$
so that we also have $L_{V_{1}}(u) = L_{V_{2}}(u)$.
\end{theorem}
\begin{proof}
This follows from \cref{det_up_to_iso}.
\end{proof}
Because of \cref{dep} combined with the fact that two Artin representations are isomorphic as $\mathbb{C}[G]$-modules if and only if they have the same character, we will often write $L(u,\psi)$ instead of $L_{V}(u)$, where $\psi$ is the character of $V$, and similarly for $c_{V}(u)$ and $h_{V}(u)$.

\begin{theorem} \label{art_c}
With the notation as above, one has
$$c_{V}(u) = (1-u^{2})^{-\chi_{V}(Y)} \in \mathbb{C}\llbracket u \rrbracket. $$
\end{theorem}
\begin{proof}
Because of the additivity property of $\chi_{V}$ in $V$, \cref{exp_beh,mult_det}, it suffices to show the claim for irreducible Artin representations.  If $V$ is an irreducible Artin representation with character $\psi = \psi_{V}$, then $\psi:\mathbb{C}[G] \rightarrow \mathbb{C}$ is a $\mathbb{C}$-algebra morphism.  Thus, using \cref{sigma_irr,basic_prop,art_eul_char}, one calculates
\begin{equation*}
\begin{aligned}
c(u,\psi) &= \psi\left((1-u^{2})^{-\chi_{\mathbb{C}[G]}(Y)} \right) \\
&= (1-u^{2})^{-\psi(\chi_{\mathbb{C}[G]}(Y))} \\
&= (1-u^{2})^{-\chi_{\psi}(Y)},
\end{aligned}
\end{equation*}
and this ends the proof.
\end{proof}

Similarly, we can give an explicit description of the polynomial $h_{V}(u)$.  In order to do so, we need the following proposition.
\begin{proposition} \label{useful_c_d}
Let $G$ be a finite abelian group acting on a finite graph $Y$ without inversion, and let $V$ be an Artin representation.  To simplify the notation, let $M=C_{0}(Y)$.  Then, the diagram
\begin{equation} \label{com_h}
\begin{tikzcd}
 {\rm End}_{\mathbb{C}[G]}(M)[u] \arrow["T_{V}",d] \arrow["\omega_{M}",r] & {\rm End}_{\mathbb{C}[G][u]}(M[u])  \arrow["{\rm det}_{\mathbb{C}[G][u]}",r] & \mathbb{C}[G][u]  \arrow["\sigma_{V}",d] \\
 {\rm End}_{\mathbb{C}}(T_{V}(M))[u] \arrow["\omega_{T_{V}(M)}",r] & {\rm End}_{\mathbb{C}[u]}(T_{V}(M)[u]) \arrow["{\rm det}_{\mathbb{C}[u]}",r]&   \mathbb{C}[u]
\end{tikzcd}
\end{equation}
is commutative.  Here the first vertical arrow to the left of the diagram above is obtained by applying $T_{V}$ to each coefficient of a polynomial in ${\rm End}_{\mathbb{C}[G]}(M)[u]$.
\end{proposition}
\begin{proof}
If $V  = V_{1} \oplus V_{2}$, then
$$T_{V}(M) \simeq T_{V_{1}}(M) \oplus T_{V_{2}}(M), $$
so that
$${\rm End}_{\mathbb{C}[u]}(T_{V}(M)[u]) \simeq {\rm End}_{\mathbb{C}[u]}(T_{V_{1}}(M)[u] \oplus T_{V_{2}}(M)[u]). $$
A simple calculation shows that
\begin{equation*}
\omega_{T_{V}(M)} \circ T_{V}(P) = \left(\omega_{T_{V_{1}}(M)} \circ T_{V_{1}}(P)\right) \oplus \left(\omega_{T_{V_{2}}(M)} \circ T_{V_{2}}(P) \right),
\end{equation*} 
and it follows that
$${\rm det}_{\mathbb{C}[u]}\left(\omega_{T_{V}(M)} \circ T_{V}(P) \right) = {\rm det}_{\mathbb{C}[u]}\left(\omega_{T_{V_{1}}(M)} \circ T_{V_{1}}(P) \right) \cdot {\rm det}_{\mathbb{C}[u]}\left(\omega_{T_{V_{2}}(M)} \circ T_{V_{2}}(P) \right). $$
Since $\sigma_{V}$ satisfies a similar property by \cref{mult_det}, it suffices to prove the claim when $V$ is irreducible, so let us assume so, and let $\psi$ be the character of $V$.  Using \cref{sigma_irr}, the diagram (\ref{com_h}) becomes
\begin{equation*}
\begin{tikzcd}
 {\rm End}_{\mathbb{C}[G]}(M)[u] \arrow[d] \arrow["\omega_{M}",r] & {\rm End}_{\mathbb{C}[G][u]}(M[u]) \arrow[d] \arrow["{\rm det}_{\mathbb{C}[G][u]}",r] & \mathbb{C}[G][u]  \arrow["\psi",d] \\
 {\rm End}_{\mathbb{C}}(e_{\psi}M)[u] \arrow["\omega_{e_{\psi}M}",r] & {\rm End}_{\mathbb{C}[u]}(e_{\psi}M[u]) \arrow["{\rm det}_{\mathbb{C}[u]}",r]&   \mathbb{C}[u].
\end{tikzcd}
\end{equation*}
The square on the right commutes because of the usual base change property of the determinant with respect to the ring morphism $\psi:\mathbb{C}[G] \rightarrow \mathbb{C}$, and the commutativity of the square on the left is a simple calculation left to the reader.
\end{proof}
Given an Artin representation $V$ of $G$, we let
$$\Delta_{V}(u) := \mathcal{I}_{V} - \mathcal{A}_{V}u + \mathcal{Q}_{V}u^{2} \in {\rm End}_{\mathbb{C}}(T_{V}(C_{0}(Y)))[u]. $$
Again, via the isomorphism provided by \cref{conv_iso_pol_vs}, we view $\Delta_{V}(u)$ as an operator on the free $\mathbb{C}[u]$-module $T_{V}(C_{0}(Y))[u]$.
\begin{corollary} \label{3_term}
With the notation as above, one has
$$h_{V}(u) = {\rm det}_{\mathbb{C}[u]}(\Delta_{V}(u)) \in \mathbb{C}[u]. $$
\end{corollary}
\begin{proof}
This is a direct consequence of \cref{useful_c_d}.
\end{proof}

Given any $\psi \in \widehat{G}$, the $\mathbb{C}$-algebra morphism $\psi:\mathbb{C}[G] \rightarrow \mathbb{C}$ induces yet another $\mathbb{C}$-algebra morphism
$$\psi:\mathbb{C}[G]\llbracket u \rrbracket \rightarrow  \mathbb{C}\llbracket u \rrbracket $$
by applying $\psi$ to the coefficients of power series in $\mathbb{C}[G]\llbracket u \rrbracket$.  Given $P(u) \in \mathbb{C}[G]\llbracket u \rrbracket$, we let $P(u,\psi) = \psi(P(u)) \in \mathbb{C}\llbracket u \rrbracket$.  (Note that this is consistent with \cref{sigma_irr}.)  Since
$$\mathbb{C}[G]\llbracket u \rrbracket = \bigoplus_{\psi \in \widehat{G}}\mathbb{C}\llbracket u \rrbracket e_{\psi}, $$
one has
$$P(u) = \sum_{\psi \in \widehat{G}} P(u,\psi) e_{\psi}, $$
for all $P(u) \in \mathbb{C}[G]\llbracket u \rrbracket$.  In particular, we have
$$ \gamma_{Y}(u) = \sum_{\psi \in \widehat{G}}c(u,\psi) e_{\psi} \text{ and } \eta_{Y}(u) = \sum_{\psi \in \widehat{G}}h(u,\psi) e_{\psi}.$$
It follows from \cref{art_eul_char,art_c} that if $G$ acts freely on $\mathbf{E}_{Y}$ and $\chi(X) \le 0$, then $\gamma_{Y}(u) \in \mathbb{C}[G][u]$, a fact which was perhaps not obvious from the original definition of $\gamma_{Y}(u)$.

\begin{comment}
\begin{example} \label{ex2}
We revisit \cref{ex1}.  Let $\psi_{0}$ be the trivial character of $G$, and let $\psi$ be the unique non-trivial character of $G$.  Then, we have
$$h(u,\psi_{0}) = \psi_{0}(\eta_{Y}(u)) = 1-u^{2} - 8u^{3} - u^{4} + 9u^{6} \text{ and } h(u,\psi) = \psi(\eta_{Y}(u)) = 1 + u^{2},$$
and we verify that
$$\eta_{Y}(u) = h(u,\psi_{0})e_{\psi_{0}} + h(u,\psi)e_{\psi}. $$
Moreover,
$$c(u,\psi_{0}) = 1 \text{ and } c(u,\psi) = (1-u^{2})^{2}. $$
Since 
$$\chi_{\mathbb{C}[G]}(Y) = 2e_{\psi_{0}} - 2, $$
this gives us the equality
$$(1-u^{2})^{2e_{\psi_{0}} - 2} = e_{\psi_{0}} + (1-u^{2})^{2}e_{\psi} \in \mathbb{C}[G][u]. $$
\end{example}
\end{comment}

\begin{example} \label{ex22}
We revisit \cref{ex11}.  Let $\psi_{0}$ be the trivial character of $G$, $\psi_{1}$ the unique character of order $2$ of $G$, and $\psi_{2}, \psi_{3}$ the two characters of order $4$.  Then, we have
$$h(u,\psi_{0}) = \psi_{0}(\eta_{Y}(u)) = 1-4u^{2} + 3u^{4}, h(u,\psi_{1}) = \psi_{1}(\eta_{Y}(u)) = 1 + 4u^{2} + 3u^{4},$$
and
$$h(u,\psi_{j}) = \psi_{j}(\eta_{Y}(u)) = 1+u^{2}, $$
for $j=2,3$.  We verify that
$$\eta_{Y}(u) = \sum_{j=0}^{3}h(u,\psi_{j})e_{\psi_{j}}. $$
Moreover,
$$c(u,\psi_{0}) = 1 = c(u,\psi_{1}), $$
and
$$c(u,\psi_{2}) = 1 - u^{2} = c(u,\psi_{3}) $$
Since 
$$\chi_{\mathbb{C}[G]}(Y) = e_{2} - 1, $$
this gives us the equality
$$(1-u^{2})^{1 - e_{2}} = e_{\psi_{0}} + e_{\psi_{1}} + (1-u^{2})e_{\psi_{2}} + (1-u^{2})e_{\psi_{3}} \in \mathbb{C}[G][u]. $$
\demo
\end{example}

%\subsection{The Artin formalism}
\subsection{The Artin formalism} \label{af}
The setup is the usual one: $Y$ is a finite graph on which a finite abelian group $G$ acts without inversion.  In this section, we study the Artin formalism for the Ihara $L$-functions.  By the Artin formalism, we mean the usual three properties of additivity, induction and inflation satisfied by the classical Artin $L$-functions in algebraic number theory as listed for instance in \cite[page 15]{Tate:1984}.  We start with the additivity property.

\begin{theorem} \label{formalism_one}
Let $V_{1}$ and $V_{2}$ be two Artin representations of $G$, then one has
$$c_{V_{1} \oplus V_{2}}(u) = c_{V_{1}}(u) \cdot c_{V_{2}}(u) \text{ and } h_{V_{1} \oplus V_{2}}(u) = h_{V_{1}}(u) \cdot h_{V_{2}}(u), $$
so that we also have
$$L_{V_{1} \oplus V_{2}}(u) = L_{V_{1}}(u) \cdot L_{V_{2}}(u). $$
\end{theorem}
\begin{proof}
This follows directly from \cref{mult_det}.
\end{proof}

The induction property is contained in the following theorem.
\begin{theorem} \label{formalism_two}
Let $H$ be a subgroup of $G$, $V$ an Artin representation of $H$, and let ${\rm Ind}(V)$ denote the induced representation from $H$ to $G$.  Then
$$c_{{\rm Ind}(V)}(u) = c_{V}(u) \text{ and } h_{{\rm Ind}(V)}(u) = h_{V}(u), $$
so that we also have
$$L_{{\rm Ind}(V)}(u) = L_{V}(u). $$
\end{theorem}
\begin{proof}
This follows from \cref{useful_for_induction,ind_really}.
\end{proof}

We point out that the inflation property is not satisfied in general in the situation of branched covers of graphs.  This is already apparent from \cref{ex11}.  Indeed, one calculates
$$h_{X}(u) =  (1-u^{2})^{2} \text{ and } c_{X}(u) = 1,$$
but
$$h(u,\psi_{0}) = 1-4u^{2} + 3u^{4} \text{ and } c(u,\psi_{0}) = 1, $$
as we calculated in \cref{ex22} above.  If the inflation property were satisfied, then we would have $h_{X}(u) = h(u,\psi_{0})$ which is not the case.

As a direct consequence of \cref{formalism_two}, we obtain the following result.
\begin{corollary} \label{product_form_for_c}
Let $G$ be a finite abelian group acting without inversion on a finite graph $Y$.  Then, one has
$$c_{Y}(u) = \prod_{\psi \in \widehat{G}}c(u,\psi). $$
\end{corollary}
\begin{proof}
Let $V_{0}$ be the trivial representation of the trivial group, and consider $V = {\rm Ind}(V_{0})$, the induced representation up to $G$.  The representation $V$ is isomorphic to 
$$\mathbb{C}[G] = \bigoplus_{\psi \in \widehat{G}} \mathbb{C}[G] e_{\psi}. $$
It follows from \cref{formalism_one,formalism_two} that
$$c_{Y}(u) = \prod_{\psi \in \widehat{G}} c(u,\psi)$$
as desired.
\end{proof}

\begin{remark}
By definition, we have $c_{Y}(u) = (1-u^{2})^{-\chi(Y)}$ and if $\psi \in \widehat{G}$ and $G$ acts freely on $\mathbf{E}_{Y}$, then 
$$c(u,\psi) = (1-u^{2})^{-\chi(X) + r_{0}(\psi)},$$ 
because of \cref{art_c,art_eul_char}, where $X = Y_{G}$.  Therefore, the conclusion of \cref{product_form_for_c} is true if and only if
\begin{equation} \label{rh_disguised}
\chi(Y) = |G| \cdot \chi(X) - \sum_{\psi \in \widehat{G}}r_{0}(\psi). 
\end{equation}
In view of \cref{sum_order} below, noting that $|G| = [Y:X]$ when $Y$ is connected, (\ref{rh_disguised}) is equivalent to the Riemann-Hurwitz formula of Baker and Norine (\cref{riemann_hurwitz}) in our situation.
\end{remark}

\begin{lemma} \label{sum_order}
One has
$$\sum_{\psi \in \widehat{G}} r_{0}(\psi) = \sum_{w \in V_{Y}}(m_{w} - 1). $$
\end{lemma}
\begin{proof}
Observe first that if $\psi \in \widehat{G}$, then \cref{dim_iso_component} implies
$${\rm dim}_{\mathbb{C}} \, (e_{\psi} C_{0}(Y)) = |V_{X}| - r_{0}(\psi). $$
Since
$$C_{0}(Y) = \bigoplus_{\psi \in \widehat{G}} e_{\psi} C_{0}(Y),$$
we have
\begin{equation*}
\begin{aligned} 
{\rm dim}_{\mathbb{C}} (C_{0}(Y)) &= \sum_{\psi \in \widehat{G}} {\rm dim}_{\mathbb{C}} (e_{\psi} C_{0}(Y)) \\
&= \sum_{\psi \in \widehat{G}} (|V_{X}| - r_{0}(\psi)) \\
&= |G||V_{X}| - \sum_{\psi \in \widehat{G}} r_{0}(\psi). \\
\end{aligned}
\end{equation*}
It follows that
\begin{equation*}
\begin{aligned}
\sum_{\psi \in \widehat{G}} r_{0}(\psi) &= |G||V_{X}| - {\rm dim}_{\mathbb{C}} (C_{0}(Y)) \\
&= |G||V_{X}| - \sum_{v \in V_{X}}(G:G_{v})\\
&= \sum_{v \in V_{X}}(|G| - (G:G_{v})) \\
&= \sum_{v \in V_{X}} \sum_{w \in f^{-1}(v)}(m_{w} - 1)  \\
&= \sum_{w \in V_{Y}} (m_{w} - 1),
\end{aligned}
\end{equation*}
and this ends the proof.
\end{proof}

\begin{comment}
\begin{example}
We revisit \cref{ex1} and \cref{ex2}.  One has $\chi(Y) = -2$, and thus
$$c_{Y}(u) = (1-u^{2})^{2}. $$
Since $c(u,\psi_{0}) = 1$ and $c(u,\psi) = (1-u^{2})^{2}$, we do have the equality
$$c_{Y}(u) = c(u,\psi_{0}) \cdot c(u,\psi) $$
as predicted by \cref{product_form_for_c}.  The Riemann-Hurwitz formula of Baker and Norine reads in this example as
\begin{equation*}
\begin{aligned}
\chi(Y) &= 2 \cdot \chi(X) - \sum_{w \in V_{Y}}(m_{w} - 1) \\
&= 2 \cdot 0 - (1 + 0 + 0 + 1) \\
&= -2.
\end{aligned}
\end{equation*}
\end{example}
\end{comment}

\begin{example}
We revisit \cref{ex11,ex22}.  One has $\chi(Y) = -2$, and thus
$$c_{Y}(u) = (1-u^{2})^{2}. $$
Since $c(u,\psi_{0}) = c(u,\psi_{1}) = 1$, and $c(u,\psi_{2}) = c(u,\psi_{3}) = (1-u^{2})$, we do have the equality
$$c_{Y}(u) = \prod_{j=0}^{3}c(u,\psi_{j}) $$
as predicted by \cref{product_form_for_c}.  The Riemann-Hurwitz formula of Baker and Norine reads in this example as
\begin{equation*}
\begin{aligned}
\chi(Y) &= 4 \cdot \chi(X) - \sum_{w \in V_{Y}}(m_{w} - 1) \\
&= 4 \cdot 0 - (1 + 0 + 0  + 0+ 0+ 1) \\
&= -2.
\end{aligned}
\end{equation*}
\demo
\end{example}

The function $h_{Y}(u)$ satisfies a product formula as well similar to the corresponding one for $c_{Y}(u)$ of \cref{product_form_for_c}. 
\begin{corollary} \label{prod_formula_th}
Let $Y$ be a finite graph on which a finite abelian group $G$ acts without inversion.  One has
\begin{equation} \label{prod_formula}
h_{Y}(u) = \prod_{\psi \in \widehat{G}} h(u,\psi).
\end{equation}
\end{corollary}
\begin{proof}
The proof is identical to the one for \cref{product_form_for_c}.
\end{proof}

\begin{comment}
\begin{example}
Going back to \cref{ex1} and \cref{ex2}, one calculates
$$h_{Y}(u) = 1 - 8u^{3} - 2u^{4} - 8u^{5} + 8u^{6} + 9u^{8}, $$
and one sees that
\begin{equation*}
\begin{aligned}
h_{Y}(u) &= h(u,\psi_{0}) \cdot h(u,\psi) \\
&= (1-u^{2} - 8u^{3} - u^{4} + 9u^{6}) \cdot (1 + u^{2}), 
\end{aligned}
\end{equation*}
as predicted by \cref{prod_formula_th}.
\end{example}
\end{comment}

\begin{example}
Going back to \cref{ex11,ex22}, one calculates
$$h_{Y}(u) = 1  + 2u^{2} - 9u^{4} - 20u^{6} - u^{8} + 18u^{10} + 9u^{12}, $$
and one sees that
\begin{equation*}
\begin{aligned}
h_{Y}(u) &= \prod_{j=0}^{3}h(u,\psi_{j}) \\
&= (1-4u^{2} + 3u^{4}) \cdot (1+4u^{2} + 3u^{4})  \cdot (1 + u^{2}) \cdot (1+u^{2}), 
\end{aligned}
\end{equation*}
as predicted by \cref{prod_formula_th}.
\demo
\end{example}

%\section{Branched $\mathbb{Z}_{p}$-towers of graphs}
\section{Branched \texorpdfstring{$\mathbb{Z}_{p}$}{L}-towers of graphs} \label{br_section}
Let $p$ be a prime number.  Given a finite connected graph $X$, a function $\alpha:\mathbf{E}_{X} \rightarrow \mathbb{Z}_{p}$ satisfying $\alpha(\bar{s}) = -\alpha(s)$ for all $s \in \mathbf{E}_{X}$, and a family $\mathcal{G}$ of closed subgroups of $\mathbb{Z}_{p}$ indexed by $V_{X}$, we consider the graphs $X_{n} = X(\mathbb{Z}_{p}/p^{n}\mathbb{Z}_{p},\mathcal{G}_{n},\alpha_{n})$ as explained in \cite[\S 4.3]{Gambheera/Vallieres:2024}.  We assume that the graphs $X_{n}$ are connected for all $n \ge 0$.  These graphs come with natural branched covers $X_{n+1} \rightarrow X_{n}$ of degree $p$ which together form a branched $\mathbb{Z}_{p}$-tower of finite connected graphs
\begin{equation} \label{z_p_tower}
X = X_{0} \leftarrow X_{1} \leftarrow X_{2} \leftarrow \ldots \leftarrow X_{n} \leftarrow \ldots 
\end{equation}
as explained in more details in \cite[\S 4.3]{Gambheera/Vallieres:2024}.  In this section, we use Ihara $L$-functions to study how ${\rm ord}_{p}(\kappa(X_{n}))$ varies as $n$ becomes large.
%\subsection{The number of spanning trees in abelian branched covers}
\subsection{The number of spanning trees in abelian branched covers} \label{num_of_span}
We use the same notation as in \cref{equiv_sec}.

\begin{proposition} \label{order_of_van}
Let $Y$ be a finite connected graph on which a finite abelian group $G$ acts without inversion and freely on directed edges, and let $X = Y_{G}$ be the quotient graph.  Let $C = (c_{ij})$ be the $g \times g$ diagonal matrix satisfying $c_{ii} = |G_{i}|$, for all $i =1,\ldots,g$.  If $D$ and $A$ denote the degree and adjacency matrices of $X$, respectively, then
$$h(u,\psi_{0}) = {\rm det}(I - ACu + (DC - I)u^{2}), $$
where $I$ is the identity matrix.  Moreover, one has ${\rm ord}_{u=1}(h(u,\psi_{0})) \ge 1$.
\end{proposition}
\begin{proof}
By \cref{equi_exp}, we have
\begin{equation*}
\begin{aligned}
h(u,\psi_{0}) &= \psi_{0}({\rm det}(\mathbf{I} - \mathbf{A}u + \mathbf{Q}u^{2})) \\
&= {\rm det}(I - \psi_{0}(\mathbf{A})u + \psi_{0}(\mathbf{Q})u^{2}).
\end{aligned}
\end{equation*}
For a given $i \in \{1,\dots,g \}$, let $\{\sigma_{1},\sigma_{2},\ldots,\sigma_{r} \}$ be a complete set of representatives of $G/G_{i}$.  Using (\ref{useful_identity}), one has
$$\psi_{0}(\ell_{i}(w_{j})) = \sum_{k=1}^{r}a_{\sigma_{k}w_{i}}(w_{j}). $$ 
Since the function $\mathbf{E}_{Y,w_{j}} \rightarrow \mathbf{E}_{X,v_{j}}$ is $m_{j}$-to-$1$, where $m_{j} = m_{w_{j}}$, one has
$$\psi_{0}(\ell_{i}(w_{j})) = a_{ij} |G_{j}|, $$
where $a_{ij}$ is the number of directed edges of $X$ going from $v_{j}$ to $v_{i}$.  It follows that
\begin{equation} \label{a}
\psi_{0}(\mathbf{A}) = AC. 
\end{equation}
Moreover,
$$\psi_{0}\left(({\rm val}_{Y}(w_{i}) - 1)e_{i} \right) = |G_{i}|{\rm val}_{X}(v_{i}) -1, $$
from which it follows that
\begin{equation} \label{q}
\psi_{0}(\mathbf{Q}) = DC - I. 
\end{equation}
Putting (\ref{a}) and (\ref{q}) together gives
\begin{equation} \label{conc_pro}
h(u,\psi_{0}) =  {\rm det}(I - ACu + (DC - I)u^{2}),
\end{equation}
and this ends the proof of the first part of the proposition.  For the second part, using (\ref{conc_pro}), one has
$$h(1,\psi_{0}) = {\rm det}(DC - AC) = {\rm det}(\mathcal{L}_{X}) \cdot {\rm det}(C), $$
where $\mathcal{L}_{X}$ is the laplacian operator for $X$.  By (\ref{lap_singular}), we have $h(1,\psi_{0}) = 0$ which ends the proof.
\end{proof}
This last proposition allows us to determine the order of vanishing at $u=1$ of the polynomials $h(u,\psi)$.
\begin{theorem} \label{order_of_vanishing}
Let $Y$ be a finite connected graph on which a finite abelian group $G$ acts without inversion and freely on directed edges.  If $\chi(Y) \neq 0$, then 
$${\rm ord}_{u=1}h(u,\psi) = 0 $$
for all non-trivial character $\psi \in \widehat{G}$, and ${\rm ord}_{u=1}h(u,\psi_{0}) = 1$.
\end{theorem}
\begin{proof}
By (\ref{lap_singular}) and \cref{hashimoto}, we have ${\rm ord}_{u=1}(h_{Y}(u)) = 1$.  Since by \cref{order_of_van}, ${\rm ord}_{u=1}(h(u,\psi_{0})) \ge 1$, the result follows at once from (\ref{prod_formula}).
\end{proof}

\begin{corollary}
Let $Y$ be a finite connected graph on which a finite abelian group $G$ acts without inversion and freely on directed edges.  If $\chi(Y) \neq 0$, then 
\begin{equation} \label{ess_pr}
h_{Y}'(1) = h'(1,\psi_{0}) \cdot \prod_{\psi \neq \psi_{0}} h(1,\psi).
\end{equation}
\end{corollary}
\begin{proof}
This follows directly from \cref{order_of_vanishing} after differentiating (\ref{prod_formula}) and evaluating at $u=1$.
\end{proof}
Combining (\ref{ess_pr}) with \cref{hashimoto} gives
\begin{equation} \label{starting_pt}
-2 \chi(Y) \kappa(Y) = h'(1,\psi_{0}) \cdot \prod_{\psi \neq \psi_{0}} h(1,\psi), 
\end{equation}
provided $Y$ is connected and $\chi(Y) \neq 0$.  This last formula will be our starting point in \cref{ana_sec} below to derive the analogue of Iwasawa's asymptotic class number formula using Ihara $L$-functions.  

%\subsection{Ihara $L$-functions and voltage assignments}
\subsection{Ihara \texorpdfstring{$L$}{L}-functions and voltage assignments} \label{Ihar}
Throughout this section, we let $X$ be a finite graph, $G$ a finite abelian group, $\alpha:\mathbf{E}_{X} \rightarrow G$ a function satisfying $\alpha(\bar{s}) = \alpha(s)^{-1}$ for all $s \in \mathbf{E}_{X}$, and $\mathcal{G} =\{(v,G_{v}):G_{v} \le G \}$ a collection of subgroups of $G$ indexed by the vertices of $X$.  We pick a labeling of the vertices of $X$, say $V_{X} = \{v_{1},\ldots,v_{g} \}$, and we consider the graph $Y = X(G,\mathcal{G},\alpha)$.
\begin{proposition} \label{exp_adj_op}
With the notation as above, for each $i=1,\ldots,g$, let $w_{i} = (v_{i},G_{i}) \in V_{Y}$.  Then, for each $j=1,\ldots,g$, one has
$$\mathcal{A}(w_{j}) = \sum_{i=1}^{g}\lambda_{ij}w_{i},$$
where
$$\lambda_{ij} = \sum_{\substack{s \in \mathbf{E}_{X} \\ {\rm inc}(s) = (v_{j},v_{i})}}\alpha(s) N_{G_{j}} \in \mathbb{Z}[G] \subseteq \mathbb{C}[G]. $$
\end{proposition}
\begin{proof}
We calculate
\begin{equation*}
\begin{aligned}
\mathcal{A}(w_{j}) &= \sum_{\varepsilon \in \mathbf{E}_{Y,w_{j}}} t(\varepsilon) \\
&= \sum_{s \in \mathbf{E}_{X,v_{j}}} \sum_{h \in G_{j}} t((s,h)) \\
&= \sum_{s \in \mathbf{E}_{X,v_{j}}} \sum_{h \in G_{j}} (t(s),h\alpha(s)G_{t(s)}) \\
&= \sum_{s \in \mathbf{E}_{X,v_{j}}} \sum_{h \in G_{j}} h \alpha(s) (t(s),G_{t(s)}) \\
&= \sum_{i=1}^{g} \lambda_{ij} w_{i},
\end{aligned}
\end{equation*}
as we wanted to show.
\end{proof}

\begin{definition} \label{zak_mat}
With the same notation as above, we define two matrices $\mathbf{A}_{\alpha}, \mathbf{C} \in M_{g}(\mathbb{Z}[G])$ as follows.  The matrix $\mathbf{A}_{\alpha} = (\mathbf{a}_{ij}(\alpha))$ is defined via
$$\mathbf{a}_{ij}(\alpha) = \sum_{\substack{s \in \mathbf{E}_{X}\\{\rm inc}(s)=(v_{j},v_{i})}} \alpha(s), $$
and the matrix $\mathbf{C} = (\mathbf{c}_{ij})$ is the diagonal matrix defined via $\mathbf{c}_{ii} = N_{G_{i}}$.  We also let $\mathbf{D} = (d_{ij}) \in M_{g}(\mathbb{Z})$ be the usual degree matrix of $X$, and $\mathbf{I}$ the $g \times g$ identity matrix in $M_{g}(\mathbb{Z}[G])$.  At last, we define
$$\xi_{Y}(u) = {\rm det}(\mathbf{I} - \mathbf{A}_{\alpha}\mathbf{C}u + (\mathbf{D}\mathbf{C} - \mathbf{I})u^{2}) \in 1 + u\mathbb{Z}[G][u]. $$
\end{definition}
Given $\psi \in \widehat{G}$, similarly to the polynomial $h(u,\psi)$, we let 
$$z(u,\psi) = \psi(\xi_{Y}(u)) \in 1 + u\mathbb{C}[u].$$
\begin{theorem} \label{link_zak}
With the notation as above, for each $\psi \in \widehat{G}$, one has
$$(1-u^{2})^{r_{0}(\psi)} \cdot h(u,\psi) = z(u,\psi). $$
Moreover, given $\psi \in \widehat{G}$, if we let $\widetilde{\mathbf{A}}_{\alpha}, \widetilde{\mathbf{C}}$, and $\widetilde{\mathbf{D}}$ be obtained from the corresponding matrices $\mathbf{A}_{\alpha}, \mathbf{C}$, and $\mathbf{D}$, respectively, by removing the rows and columns corresponding to the indices $i$ such that $G_{i} \not \subseteq {\rm ker}(\psi)$, then one has
$$h(u,\psi) = {\rm det}(\widetilde{I} - \psi(\widetilde{\mathbf{A}}_{\alpha})\psi(\widetilde{\mathbf{C}})u + (\psi(\widetilde{\mathbf{D}})\psi(\widetilde{\mathbf{C}}) - \widetilde{I})u^{2}), $$
where $\widetilde{I}$ is the identity matrix with $|V_{X}| - r_{0}(\psi)$ number of rows and columns.
\end{theorem}
\begin{proof}
Let $\psi \in \widehat{G}$, and let $j$ be such that $G_{j} \not \subseteq {\rm ker}(\psi)$.  \cref{sim_for} shows that applying $\psi$ to the $j$th column of $\mathbf{A}$ gives a column consisting of zeros only.  The same is true if one applies $\psi$ to the $j$th column of $\mathbf{A}_{\alpha}\mathbf{C}$.  Applying $\psi$ to the $j$th column of $\mathbf{Q}$ also gives a column consisting of zeros only, but applying $\psi$ to the $j$th column of $\mathbf{D}\mathbf{C} - \mathbf{I}$ gives a column consisting of zeros except in the $j$th row where one has a $-1$.  It follows that the $j$th column of $\psi(\mathbf{I} - \mathbf{A}u + \mathbf{Q}u^{2})$ consists of zeros except for a $1$ in the $j$th row, whereas the $j$th column of $\psi(\mathbf{I} - \mathbf{A}_{\alpha}\mathbf{C}u + (\mathbf{D}\mathbf{C}-\mathbf{I})u^{2})$ consists of zeros except for a $1-u^{2}$ in the $j$th row.  The claim will then follows if we show that the $i$th row of both 
$$\psi(\mathbf{I} - \mathbf{A}u + \mathbf{Q}u^{2}) \text{ and } \psi(\mathbf{I} - \mathbf{A}_{\alpha}\mathbf{C}u + (\mathbf{D}\mathbf{C}-\mathbf{I})u^{2})$$
are the same whenever $G_{i} \subseteq {\rm ker}(\psi)$.  But this follows from \cref{dim_iso_component,dec_adj,exp_adj_op}.
\end{proof}
Note that
$$(\mathbf{I} - \mathbf{A}_{\alpha}\mathbf{C}u + (\mathbf{D}\mathbf{C}-\mathbf{I})u^{2})^{t} = \mathbf{I} - \mathbf{C}\mathbf{A}_{\alpha}^{t}u + (\mathbf{C}\mathbf{D}-\mathbf{I})u^{2} $$
and thus 
\begin{equation*}
\begin{aligned}
L(u,\psi)^{-1} &= c(u,\psi) \cdot h(u,\psi) \\
&= (1-u^{2})^{-\chi(X) + r_{0}(\psi)} \cdot h(u,\psi) \\
&= (1-u^{2})^{-\chi(X)}z(u,\psi),
\end{aligned}
\end{equation*}
where 
$$z(u,\psi) = {\rm det}\left( I - \psi(\mathbf{C})\psi(\mathbf{A}_{\alpha}^{t})u + (\psi(\mathbf{C})\psi(\mathbf{D})-I)u^{2} \right), $$
and we invite the reader to compare the expression for $L(u,\psi)^{-1}$ above with the one contained in \cite[Theorem 4.17]{Zakharov:2021} in the situation where $G$ is abelian and the action is without inversion (so that there are no legs).
\begin{example}
We revisit \cref{ex11,ex22}.  We calculate
\begin{equation*}
\mathbf{A}_{\alpha} = 
\begin{pmatrix}
0 & \bar{3} + \bar{2} \\
\bar{1} + \bar{2} & 0
\end{pmatrix},
\mathbf{C} =
\begin{pmatrix}
1 & 0 \\
0 & \bar{0} + \bar{2}
\end{pmatrix}, \text{ and }
\mathbf{D} =
\begin{pmatrix}
2 & 0 \\
0 & 2
\end{pmatrix},
\end{equation*}
so that
\begin{equation*}
\begin{aligned} 
\xi_{Y}(u) &= {\rm det}(\mathbf{I} - \mathbf{A}_{\alpha}\mathbf{C}u + (\mathbf{D}\mathbf{C} - \mathbf{I})u^{2}) \\
&= 1 -2(\bar{1} + \bar{3})u^{2} + (1 + 2 \cdot \bar{2})u^{4} \in \mathbb{Z}[G][u].
\end{aligned}
\end{equation*}
We have
$$r_{0}(\psi_{0}) = r_{0}(\psi_{1}) = 0 \text{ and } r_{0}(\psi_{2}) = r_{0}(\psi_{3}) = 1. $$
We then calculate
\begin{equation*}
\begin{aligned}
z(u,\psi_{0}) &= 1 - 4u^{2} + 3u^{4} = h(u,\psi_{0}) \\
z(u,\psi_{1}) &= 1 + 4u^{2} + 3u^{4} = h(u,\psi_{1}) \\
z(u,\psi_{2}) &= 1-u^{4} = (1-u^{2}) \cdot h(u,\psi_{2}) \\
z(u,\psi_{3}) &= 1-u^{4} = (1-u^{2}) \cdot h(u,\psi_{3}),
\end{aligned}
\end{equation*}
as expected by \cref{link_zak}.
\end{example}

%\subsection{The analogue of Iwasawa'a asymptotic class number formula}
\subsection{The analogue of Iwasawa's asymptotic class number formula} \label{ana_sec}
We start by reminding the reader about our notation.  Let $p$ be a rational prime, $X$ a finite connected graph, and $\alpha:\mathbf{E}_{X} \rightarrow \mathbb{Z}_{p}$ a function satisfying $\alpha(\bar{s}) = -\alpha(s)$ for all $s \in \mathbf{E}_{X}$.  Let also $\mathcal{G}$ be a collection of closed subgroups $G_{v}$ of $\mathbb{Z}_{p}$ indexed by the vertices of $X$.  For every $n \in \mathbb{Z}_{\ge 0}$, we let $\Gamma_{n} = \mathbb{Z}_{p}/p^{n}\mathbb{Z}_{p} \simeq \mathbb{Z}/p^{n}\mathbb{Z}$, and we have a natural surjective group morphism $\pi_{n}:\mathbb{Z}_{p} \rightarrow \Gamma_{n}$.  Given $v \in V_{X}$, we let $\Gamma_{n,v} = \pi_{n}(G_{v}) \le \Gamma_{n}$, and we set
$$\mathcal{G}_{n} = \{(v,\Gamma_{n,v}): v \in V_{X} \}. $$
The composition $\alpha_{n} = \pi_{n} \circ \alpha$ also satisfies $\alpha_{n}(s) = -\alpha_{n}(s)$ for all $s \in \mathbf{E}_{X}$, and for each $n \ge 1$, we consider the finite graph $X_{n} = X(\Gamma_{n},\mathcal{G}_{n},\alpha_{n})$.  As explained in \cite[\S 4.2]{Gambheera/Vallieres:2024}, the graphs $X_{n}$ comes with natural branched covers which together form a branched $\mathbb{Z}_{p}$-towers 
$$X = X_{0} \leftarrow X_{1} \leftarrow X_{2} \leftarrow \ldots \leftarrow X_{n} \leftarrow \ldots $$
of finite graphs.  We will always assume that all graphs $X_{n}$ are connected (see \cite[\S 4.5]{Gambheera/Vallieres:2024} for a condition that guarantees the connectedness of the graphs $X_{n}$).  As in \cite{Gambheera/Vallieres:2024}, we let
$$V^{unr} = \{v \in V_{X}: G_{v} = 0 \} \text{ and } V^{ram} = \{v \in V_{X}: G_{v} \neq 0 \}. $$
Moreover, given $v \in V^{ram}$, we let $k_{v} \in \mathbb{Z}_{\ge 0}$ be such that $G_{v} = p^{k_{v}}\mathbb{Z}_{p}$, and we let
$$n_{1} = {\rm max}\{k_{v}: v \in V^{ram} \}. $$
Note that
$$\Gamma_{n,v} \simeq p^{{\rm min}(k_{v},n)}\mathbb{Z}_{p}/p^{n}\mathbb{Z}_{p} $$
so that if $n \ge k_{v}$, then one has
$$\Gamma_{n,v} \simeq \mathbb{Z}/p^{n-k_{v}}\mathbb{Z}. $$
It follows that if $n \ge n_{1}$, then there are $p^{k_{v}}$ vertices of $X_{n}$ above $v$ with ramification indices $m_{v,n} = p^{n-k_{v}}$ if $v \in V^{ram}$, whereas if $v \in V^{unr}$, then there are $p^{n}$ vertices of $X_{n}$ lying above $v$ and $m_{v,n} = 1$.  Just as before, we introduce a labeling $V_{X} = \{v_{1},\ldots,v_{g} \}$, but we also assume that 
$$V^{unr} = \{v_{1},\ldots,v_{t} \} \text{ and } V^{ram} = \{v_{t+1},\ldots,v_{g}\}.$$
We will often replace a vertex $v_{i}$ appearing in any notation by the corresponding index only.  For instance, we will write $k_{i}$ instead of $k_{v_{i}}$ if $i \in \{t+1,\ldots,g \}$.
  
Throughout this section, we also assume that there exists $n_{2} \ge 1$ such that $\chi(X_{n_{2}}) < 0$.  Note that \cref{riemann_hurwitz} implies that $\chi(X_{n}) \neq 0$ for all $n \ge n_{2}$.  If $n \ge n_{2}$, then (\ref{starting_pt}) gives
\begin{equation} \label{starting_pt1}
-2 \chi(X_{n}) \kappa(X_{n}) = h'(1,\psi_{0})\prod_{\substack{\psi \in \widehat{\Gamma}_{n} \\ \psi \neq \psi_{0}}} h(1,\psi).
\end{equation}
In order to understand ${\rm ord}_{p}(\kappa(X_{n}))$, we have to understand the $p$-adic valuation of $\chi(X_{n})$, $h'(1,\psi_{0})$, and $h(1,\psi)$ when $\psi \in \widehat{\Gamma}_{n}$ is non-trivial.  Therefore, it is convenient to work with $\mathbb{C}_{p}$ instead of $\mathbb{C}$.  Everything in this paper could have been done with $\mathbb{C}_{p}$ instead of $\mathbb{C}$, since $\mathbb{C}_{p}$ is also an algebraically closed field of characteristic $0$.  So from now on, we will see all of the numbers involved inside $\mathbb{C}_{p}$.  
\begin{proposition} \label{ec_i}
With the notation as above, there exists a constant $c \in \mathbb{Z}$ such that
\begin{equation*}
{\rm ord}_{p}(\chi(X_{n})) =
\begin{cases}
n + {\rm ord}_{p}(\chi(X)), &\text{ if } V^{ram} = \varnothing \text{ and } \chi(X) \neq 0; \\
c, &\text{ if } V^{ram} \neq \varnothing,
\end{cases}
\end{equation*}
when $n$ is large enough.
\end{proposition}
\begin{proof}
\cref{riemann_hurwitz} implies that if $n \ge n_{1}$, one has
\begin{equation*}
\begin{aligned}
\chi(X_{n}) &= p^{n}\chi(X) - \sum_{w \in V_{X_{n}}}(m_{w}-1) \\
&= p^{n}\chi(X) - \sum_{v \in V^{ram}}p^{k_{v}}(p^{n-k_{v}}-1) \\
&= p^{n}(\chi(X) - |V^{ram}|) + \sum_{v \in V^{ram}} p^{k_{v}}
\end{aligned}
\end{equation*}
from which the result follows.
\end{proof}
For the sake of clarity in this section, we will write $h_{X_{n}}(1,\psi)$ instead of simply $h(1,\psi)$ as we have done so far, where $\psi$ is a character of $\Gamma_{n}$.

\begin{proposition} \label{tc_i}
With the notation as above, there exists $\lambda_{0} \in \mathbb{Z}_{\ge 0}$ and $\nu_{0} \in \mathbb{Z}$ such that
$${\rm ord}_{p}(h_{X_{n}}'(1,\psi_{0})) = \lambda_{0} n + \nu_{0} $$
when $n$ is large enough.
\end{proposition}
\begin{proof}
We let $C_{n} = (c_{ij}^{(n)})$ be the diagonal matrix from \cref{order_of_van}.  If $n \ge n_{1}$, then one has
\begin{equation*}
c_{ii}^{(n)} = 
\begin{cases}
p^{n-k_{i}}, &\text{ if } i=t+1,\ldots,g; \\
1, &\text{ otherwise}.
\end{cases}
\end{equation*}
\cref{order_of_van} implies the equality
$$h_{X_{n}}(u,\psi_{0}) = {\rm det}(I - AC_{n}u + (DC_{n} - I)u^{2}), $$
where $A$ and $D$ are the adjacency and the degree matrices of $X$, respectively.  It follows that there exists
$$P(x_{t+1},\ldots,x_{g},u) \in \mathbb{Z}[x_{t+1},\ldots,x_{g},u] $$
such that
$$h_{X_{n}}(u,\psi_{0}) = P(p^{n - k_{t+1}},\ldots,p^{n - k_{g}},u),$$
provided $n \ge n_{1}$.  Therefore, there exists $Q(T) \in \overline{\mathbb{Q}}_{p}[T]$ such that
$$h_{X_{n}}'(1,\psi_{0}) = Q(p^{n}), $$
provided $n \ge n_{1}$.  The result then follows from \cref{useful_r_t}.
\end{proof}
From now on, whenever $0 \le j \le n$ we let
$$\widehat{\Gamma}_{n}^{(j)} = \{\psi \in \widehat{\Gamma}_{n}: {\rm ord}(\psi) = p^{j} \}.$$ 
Note that all $\psi \in \widehat{\Gamma}_{n}^{(j)}$ for a fixed $j$ have the same kernel, and moreover one has $|\widehat{\Gamma}_{n}^{(j)}| =\phi(p^{j})$, where $\phi$ here is the Euler $\phi$-function.

\begin{proposition} \label{ram_i}
With the notation as above, if $0 < j \le n_{1}$, then there exists $\lambda_{j} \in \mathbb{Z}_{\ge 0}$ and $\nu_{j} \in \mathbb{Z}$ such that
$${\rm ord}_{p} \left( \prod_{\psi \in \widehat{\Gamma}_{n}^{(j)}} h_{X_{n}}(1,\psi)\right) = \lambda_{j}n + \nu_{j} $$
when $n$ is large enough.
\end{proposition}
\begin{proof}
Let $\psi \in \widehat{\Gamma}_{n}^{(j)}$ with $ 0 < j \le n_{1} \le n$, and let
$$V_{\psi,n} = \{v \in V^{ram}: \Gamma_{n,v} \not \subseteq {\rm ker}(\psi) \}. $$
Note that the set $V_{\psi,n}$ does not depend on $n$ when $n \ge n_{1}$.  Therefore we will simply write $V_{\psi}$ for this set instead of $V_{\psi,n}$ provided $n \ge n_{1}$.  Consider the matrices $\widetilde{\mathbf{A}}_{\alpha_{n}}$, $\widetilde{\mathbf{C}}_{n}$, and $\widetilde{\mathbf{D}}$ from \cref{link_zak} so that we have
$$h_{X_{n}}(u,\psi) = {\rm det}(\widetilde{I} - \psi(\widetilde{\mathbf{A}}_{\alpha_{n}})\psi(\widetilde{\mathbf{C}}_{n})u + (\psi(\widetilde{\mathbf{D}})\psi(\widetilde{\mathbf{C}}_{n})- \widetilde{I})u^{2}), $$
from which it follows that
$$h_{X_{n}}(1,\psi) = {\rm det}(\widetilde{\mathbf{D}} - \psi(\widetilde{\mathbf{A}}_{\alpha_{n}})){\rm det}(\psi(\widetilde{\mathbf{C}}_{n})). $$
Observe that for $n \ge n_{1}$, the quantity ${\rm det}(\widetilde{\mathbf{D}} - \psi(\widetilde{\mathbf{A}}_{\alpha_{n}}))$ is independent of $n$.  On the other hand, the matrix $\psi(\widetilde{\mathbf{C}}_{n})$ is a diagonal matrix whose elements on the diagonal are either $1$'s or of the form $p^{n - k_{v}}$, when $v \in V^{ram} \smallsetminus V_{\psi}$.  Therefore, there exists $Q(T) \in \overline{\mathbb{Q}}_{p}[T]$ such that
$$h_{X_{n}}(1,\psi) = Q(p^{n}), $$
provided $n \ge n_{1}$.  The result then follows from \cref{useful_r_t}.
\end{proof}

\begin{proposition} \label{unr_i}
With the notation as above, there exists a power series $g(T) \in \mathbb{Z}_{p}\llbracket T\rrbracket$ such that
$$h_{X_{n}}(1,\psi) = g(\psi(\bar{1}) - 1) $$
for all $\psi \in \widehat{\Gamma}_{n}^{(j)}$ whenever $n_{1}<j\le n$.  Therefore, if we let 
$$\mu_{unr} = \mu(g(T)) \text{ and } \lambda_{unr} = \lambda(g(T)),$$ 
then there exists $\nu_{unr} \in \mathbb{Z}$ such that one has
$${\rm ord}_{p} \left(\prod_{j=n_{1}+1}^{n} \prod_{\psi \in \widehat{\Gamma}_{n}^{(j)}} h_{X_{n}}(1,\psi) \right) = \mu_{unr} p^{n} + \lambda_{unr}n + \nu_{unr}$$
when $n$ is large enough.  
\end{proposition}
\begin{proof}
If $\psi \in \widehat{\Gamma}_{n}^{(j)}$ with $n_{1} < j \le n$, then $G_{i} \not \subseteq {\rm ker}(\psi)$ for all $i=t+1,\ldots,g$.  Similarly as in the previous proof, consider the matrices $\widetilde{\mathbf{A}}_{\alpha_{n}}$, and $\widetilde{\mathbf{D}}$ from \cref{link_zak}.  One has
$$h_{X_{n}}(1,\psi) = {\rm det}(\widetilde{\mathbf{D}} - \psi(\widetilde{\mathbf{A}}_{\alpha_{n}})) $$
for all $\psi \in \widehat{\Gamma}_{n}^{(j)}$ provided $n_{1} < j \le n$.  Consider now the group morphism
$$\rho:\mathbb{Z}_{p} \rightarrow \mathbb{Z}_{p}\llbracket T \rrbracket^{\times} $$
given by 
$$a \mapsto \rho(a) = (1+T)^{a} = \sum_{i=0}^{\infty}\binom{a}{i}T^{i} \in \mathbb{Z}_{p}\llbracket T \rrbracket^{\times}, $$
where the $\binom{X}{i}$ are the usual integer-valued polynomials, and define
$$g(T) = {\rm det}(\widetilde{\mathbf{D}} - \widetilde{\mathbf{A}}_{\rho}) \in \mathbb{Z}_{p}\llbracket T \rrbracket, $$
where the matrix $\widetilde{\mathbf{A}}_{\rho} = (\mathbf{a}_{ij}(\rho)) \in M_{t}(\mathbb{Z}_{p}\llbracket T \rrbracket)$ is given by
$$\mathbf{a}_{ij}(\rho) = \sum_{\substack{s \in \mathbf{E}_{X} \\ {\rm inc}(s) = (v_{j},v_{i})}} \rho(\alpha(s))$$
with $1 \le i,j \le t$.  By \cref{link_zak}, the power series $g(T)$ satisfies
$$h_{X_{n}}(1,\psi) = g(\psi(\bar{1}) - 1), $$
for all $\psi \in \widehat{\Gamma}_{n}^{(j)}$ provided $n_{1}< j \le n$ (see for instance the argument in the paragraph preceding \cite[Remark 5.1]{Sage/Vallieres:2022}).  It then follows that
\begin{equation*}
\begin{aligned}
\prod_{j=n_{1}+1}^{n} \prod_{\psi \in \widehat{\Gamma}_{n}^{(j)}} h_{X_{n}}(1,\psi) &= \prod_{j=n_{1}+1}^{n} \prod_{\substack{\zeta \in \mu_{p^{n}}\\{\rm ord}(\zeta) = p^{j}}}g(\zeta - 1), 
\end{aligned}
\end{equation*}
and the result follows from \cref{useful_p_s}.
\end{proof}

As a result, we obtain our main result of this paper.
\begin{theorem} \label{main}
Let $X = (V_{X},\mathbf{E}_{X})$ be a finite connected graph, and let $\alpha:\mathbf{E}_{X} \rightarrow \mathbb{Z}_{p}$ be a voltage assignment.  Consider a family $\mathcal{G}$ of closed subgroups of $\mathbb{Z}_{p}$ indexed by the vertices of $V_{X}$ and consider the branched $\mathbb{Z}_{p}$-tower of graphs
$$X = X_{0} \leftarrow X_{1} \leftarrow X_{2} \leftarrow \ldots \leftarrow X_{n} \leftarrow \ldots, $$
where $X_{n} = X(\Gamma_{n},\mathcal{G}_{n},\alpha_{n})$.  Assume that all finite graphs $X_{n}$ are connected and that $\chi(X_{n}) < 0$ for $n$ large.  Then, there exist $\mu,\lambda \in \mathbb{Z}_{\ge 0}$ and $\nu \in \mathbb{Z}$ such that
\begin{equation} \label{that_one}
{\rm ord}_{p}(\kappa(X_{n})) = \mu p^{n} + \lambda n + \nu 
\end{equation}
when $n$ is large enough.
\end{theorem}
\begin{proof}
This is a consequence of (\ref{starting_pt1}), and \cref{ec_i,tc_i,ram_i,unr_i}.
\end{proof}

\begin{remark}\label{lambda explicit}
With care, the Iwasawa invariants in \cref{main} can be calculated explicitly.  Indeed, with the notation as in \cref{ec_i,tc_i,ram_i,unr_i}, one has $\mu = \mu_{unr}$ and
\begin{equation*}
\lambda = 
\begin{cases}
\lambda_{0} + \sum_{j=1}^{n_{1}} \lambda_{j} + \lambda_{unr} - 1, & \text{ if } V^{ram} = \varnothing \text{ and } \chi(X) \neq 0;\\
\lambda_{0} + \sum_{j=1}^{n_{1}} \lambda_{j} + \lambda_{unr}, & \text{ if } V^{ram} \neq \varnothing.
\end{cases}
\end{equation*}
With even more care in \cref{ec_i,tc_i,ram_i,unr_i}, one can also find an explicit $n_{3} \ge 0$ such that (\ref{that_one}) holds true for $n \ge n_{3}$.  We leave this to the interested reader.
\end{remark}

%\subsection{Revisiting the main conjecture}
\subsection{Revisiting the main conjecture} \label{section:revi}

In this subsection we link the Iwasawa main conjecture proved in \cite[\S 5.2]{Gambheera/Vallieres:2024} to the Ihara $L$-functions attached to each layer of a branched $\mathbb{Z}_{p}$-tower. More precisely, we calculate the characteristic ideal of ${\rm Pic}^{0}_{\Lambda}$, which we are going to recall now, in terms of the power series $g(T)$ that was defined in the previous subsection.  Recall that $\Lambda$ denotes the usual Iwasawa algebra of the profinite group $\mathbb{Z}_{p}$. 

\begin{definition}
Let $X = (V_{X},\mathbf{E}_{X})$ be a finite connected graph, and let $\alpha:\mathbf{E}_{X} \rightarrow \mathbb{Z}_{p}$ be a voltage assignment.  Consider a family $\mathcal{G}$ of closed subgroups of $\mathbb{Z}_{p}$ indexed by the vertices of $V_{X}$ and consider the branched $\mathbb{Z}_{p}$-tower of graphs
$$X = X_{0} \leftarrow X_{1} \leftarrow X_{2} \leftarrow \ldots \leftarrow X_{n} \leftarrow \ldots, $$
where $X_{n} = X(\Gamma_{n},\mathcal{G}_{n},\alpha_{n})$, and where we assume each $X_{n}$ to be connected. Let ${\rm Pic}^0(X_n)$ be the Picard group of degree zero of $X_n$ for each $n$ (see \cite[\S 2.2]{Gambheera/Vallieres:2024} for the precise definition). Then, we define
$$\mathrm{Pic}^{0}_{\Lambda} := \varprojlim_{n \ge 0}\mathrm{Pic}^{0}(X_{n})[p^{\infty}],$$
where the projective limit is taken with respect to the natural projection maps (see \cite[\S 4.3]{Gambheera/Vallieres:2024}).
\end{definition}
It follows from \cite[Theorem 5.6]{Gambheera/Vallieres:2024} that $\mathrm{Pic}^{0}_{\Lambda}$ is a finitely generated torsion $\Lambda$-module, and thus one can talk about its characteristic ideal.  For the following result, we use the usual non-canonical isomorphism $\Lambda \simeq \mathbb{Z}_{p}\llbracket T \rrbracket$.
\begin{theorem}\label{main conjecture}
Let $X = (V_{X},\mathbf{E}_{X})$ be a finite connected graph, and let $\alpha:\mathbf{E}_{X} \rightarrow \mathbb{Z}_{p}$ be a voltage assignment.  Consider a family $\mathcal{G}$ of closed subgroups of $\mathbb{Z}_{p}$ indexed by the vertices of $V_{X}$ and consider the branched $\mathbb{Z}_{p}$-tower of graphs
$$X = X_{0} \leftarrow X_{1} \leftarrow X_{2} \leftarrow \ldots \leftarrow X_{n} \leftarrow \ldots,$$
where $X_{n} = X(\Gamma_{n},\mathcal{G}_{n},\alpha_{n})$.  Assume that all finite graphs $X_{n}$ are connected.  With the notation as above, one has
$$\mathrm{char}_{\mathbb{Z}_{p}\llbracket T \rrbracket}(\mathrm{Pic}^0_{\Lambda}) \cdot (T) = (g(T))\cdot \prod_{v\in V^{ram}} ((1+T)^{p^{k_v}}-1)$$
\end{theorem}
\begin{proof}
By \cite[Theorem 6.1]{Gambheera/Vallieres:2024}, we have 
$$\mathrm{char}_{\mathbb{Z}_{p}\llbracket T \rrbracket}(\mathrm{Pic}_{\Lambda}^{0}) \cdot (T) = (f_{X,\mathcal{I},\alpha}(T)), $$ where $f_{X,\mathcal{I},\alpha}(T)$ is as defined in \cite[(6.2)]{Gambheera/Vallieres:2024}.  (In \cite{Gambheera/Vallieres:2024}, we used $\mathcal{I}$ instead of $\mathcal{G}$ to denote the collection of closed subgroups.)  Now, observe that by the definitions of $f_{X,\mathcal{I},\alpha}(T)$ and $g(T)$, we have
$$f_{X,\mathcal{I},\alpha}(T)= g(T)\cdot \prod_{v\in V^{ram}} ((1+T)^{p^{k_v}}-1).$$
This completes the proof.
\end{proof}

Notice that \cref{lambda explicit} and \cite[Theorem 5.6]{Gambheera/Vallieres:2024} give us two different ways to calculate the $\mu$ and $\lambda$ invariants appearing in \cref{main}.  Comparing the $\mu$ invariants, we have $\mu(\mathrm{Pic}_{\Lambda}^{0})=\mu_{unr}$. On the other hand, by \cref{main conjecture}, we have
\begin{equation}\label{lambda-algebraic}
\begin{aligned}
\lambda({\rm Pic}^0_{\Lambda}) &= -1 +\lambda_{unr}+ \sum_{v\in V^{ram}} p^{k_v}\\
&= -1 + \lambda_{unr}+ \sum_{j=0}^{n_1} p^j \cdot |\{v\in V^{ram}: k_v=j\}|
\end{aligned}
\end{equation}
Now, notice that for a character  $\psi \in \widehat{\Gamma}_{n}^{(j)}$ with $0<j\leq n_{1}$, we have
\begin{equation*} 
\begin{aligned}
V_{j} = V_{\psi} &= \{v\in V^{ram}: \Gamma_{n,v} \not\subseteq {\rm ker}(\psi)\}\\
&= \{v\in V^{ram}: k_v <j\}.
\end{aligned}
\end{equation*}
Moreover, by the proof of \cref{ram_i}, ${\rm ord}_{p}(h_{X_n}(1,\psi))=|V^{ram}\setminus V_{\psi}| \cdot n + C$ for $n$ large so that $\lambda_{j} = |V^{ram}\setminus V_{\psi}|$.  Hence, by \cref{ram_i} we have
\begin{equation}\label{lambda-analytic}
\begin{aligned}
\sum_{j=1}^{n_1} \lambda_j &=\sum_{j=1}^{n_1} \phi(p^j)\cdot |V^{ram}\setminus V_{j}| \\
&=\sum_{j=1}^{n_1} \phi(p^j)\cdot |\{v\in V^{ram} : k_v\geq j\}| \\
&= \sum_{j=1}^{n_1} \sum_{k=1}^j \phi(p^k) \cdot |\{v\in V^{ram} : k_v=j\}| \\
&=  \sum_{j=1}^{n_1} (p^j-1) \cdot |\{v\in V^{ram} : k_v=j\}| \\
&=  \sum_{j=0}^{n_1} (p^j-1) \cdot |\{v\in V^{ram} : k_v=j\}| \\
&= \sum_{j=0}^{n_1} p^j \cdot |\{v\in V^{ram} : k_v=j\}| -|V^{ram}|.
\end{aligned}
\end{equation}
Now, \cref{lambda explicit} together with (\ref{lambda-algebraic}) and (\ref{lambda-analytic}) give us the following corollary.
\begin{corollary}
Let $\lambda_0$ be as defined in \cref{tc_i}. Then, we have
\begin{equation*}
\lambda_0 = 
\begin{cases}
0, &\text{ if } V^{ram}=\varnothing \text{ and } \chi(X)\neq 0; \\
|V^{ram}|-1, &\text{ if } V^{ram} \neq \varnothing.
\end{cases}
\end{equation*}
\end{corollary}

%\section{Examples}
\section{Examples} \label{exa}
In this section, we revisit the three examples of \cite[\S 6]{Gambheera/Vallieres:2024}.
\begin{enumerate}
\item In this example, $p=2$, and one has $n_{1} = k_{1} = 2$.  \cref{riemann_hurwitz} gives $\chi(X_{n}) = 2^{2}(-2^{n-1} +1)$, provided $n \ge 2$, from which it follows that 
$${\rm ord}_{2}(\chi(X_{n})) = 2$$ 
for $n \ge 2$.  The polynomial $Q_{0}(T)$ in the proof of \cref{tc_i} is given by $Q_{0}(T) = T - 2$, that is one has $h_{X_{n}}'(1,\psi_{0}) = Q_{0}(2^{n})$ provided $n \ge 2$.  Thus
$$\lambda_{0} = 0 \text{ and } \nu_{0} = 1. $$ 
Similarly, if one lets $Q_{1}(T) = 2T$, then $h_{X_{n}}(1,\psi) = Q_{1}(2^{n})$ whenever $\psi \in \widehat{\Gamma}_{n}^{(1)}$ and $n \ge 2$.  It follows that \cref{ram_i} gives
$$\lambda_{1} = 1 \text{ and } \nu_{1} = 1. $$
At last, if one lets $Q_{2}(T) = T$, then $h_{X_{n}}(1,\psi) = Q_{2}(2^{n})$ whenever $\psi \in \widehat{\Gamma}_{n}^{(2)}$ and $n \ge 2$.  Since $|\widehat{\Gamma}_{n}^{(2)}| = 2$, \cref{ram_i} gives
$$\lambda_{2} = 2 \text{ and } \nu_{2} = 0. $$
If $2< j \le n$, and $\psi \in \widehat{\Gamma}_{n}^{(j)}$, then $h_{X_{n}}(u,\psi) = 1$, and 
$$\mu_{unr} = \lambda_{unr} = \nu_{unr} = 0. $$ 
Thus (\ref{starting_pt1}) gives
$$1 + 2 + {\rm ord}_{2}(\kappa(X_{n})) = 1 + (n+1) + 2n $$
or in other words
$${\rm ord}_{2}(\kappa(X_{n})) = 3n - 1, $$
provided $n \ge 2$ which is consistent with what was obtained in Example $1$ of \cite[\S 6]{Gambheera/Vallieres:2024}.
\item In this example, $p=3$, and one has $n_{1} = k_{2} = 1$.  \cref{riemann_hurwitz} gives $\chi(X_{n}) = 3(-2 \cdot 3^{n-1} +1)$, provided $n \ge 1$, from which it follows that 
$${\rm ord}_{3}(\chi(X_{n})) = 1$$ 
for $n \ge 1$.  The polynomial $Q_{0}(T)$ in the proof of \cref{tc_i} is given by $Q_{0}(T) = \frac{4}{3}T - 2$, that is one has $h_{X_{n}}'(1,\psi_{0}) = Q_{0}(3^{n})$ provided $n \ge 1$.  Thus
$$\lambda_{0} = 0 \text{ and } \nu_{0} = 0. $$ 
Similarly, if one lets $Q_{1}(T) = 5T$, then $h_{X_{n}}(1,\psi) = Q_{1}(3^{n})$ whenever $\psi \in \widehat{\Gamma}_{n}^{(1)}$ and $n \ge 1$.  Since $|\widehat{\Gamma}_{n}^{(1)}| = 2$, \cref{ram_i} gives
$$\lambda_{1} = 2 \text{ and } \nu_{1} = 0. $$
If $1< j \le n$, then letting 
$$g(T) = 1 - T^{2} + T^{3} - T^{4} + \ldots \in \mathbb{Z}_{3}\llbracket T \rrbracket, $$
one has $h_{X_{n}}(1,\psi) = g(\psi(\bar{1}) - 1)$ for all $\psi \in \widehat{\Gamma}_{n}^{(j)}$, and thus
$$\mu_{unr} = \lambda_{unr} = 0. $$ 
One can also compute $\nu_{unr} = 0$ and the formula from \cref{unr_i} is valid for $n \ge 1$.  Thus (\ref{starting_pt1}) gives
$$1 + {\rm ord}_{3}(\kappa(X_{n})) = 2n $$
or in other words
$${\rm ord}_{3}(\kappa(X_{n})) = 2n - 1, $$
provided $n \ge 1$ which is consistent with what was obtained in Example $2$ of \cite[\S 6]{Gambheera/Vallieres:2024}.
\item In this example, $p=3$, and one has $n_{1} = k_{2} = 1$ again.  \cref{riemann_hurwitz} gives $\chi(X_{n}) = 3(-2 \cdot 3^{n} +1)$, provided $n \ge 1$, from which it follows that 
$${\rm ord}_{3}(\chi(X_{n})) = 1$$ 
for $n \ge 1$.  The polynomial $Q_{0}(T)$ in the proof of \cref{tc_i} is given by $Q_{0}(T) = 12T - 6$, that is one has $h_{X_{n}}'(1,\psi_{0}) = Q_{0}(3^{n})$ provided $n \ge 1$.  Thus
$$\lambda_{0} = 0 \text{ and } \nu_{0} = 1. $$ 
Similarly, if one lets $Q_{1}(T) = 21T$, then $h_{X_{n}}(1,\psi) = Q_{1}(3^{n})$ whenever $\psi \in \widehat{\Gamma}_{n}^{(1)}$ and $n \ge 1$.  Since $|\widehat{\Gamma}_{n}^{(1)}| = 2$, \cref{ram_i} gives
$$\lambda_{1} = 2 \text{ and } \nu_{1} = 2. $$
If $1< j \le n$, then letting 
$$g(T) = 3 - 3T^{2} + 3T^{3} - 3T^{4} + \ldots \in \mathbb{Z}_{3}\llbracket T \rrbracket, $$
one has $h_{X_{n}}(1,\psi) = g(\psi(\bar{1}) - 1)$ for all $\psi \in \widehat{\Gamma}_{n}^{(j)}$, and thus
$$\mu_{unr} = 1 \text{ and } \lambda_{unr} = 0. $$ 
One can also compute $\nu_{unr} = -3$, and the formula from \cref{unr_i} is valid for $n \ge 1$.  Thus (\ref{starting_pt1}) gives
$$1 + {\rm ord}_{3}(\kappa(X_{n})) = 1 + (2n + 2) + (3^{n} - 3) $$
or in other words
$${\rm ord}_{3}(\kappa(X_{n})) = 3^{n} + 2n -1, $$
provided $n \ge 1$ which is consistent with what was obtained in Example $3$ of \cite[\S 6]{Gambheera/Vallieres:2024}.
\end{enumerate}


\begin{thebibliography}{10}

\bibitem{Atiyah/MacDonald:1969}
M.~F. Atiyah and I.~G. Macdonald.
\newblock {\em Introduction to commutative algebra}.
\newblock Addison-Wesley Publishing Co., Reading, Mass.-London-Don Mills, Ont., 1969.

\bibitem{Baker/Norine:2009}
Matthew Baker and Serguei Norine.
\newblock Harmonic morphisms and hyperelliptic graphs.
\newblock {\em Int. Math. Res. Not. IMRN}, 2009(15):2914--2955, 2009.

\bibitem{Bass:1992}
Hyman Bass.
\newblock The {I}hara-{S}elberg zeta function of a tree lattice.
\newblock {\em Internat. J. Math.}, 3(6):717--797, 1992.

\bibitem{Bourbaki:1998}
Nicolas Bourbaki.
\newblock {\em Lie groups and {L}ie algebras. {C}hapters 1--3}.
\newblock Elements of Mathematics (Berlin). Springer-Verlag, Berlin, 1998.
\newblock Translated from the French, Reprint of the 1989 English translation.

\bibitem{Corry:2012}
Scott Corry.
\newblock Harmonic {G}alois theory for finite graphs.
\newblock In {\em Galois-{T}eichm\"{u}ller theory and arithmetic geometry}, volume~63 of {\em Adv. Stud. Pure Math.}, pages 121--140. Math. Soc. Japan, Tokyo, 2012.

\bibitem{Sage/Vallieres:2022}
Sage DuBose and Daniel Valli\`eres.
\newblock On {$\mathbb Z^d_\ell$}-towers of graphs.
\newblock {\em Algebr. Comb.}, 6(5):1331--1346, 2023.

\bibitem{Gambheera/Vallieres:2024}
Rusiru Gambheera and Daniel Valli\`eres.
\newblock Iwasawa theory for branched {$\mathbb{Z}_p$}-towers of finite graphs.
\newblock {\em Doc. Math.}, 29(6):1435--1468, 2024.

\bibitem{Goldman:1961}
Oscar Goldman.
\newblock Determinants in projective modules.
\newblock {\em Nagoya Math. J.}, 18:27--36, 1961.

\bibitem{Gonet:2021a}
Sophia~R. Gonet.
\newblock {\em Jacobians of finite and infinite voltage covers of graphs}.
\newblock ProQuest LLC, Ann Arbor, MI, 2021.
\newblock Thesis (Ph.D.)--The University of Vermont and State Agricultural College.

\bibitem{Gonet:2022}
Sophia~R. Gonet.
\newblock Iwasawa theory of {J}acobians of graphs.
\newblock {\em Algebr. Comb.}, 5(5):827--848, 2022.

\bibitem{Hashimoto:1990}
Ki-ichiro Hashimoto.
\newblock On zeta and {$L$}-functions of finite graphs.
\newblock {\em Internat. J. Math.}, 1(4):381--396, 1990.

\bibitem{Hashimoto:1992}
Ki-ichiro Hashimoto.
\newblock Artin type {$L$}-functions and the density theorem for prime cycles on finite graphs.
\newblock {\em Internat. J. Math.}, 3(6):809--826, 1992.

\bibitem{Iwasawa:1959}
Kenkichi Iwasawa.
\newblock On {$\Gamma $}-extensions of algebraic number fields.
\newblock {\em Bull. Amer. Math. Soc.}, 65:183--226, 1959.

\bibitem{Iwasawa:1972}
Kenkichi Iwasawa.
\newblock {\em Lectures on {$p$}-adic {$L$}-functions}.
\newblock Annals of Mathematics Studies, No. 74. Princeton University Press, Princeton, NJ; University of Tokyo Press, Tokyo, 1972.

\bibitem{Iwasawa:1973}
Kenkichi Iwasawa.
\newblock On {${\bf Z}_{l}$}-extensions of algebraic number fields.
\newblock {\em Ann. of Math. (2)}, 98:246--326, 1973.

\bibitem{Kataoka:2024}
Takenori Kataoka.
\newblock Fitting ideals of {J}acobian groups of graphs.
\newblock {\em Algebr. Comb.}, 7(3):597--625, 2024.

\bibitem{Kleine/Muller:2023}
S\"{o}ren Kleine and Katharina M\"{u}ller.
\newblock On the growth of the {J}acobian in {$\mathbb{Z}_{p}^{l}$}-voltage covers of graphs.
\newblock {\em Algebr. Comb.}, 7(4):1011--1038, 2024.

\bibitem{Kotani/Sunada:2000}
Motoko Kotani and Toshikazu Sunada.
\newblock Zeta functions of finite graphs.
\newblock {\em J. Math. Sci. Univ. Tokyo}, 7(1):7--25, 2000.

\bibitem{Len/Ulirsch/Zakharov:2024}
Yoav Len, Martin Ulirsch, and Dmitry Zakharov.
\newblock Abelian tropical covers.
\newblock {\em Math. Proc. Cambridge Philos. Soc.}, 176(2):395--416, 2024.

\bibitem{Malmskog/Manes:2010}
Beth Malmskog and Michelle Manes.
\newblock ``{A}lmost divisibility'' in the {I}hara zeta functions of certain ramified covers of {$q+1$}-regular graphs.
\newblock {\em Linear Algebra Appl.}, 432(10):2486--2506, 2010.

\bibitem{mcgownvallieresII}
Kevin McGown and Daniel Valli\`eres.
\newblock On abelian {$\ell$}-towers of multigraphs {II}.
\newblock {\em Ann. Math. Qu\'{e}.}, 47(2):461--473, 2023.

\bibitem{mcgownvallieresIII}
Kevin McGown and Daniel Valli\`eres.
\newblock On abelian {$\ell$}-towers of multigraphs {III}.
\newblock {\em Ann. Math. Qu\'{e}.}, 48(1):1--19, 2024.

\bibitem{Ribes:2017}
Luis Ribes.
\newblock {\em Profinite graphs and groups}, volume~66 of {\em Ergebnisse der Mathematik und ihrer Grenzgebiete. 3. Folge. A Series of Modern Surveys in Mathematics}.
\newblock Springer, Cham, 2017.

\bibitem{Sambale:2023}
Benjamin Sambale.
\newblock An invitation to formal power series.
\newblock {\em Jahresber. Dtsch. Math.-Ver.}, 125(1):3--69, 2023.

\bibitem{Serre:1977}
Jean-Pierre Serre.
\newblock {\em Arbres, amalgames, {${\rm SL}_{2}$}}.
\newblock Ast\'{e}risque, No. 46. Soci\'{e}t\'{e} Math\'{e}matique de France, Paris, 1977.
\newblock Avec un sommaire anglais, R\'{e}dig\'{e} avec la collaboration de Hyman Bass.

\bibitem{Silvester:2000}
John~R. Silvester.
\newblock Determinants of block matrices.
\newblock {\em Math. Gaz.}, 84(501):460--467, 2000.

\bibitem{Stark/Terras:1996}
Harold~M. Stark and Audrey~A. Terras.
\newblock Zeta functions of finite graphs and coverings.
\newblock {\em Adv. Math.}, 121(1):124--165, 1996.

\bibitem{Stark/Terras:2000}
Harold~M. Stark and Audrey~A. Terras.
\newblock Zeta functions of finite graphs and coverings. {II}.
\newblock {\em Adv. Math.}, 154(1):132--195, 2000.

\bibitem{Stark/Terras:2007}
Harold~M. Stark and Audrey~A. Terras.
\newblock Zeta functions of finite graphs and coverings. {III}.
\newblock {\em Adv. Math.}, 208(1):467--489, 2007.

\bibitem{Sunada:2013}
Toshikazu Sunada.
\newblock {\em Topological crystallography}, volume~6 of {\em Surveys and Tutorials in the Applied Mathematical Sciences}.
\newblock Springer, Tokyo, 2013.
\newblock With a view towards discrete geometric analysis.

\bibitem{Tate:1984}
John Tate.
\newblock {\em Les conjectures de {S}tark sur les fonctions {$L$} d'{A}rtin en {$s=0$}}, volume~47 of {\em Progress in Mathematics}.
\newblock Birkh\"{a}user Boston, Inc., Boston, MA, 1984.
\newblock Lecture notes edited by Dominique Bernardi and Norbert Schappacher.

\bibitem{Terras:2011}
Audrey Terras.
\newblock {\em Zeta functions of graphs}, volume 128 of {\em Cambridge Studies in Advanced Mathematics}.
\newblock Cambridge University Press, Cambridge, 2011.
\newblock A stroll through the garden.

\bibitem{Vallieres:2021}
Daniel Valli\`eres.
\newblock On abelian {$\ell$}-towers of multigraphs.
\newblock {\em Ann. Math. Qu\'{e}.}, 45(2):433--452, 2021.

\bibitem{Zakharov:2021}
Dmitry Zakharov.
\newblock Zeta functions of edge-free quotients of graphs.
\newblock {\em Linear Algebra Appl.}, 629:40--71, 2021.

\end{thebibliography}
\end{document}